\documentclass{article}

\usepackage{amsmath}
\usepackage{amsthm}
\usepackage{amssymb}
\usepackage{amscd}
\usepackage{xspace}
\usepackage{verbatim}
\newtheorem{theorem}{Theorem}[section]
\newtheorem{corollary}{Corollary}[section]
\newtheorem{lemma}{Lemma}[section]
\newtheorem{proposition}{Proposition}[section]
\theoremstyle{definition}
\newtheorem{definition}{Definition}[section]

\theoremstyle{remark}
\newtheorem{remark}{Remark}[section]
\numberwithin{equation}{section}

\newcommand{\ov}{\overline}

\newcommand{\e}{\varepsilon}

\newcommand{\G}{\Gamma}

\renewcommand{\O}{\Omega}

\renewcommand{\vec}[1]{\mathbf{#1}}
\newcommand{\field}[1]{\mathbb{#1}}

\newcommand{\R}{\field{R}}

\newcommand{\er}{\eqref}

\DeclareMathOperator{\Div}{div} 
\DeclareMathOperator{\supp}{supp}

\renewcommand{\O}{\Omega}

\newcommand{\f}{\varphi}
\renewcommand{\vec}[1]{\boldsymbol{#1}}

\begin{document}
\title{On the $\Gamma$-limit of singular perturbation problems with optimal profiles which are not one-dimensional.
Part I: The upper bound} \maketitle
\begin{center}
\textsc{Arkady Poliakovsky \footnote{E-mail:
poliakov@math.bgu.ac.il}
}\\[3mm]
Department of Mathematics, Ben Gurion University of the Negev,\\
P.O.B. 653, Be'er Sheva 84105, Israel
\\[2mm]
\date{}
\end{center}
%
%
%
%
%
%
\begin{abstract}
In Part I we construct an upper bound, in the spirit of $\Gamma$-
$\limsup$, achieved by multidimensional profiles, for some general
classes of singular perturbation problems, with or without the
prescribed differential constraint, taking the form
\begin{equation*} E_\e(v):=\int_\Omega
\frac{1}{\e}F\Big(\e^n\nabla^n v,...,\e\nabla
v,v\Big)dx\quad\text{for} v:\Omega\subset\R^N\to\R^k \text{ such
that } A\cdot\nabla v=0,
\end{equation*}
where the function $F\geq 0$ and $A:\R^{k\times N}\to\R^m$ is a
prescribed linear operator (for example, $A:\equiv 0$, $A\cdot\nabla
v:=\text{curl}\; v$ and $A\cdot\nabla v=\Div v$) which includes, in
particular, the problems considered in \cite{polgen}. This bound is
in general sharper then one obtained in \cite{polgen}.
\end{abstract}

\section{Introduction}
\begin{definition}
Consider a family $\{I_\varepsilon\}_{\varepsilon>0}$ of functionals
$I_\varepsilon(\phi):U\to[0,+\infty]$, where $U$ is a given metric
space. The $\Gamma$-limits of $I_\varepsilon$ are defined by:
\begin{align*}
(\Gamma-\liminf_{\varepsilon\to 0^+} I_\varepsilon)(\phi)
:=\inf\left\{\liminf_{\varepsilon\to
0^+}I_\varepsilon(\phi_\varepsilon):\;\,\{\phi_\varepsilon\}_{\varepsilon>0}\subset
U,\; \phi_\varepsilon\to\phi\text{ in }U\;
\text{as}\;\varepsilon\to 0^+\right\},\\
(\Gamma-\limsup_{\varepsilon\to 0^+} I_\varepsilon)(\phi)
:=\inf\left\{\limsup_{\varepsilon\to
0^+}I_\varepsilon(\phi_\varepsilon):\;\,\{\phi_\varepsilon\}_{\varepsilon>0}\subset
U,\; \phi_\varepsilon\to\phi\text{ in }U\;
\text{as}\;\varepsilon\to 0^+\right\},\\
(\Gamma-\lim_{\varepsilon\to 0^+}
I_\varepsilon)(\phi):=(\Gamma-\liminf_{\varepsilon\to 0^+}
I_\varepsilon\big)(\phi)=(\Gamma-\limsup_{\varepsilon\to 0^+}
I_\varepsilon)(\phi)\;\;\,\text{in the case they are equal}.
\end{align*}
\end{definition}
It is useful to know the $\Gamma$-limit of $I_\varepsilon$, because
it describes the asymptotic behavior as $\varepsilon\downarrow 0$ of
minimizers of $I_\varepsilon$, as it is clear from the following
simple statement:
\begin{proposition}[De-Giorgi]\label{propdj}
Assume that $\phi_\varepsilon$ is a minimizer of $I_\varepsilon$ for
every $\varepsilon>0$. Then:
\begin{itemize}
\item
If $I_0(\phi)=(\Gamma-\liminf_{\varepsilon\to 0^+}
I_\varepsilon)(\phi)$ and $\phi_\varepsilon\to\phi_0$ as
$\varepsilon\to 0^+$ then $\phi_0$ is a minimizer of $I_0$.

\item
If $I_0(\phi)=(\Gamma-\lim_{\varepsilon\to 0^+}
I_\varepsilon)(\phi)$ (i.e. it is a full $\Gamma$-limit of
$I_\varepsilon(\phi)$) and for some subsequence $\varepsilon_n\to
0^+$ as $n\to\infty$, we have $\phi_{\varepsilon_n}\to\phi_0$, then
$\phi_0$ is a minimizer of $I_0$.
\end{itemize}
\end{proposition}
Usually, for finding the $\Gamma$-limit of $I_\varepsilon(\phi)$, we
need to find two bounds.
\begin{itemize}
\item[{\bf(*)}] Firstly, we find a lower bound, i.e. a functional
$\underline{I}(\phi)$ such that for every family
$\{\phi_\varepsilon\}_{\varepsilon>0}$, satisfying
$\phi_\varepsilon\to \phi$ as $\varepsilon\to 0^+$, we have
$\liminf_{\varepsilon\to 0^+}I_\varepsilon(\phi_\varepsilon)\geq
\underline{I}(\phi)$.
\item[{\bf(**)}] Secondly, we find an upper
bound, i.e. a functional $\overline{I}(\phi)$, such that for every
$\phi\in U$ there exists a family
$\{\psi_\varepsilon\}_{\varepsilon>0}$, satisfying
$\psi_\varepsilon\to \phi$ as $\varepsilon\to 0^+$ and
$\limsup_{\varepsilon\to 0^+}I_\varepsilon(\psi_\varepsilon)\leq
\overline{I}(\phi)$.
\item[{\bf(***)}] If we find that
$\underline{I}(\phi)=\overline{I}(\phi):=I(\phi)$, then $I(\phi)$ is
the $\Gamma$-limit of $I_\varepsilon(\phi)$.
\end{itemize}

 In various applications we deal with the asymptotic behavior as $\e\to 0^+$ of a family of
 functionals $\{I_\e\}_{\e>0}$
of the following forms.
\begin{itemize}
\item
In the case of the first order problem the functional $I_\e$, which
acts on functions $\psi:\O\to\R^m$, has the form
\begin{equation}\label{b1..}
I_\e(\psi)=\int_\O
\e\big|\nabla\psi(x)\big|^2+\frac{1}{\e}W\Big(\psi(x),x\Big)dx\,,
\end{equation}
or more generally
\begin{equation}\label{b2..}
I_\e(\psi)=\int_\O\frac{1}{\e}G\Big(\e^n\nabla^n\psi,\ldots,\e\nabla\psi,\psi,x\Big)dx
+\int_\O\frac{1}{\e}W\big(\psi,x\big)dx\,,
\end{equation}
where $G(0,\ldots,0,\psi,x)\equiv 0$.
\item In the case of the second order problem the functional $I_\e$,
which acts on functions $v:\O\to\R^k$, has the form
\begin{equation}\label{b3..}
I_\e(v)=\int_\O \e\big|\nabla^2 v(x)\big|^2+\frac{1}{\e}W\Big(\nabla
v(x),v(x),x\Big)dx\,,
\end{equation}
or more generally
\begin{equation}\label{b4..}
I_\e(v)=\int_\O\frac{1}{\e}G\Big(\e^n\nabla^{n+1}
v,\ldots,\e\nabla^2 v,\nabla
v,v,x\Big)dx+\int_\O\frac{1}{\e}W\big(\nabla v,v,x\big)dx\,,
\end{equation}
where $G(0,\ldots,0,\nabla v,v,x)\equiv 0$.
\end{itemize}

In this paper we deal with the asymptotic behavior as $\e\to 0^+$ of
a family of functionals of the following general form: Let
$\Omega\subset{\mathbb{R}}^N$ be an open set.
For every $\varepsilon>0$ consider the general functional
\begin{multline}\label{fhjvjhvjhv}
I_{\varepsilon}(v)=\big\{I_{\varepsilon}(\Omega)\big\}(v):=
\int_{\Omega}\frac{1}{\varepsilon}G\Big(\varepsilon^n\nabla^n
v,\ldots,\varepsilon\nabla
v,v,x\Big)+\frac{1}{\varepsilon}W\big(v,x\big)dx\quad \text{with }
v:=(\nabla u,h,\psi),\\
\text{ where } u\in W^{(n+1),1}_{loc}(\Omega,{\mathbb{R}}^k),\;\,
h\in W^{n,1}_{loc}(\Omega,{\mathbb{R}}^{d\times N})\text{ s.t.
}\text{div } h\equiv 0,\;\,\psi\in
W^{n,1}_{loc}(\Omega,{\mathbb{R}}^m).
\end{multline}
Here
$$G:\R^{\big(\{k\times N\}+\{d\times N\}+m\big)\times N^n}\times\ldots\times\R^{\big(\{k\times N\}+\{d\times N\}+m\big)\times N}
\times\R^{\{k\times N\}+\{d\times N\}+m}\times\R^N\,\to\,\R$$ and
$W:\R^{\{k\times N\}+\{d\times N\}+m}\times\R^N\,\to\,\R$ are
nonnegative continuous functions and $G$ satisfies
$G(0,\ldots,0,v,x)\equiv 0$. The functionals in \er{b1..},\er{b2..}
and \er{b3..},\er{b4..} are important particular cases of the
general energy $I_\e$ in \er{fhjvjhvjhv}. In the general form
\er{fhjvjhvjhv} we also include the dependence on $\Div$-free
function $h$, which can be useful in the study of problems with
non-local terms as the Rivi\`ere-Serfaty functional and other
functionals in Micromagnetics.

 The functionals of the form \er{b1..} arise in the theories of phase
transitions and minimal surfaces. They were first studied by Modica
and Mortola \cite{mm1}, Modica \cite{modica}, Sternberg
\cite{sternberg} and others. The $\Gamma$-limit of the functional in
\er{b1..}, where $W$ does not depend on $x$ explicitly, was obtained
in the general vectorial case by Ambrosio in \cite{ambrosio}. The
$\Gamma$-limit of the functional of the form \er{b2..}, where $n=1$
and there exist $\alpha,\beta\in\R^m$ such that $W(h,x)=0$ if and
only if $h\in\{\alpha,\beta\}$, under some restriction on the
explicit dependence on $x$ of $G$ and $W$, was obtained by Fonseca
and Popovici in \cite{FonP}. The $\Gamma$-limit of the functional of
the form \er{b2..}, with $n=2$,
$G(\cdot)/\e\equiv\e^3|\nabla^2\psi|^2$ and $W$ which doesn't depend
on $x$ explicitly, was found by I.~Fonseca and C.~Mantegazza in
\cite{FM}.

%
%
%
%
%
%

 The functionals of second order of the form \er{b3..} arise, for
example, in the gradient theory of solid-solid phase transitions,
where one considers energies of the form
\begin{equation}\label{b3..part}
I_\e(v)=\int_\O \e|\nabla^2 v(x)|^2+\frac{1}{\e}W\Big(\nabla
v(x)\Big)dx\,,
\end{equation}
where $v:\O\subset\R^N\to\R^N$ stands for the deformation, and the
free energy density $W(F)$ is nonnegative and satisfies
$$W(F)=0\quad\text{if and only if}\quad F\in K:=SO(N)A\cup SO(N)B\,.$$
Here $A$ and $B$ are two fixed, invertible matrices, such that
$rank(A-B)=1$ and $SO(N)$ is the set of rotations in $\R^N$. The
simpler case where $W(F)=0$ if and only if $F\in\{A,B\}$ was studied
by Conti, Fonseca and Leoni in \cite{contiFL}. The case of problem
\er{b3..part}, where $N=2$ and $W(QF)=W(F)$ for all $Q\in SO(2)$ was
investigated by Conti and Schweizer in \cite{contiS1} (see also
\cite{contiS} for a related problem). Another important example of
the second order energy is the so called Aviles-Giga functional,
defined on scalar valued functions $v$ by
\begin{equation}\label{b5..}
\int_\O\e|\nabla^2 v|^2+\frac{1}{\e}\big(1-|\nabla
v|^2\big)^2\quad\quad\text{(see \cite{adm},\cite{ag1},\cite{ag2})}.
\end{equation}

 The main contribution of this work is to improve our method (see
\cite{pol},\cite{polgen}) for finding upper bounds in the sense of
({\bf**}) for the general functional \er{fhjvjhvjhv} in the case
where the limiting function belongs to $BV$-space.
In order to formulate the main results of this paper we present the
following definitions.

 First of all, in order to simplify the notations in \er{fhjvjhvjhv},
for every open $\mathcal{U}\subset\R^N$ consider
\begin{multline}\label{fhjvjhvjhvholhiohiovhhjhvvjvf}
\mathcal{B}(\mathcal{U}):=\bigg\{v\in
L^1_{loc}\big(\mathcal{U},{\mathbb{R}}^{k\times
N}\times{\mathbb{R}}^{d\times
N}\times{\mathbb{R}}^m\big):\;\;v=(\nabla u,h,\psi),\;\,u\in
W^{1,1}_{loc}(\mathcal{U},{\mathbb{R}}^k),\;\,\\ h\in
L^{1}_{loc}(\mathcal{U},{\mathbb{R}}^{d\times N})\text{ s.t.
}\text{div } h\equiv 0\text{ in the sense of
distributions},\;\,\psi\in
L^{1}_{loc}(\mathcal{U},{\mathbb{R}}^m)\bigg\},
\end{multline}
and
\begin{equation}\label{huighuihuiohhhiuohoh}
F\Big(\nabla^n v,\ldots,\nabla v,v,x\Big)\,:=\,G\Big(\nabla^n
v,\ldots,\nabla v,v,x\Big)\,+\,W(v,x)
\end{equation}
Then
\begin{equation}\label{fhjvjhvjhvnlhiohoioiiy}
I_{\varepsilon}(v)
=\int_{\Omega}\frac{1}{\varepsilon}F\Big(\varepsilon^n\nabla^n
v,\ldots,\varepsilon\nabla v,v,x\Big)dx\quad \text{with }
v\in \mathcal{B}(\O)\cap W^{n,1}_{loc}\big(\O,\R^{k\times
N}\times{\mathbb{R}}^{d\times N}\times{\mathbb{R}}^m\big).
\end{equation}
%
%
%
%
%
%
%
%

What can we expect as the $\Gamma$-limit or at least as an upper
bound of these general energies in the $L^p$-topology?
It is clear that if $G$ and $W$ are nonnegative and $W$
is a continuous on the argument $v$ function, then the upper bound
for $I_\e(\cdot)$ will be finite only if
\begin{equation}\label{hghiohoijojjkhhhjhjjkjgg}
W\big(v(x),x\big)=0\quad\text{for a.e.}\;\;x\in\Omega\,,
\end{equation}
i.e. if we define
\begin{equation}\label{cuyfyugugghvjjhh}
\mathcal{A}_0:=\bigg\{v\in L^p\big(\Omega,{\mathbb{R}}^{k\times
N}\times{\mathbb{R}}^{d\times
N}\times{\mathbb{R}}^m\big)\cap\mathcal{B}(\Omega):\;\,
W\big(v(x),x\big)=0\;\,\text{for a.e.}\,\;x\in\Omega\bigg\}
\end{equation}
and
\begin{equation}\label{cuyfyugugghvjjhhggihug}
\mathcal{A}:=\bigg\{v\in L^p\big(\Omega,{\mathbb{R}}^{k\times
N}\times{\mathbb{R}}^{d\times N}\times{\mathbb{R}}^m\big):\;\,
(\Gamma-\limsup_{\varepsilon\to 0^+}
I_\varepsilon)(v)<+\infty\bigg\},
\end{equation}
then clearly $\mathcal{A}\subset\mathcal{A}_0$. In most interesting
applications the set $\mathcal{A}_0$ consists of discontinuous
functions. The natural space of discontinuous functions is $BV$
space. It turns out that in the general case if $G$ and $W$ are
$C^1$-functions and if we consider
\begin{equation}\label{hkghkgh}
\mathcal{A}_{BV}:=\mathcal{A}_0\cap\mathcal{B}(\mathbb{R}^N)\cap
BV\cap L^\infty,
\end{equation}
then
\begin{equation}\label{nnloilhyoih}
\mathcal{A}_{BV}\subset\mathcal{A}\subset\mathcal{A}_0.
\end{equation}
In many cases we have $\mathcal{A}_{BV}=\mathcal{A}$. For example
this is indeed the case if the energy $I_\varepsilon(v)$ has the
simplest form $I_\varepsilon(v)=\int_\Omega\varepsilon|\nabla
v|^2+\frac{1}{\varepsilon}W(v)\,dx$, and the set of zeros of $W$:
$\{h: W(h)=0\}$ is finite. However, this is in general not the case.
For example, as was shown by Ambrosio, De Lellis and Mantegazza in
\cite{adm}, $\mathcal{A}_{BV}\subsetneq\mathcal{A}$ in the
particular case of the energy defined by \eqref{b5..} with $N=2$. On
the other hand, there are many applications where the set
$\mathcal{A}$ still inherits some good properties of $BV$ space. For
example, it is indeed the case for the energy \eqref{b5..} with
$N=2$, as was shown by Camillo de Lellis and Felix Otto in
\cite{CDFO}.

\begin{definition}
For every $\vec\nu\in S^{N-1}$ define
$Q(\vec\nu):=\big\{y\in{\mathbb{R}}^N:\;
-1/2<y\cdot\vec\nu_j<1/2\quad\forall j\big\}$, where
$\{\vec\nu_1,\ldots,\vec\nu_N\}$ is an orthonormal base in
${\mathbb{R}}^N$ such that $\vec\nu_1=\vec\nu$. Then set
\begin{multline*}
\mathcal{D}_1(v^+,v^-,\vec\nu):=\bigg\{v\in
C^n\big(\mathbb{R}^N,{\mathbb{R}}^{k\times
N}\times{\mathbb{R}}^{d\times N}\times{\mathbb{R}}^m\big)\cap
\mathcal{B}(\mathbb{R}^N):\\ v(y)\equiv \theta(\vec\nu\cdot
y)\;\,\text{and}\;\, v(y)=v^-\;\text{ if }\;y\cdot\vec\nu\leq
-1/2,\;\; v(y)=v^+\;\text{ if }\; y\cdot\vec\nu\geq 1/2\bigg\},
\end{multline*}
where $\mathcal{B}(\cdot)$ is defined in
\eqref{fhjvjhvjhvholhiohiovhhjhvvjvf}, and
\begin{multline*}
\mathcal{D}_{per}(v^+,v^-,\vec\nu):=\bigg\{v\in
C^n\big(\mathbb{R}^N,{\mathbb{R}}^{k\times
N}\times{\mathbb{R}}^{d\times N}\times{\mathbb{R}}^m\big)\cap \mathcal{B}(\mathbb{R}^N):\\
v(y)=v^-\;\text{ if }\;y\cdot\vec\nu\leq -1/2,\;\; v(y)=v^+\;\text{
if }\; y\cdot\vec\nu\geq 1/2,\;\, v(y+\vec\nu_j)=v(y)\;\;\forall
j=2,\ldots, N\bigg\}.
\end{multline*}
Next define
\begin{align}
\label{Energia1} E_{1}(v^+,v^-,\vec\nu,x):=\inf\bigg\{
\int\limits_{Q(\vec\nu_v)}\frac{1}{L}F\Big(L^n\nabla^n\zeta,\ldots,L\nabla
\zeta,\zeta,x\Big)dy:\;\, L>0,\, \zeta(y)\in
\mathcal{D}_1(v^+,v^-,\vec\nu)\bigg\}\,,\\
\label{Energia2} E_{per}(v^+,v^-,\vec\nu,x):=\inf\bigg\{
\int\limits_{Q(\vec\nu_v)}\frac{1}{L}F\Big(L^n\nabla^n\zeta,\ldots,L\nabla
\zeta,\zeta,x\Big)dy:\;\, L>0,\, \zeta(y)\in
\mathcal{D}_{per}(v^+,v^-,\vec\nu)\bigg\}\,.\\
\label{Energia3} E_{abst}(v^+,v^-,\vec\nu,x):=\Big(\Gamma-\liminf_{
\varepsilon \to 0^+}
I_\varepsilon\big(Q(\vec\nu)\big)\Big)\Big(\eta(v^+,v^-,\vec\nu)\Big),
\end{align}
where
\begin{equation}\eta(v^+,v^-,\vec\nu)(y):=
\begin{cases}
v^-\quad\text{if }\vec\nu\cdot y<0,\\
v^+\quad\text{if }\vec\nu\cdot y>0,
\end{cases}
\end{equation}
and we mean the $\Gamma-\liminf$ in $L^p$ topology for some $p\geq
1$.
\end{definition}
It is not difficult to deduce that
\begin{equation}\label{dghfihtihotj}
E_{abst}(v^+,v^-,\vec\nu,x)\leq E_{per}(v^+,v^-,\vec\nu,x)\leq
E_{1}(v^+,v^-,\vec\nu,x).
\end{equation}
Next define the functionals
$K_1(\cdot),K_{per}(\cdot),K^{*}(\cdot):\mathcal{B}(\O)\cap BV\cap
L^\infty\,\to\,\R$ by
\begin{equation}\label{hfighfighfih}
K_{1}(v):=
\begin{cases}
\int_{\O\cap
J_v}E_{1}\Big(v^+(x),v^-(x),\vec\nu_v(x),x\Big)\,d\mathcal{H}^{N-1}(x)\quad\text{if
}v\in \mathcal{A}_0,\\+\infty\quad\text{otherwise},
\end{cases}
\end{equation}
\begin{equation}\label{hfighfighfihgigiugi}
K_{per}(v):=
\begin{cases}
\int_{\O\cap
J_v}E_{per}\Big(v^+(x),v^-(x),\vec\nu_v(x),x\Big)\,d\mathcal{H}^{N-1}(x)\quad\text{if
}v\in \mathcal{A}_0,\\+\infty\quad\text{otherwise},
\end{cases}
\end{equation}
\begin{equation}\label{hfighfighfihhioh}
K^{*}(v):=
\begin{cases}
\int_{\Omega\cap
J_v}E_{abst}\Big(v^+(x),v^-(x),\vec\nu_v(x),x\Big)\,d\mathcal{H}^{N-1}(x)\quad\text{if
}v\in \mathcal{A}_0,\\+\infty\quad\text{otherwise},
\end{cases}
\end{equation}
where $J_v$ is the jump set of $v$, $\vec\nu_v$ is the jump vector
and $v^-,v^+$ are jumps of $v$.
Then, by \er{dghfihtihotj} trivially follows
\begin{equation}\label{nvhfighfrhyrtehu}
K^*\big(v\big)\leq K_{per}\big(v\big)\leq K_1\big(v\big)\,.
\end{equation}
We call $K_1(\cdot)$, $K_{per}(\cdot)$ and $K^*(\cdot)$ by the
bound, achieved by one dimensional profiles, multidimensional
periodic profiles and abstract profiles respectively.

 Our general conjecture is that $K^*(\cdot)$ coincides with the $\Gamma$-limit for the
functionals $I_\e(\cdot)$ in \er{fhjvjhvjhvnlhiohoioiiy}, under
$L^{p}$ convergence, in the case where the limiting functions $v\in
BV\cap L^\infty$.
It is known that in the case of the problem \er{b1..}, where $W\in
C^1$ doesn't depend on $x$ explicitly, this is indeed the case and
moreover, in this case we have equalities in \er{nvhfighfrhyrtehu}
(see \cite{ambrosio}). The equalities in \er{nvhfighfrhyrtehu} also
hold for the functional of the form \er{b2..}, with $n=2$,
$G(\cdot)/\e\equiv\e^3|\nabla^2\psi|^2$ and $W$ which doesn't depend
on $x$ explicitly. Moreover, as before, in this case the functional
in \er{nvhfighfrhyrtehu} is the $\Gamma$-limit (see \cite{FM}). The
same result is also known for problem \er{b5..} when $N=2$ (see
\cite{adm} and \cite{CdL},\cite{pol}). It is also the case for
problem \er{b3..part} where $W(F)=0$ if and only if $F\in\{A,B\}$,
studied by Conti, Fonseca and Leoni, if $W$ satisfies the additional
hypothesis ($H_3$) in \cite{contiFL}. However, as was shown there by
an example, if we don't assume ($H_3$)-hypothesis, then it is
possible that $E_{per}\big(\nabla v^+,\nabla v^-,\vec\nu\big)$ is
strictly smaller than $E_{1}\big(\nabla v^+,\nabla v^-,\vec\nu\big)$
and thus, in general, $K_1(\cdot)$ can differ from the
$\Gamma$-limit. In the same work it was shown that if, instead of
($H_3$) we assume hypothesis ($H_5$), then $K_{per}(\cdot)$ turns to
be equal to $K^*(\cdot)$ and the $\Gamma$-limit of \er{b3..part}
equals to $K_{per}(\cdot)\equiv K^*(\cdot)$. The similar result
known also for problem \er{b2..}, where $n=1$ and there exist
$\alpha,\beta\in\R^m$ such that $W(h,x)=0$ if and only if
$h\in\{\alpha,\beta\}$, under some restriction on the explicit
dependence on $x$ of $G$ and $W$. As was obtained by Fonseca and
Popovici in \cite{FonP} in this case we also obtain that
$K_{per}(\cdot)\equiv K^*(\cdot)$ is the $\Gamma$-limit of
\er{b2..}. In the case of problem \er{b3..part}, where $N=2$ and
$W(QF)=W(F)$ for all $Q\in SO(2)$, Conti and Schweizer in
\cite{contiS1} found that the $\Gamma$-limit equals to $K^*(\cdot)$
(see also \cite{contiS} for a related problem). However, by our
knowledge, it is not known, whether in general $K^*(\cdot)\equiv
K_{per}(\cdot)$.

 On \cite{polgen} we showed that for the general problems \er{b2..}
and \er{b4..}, $K_1(\cdot)$ is the upper bound in the sense of
{\bf(**)}, if the limiting function belongs to $BV$-class. However,
as we saw, this bound is not sharp in general. The main result of
this paper is that for the general problem
\er{fhjvjhvjhvnlhiohoioiiy}, $K_{per}(\cdot)$ is always an upper
bound in the sense of {\bf(**)} in the case where the limiting
functions $v$ belong to $BV$-space and $G,W\in C^1$. More precisely,
we have the following Theorem:
\begin{theorem}\label{ffgvfgfhthjghgjhg}
Let $\O\subset\R^N$ be an open set and
$$F:\R^{\big(\{k\times N\}+\{d\times N\}+m\big)\times N^n}\times\ldots\times\R^{\big(\{k\times N\}+\{d\times N\}+m\big)\times N}
\times\R^{\{k\times N\}+\{d\times N\}+m}\times\R^N\,\to\,\R$$ be a
nonnegative $C^1$ function. Furthermore assume that $v:=(\nabla
u,h,\psi)\in\mathcal{B}(\R^N)\cap BV\big(\R^N,\R^{k\times
N}\times\R^{d\times N}\times\R^m\big)\cap
L^\infty\big(\R^N,\R^{k\times N}\times\R^{d\times N}\times\R^m\big)$
satisfies $\Div h\equiv 0$, $|Dv|(\partial\Omega)=0$ and
$$F\Big(0,\ldots,0,v(x),x\Big)=0\quad\text{for a.e.}\;\;x\in\O.$$
Then there exists a sequence $v_\e=\big(\nabla
u_\e,h_\e,\psi_\e\big)\in \mathcal{B}(\R^N)\cap
C^\infty\big(\R^N,\R^{k\times N}\times\R^{d\times N}\times\R^m\big)$
such that $\Div h_\e\equiv 0$, for every $p\geq 1$ we have
$v_\varepsilon\to v$ in $L^p$ and
$$\lim_{\varepsilon\to 0^+}\int_{\Omega}\frac{1}{\varepsilon}F\Big(\varepsilon^n\nabla^n
v_\e(x),\ldots,\varepsilon\nabla v_\e(x)\,,\,v(x)\,,\,x\Big)dx=
K_{per}(v).$$ Here $\mathcal{B}(\R^N)$ was defined by
\er{fhjvjhvjhvholhiohiovhhjhvvjvf} and $K_{per}(\cdot)$ was defined
by \er{hfighfighfihgigiugi}.
\end{theorem}
For the equivalent formulation and additional details see Theorem
\ref{vtporbound4ghkfgkhkhhhjjklkjjhliuik}. See also Theorem
\ref{vtporbound4ghkfgkhkgen} as the analogous result for more
general functionals than that defined by \er{fhjvjhvjhv}.

\begin{remark}
If the boundary
of an open set $\O\subset\R^N$ is $\mathcal{H}^{N-1}$
$\sigma$-finite, then the condition in Theorem
\ref{ffgvfgfhthjghgjhg} that $|Dv|(\partial\Omega)=0$ for $v\in
BV(\R^N)$ is equivalent to saying that
$\mathcal{H}^{N-1}(\partial\O\cap J_v)=0$.
\end{remark}

 Next, as we showed in \cite{PII}, for the general problem \er{fhjvjhvjhv},
when $G,W$ don't depend on $x$ explicitly, $K^*(\cdot)$ is a lower
bound in the sense of {\bf(*)}. More precisely, we have the
following Theorem:
\begin{theorem}\label{dehgfrygfrgygenjklhhjkghhjggjfjkh}
Let $\O\subset\R^N$ be an open set and
$$F:\R^{\big(\{k\times N\}+\{d\times N\}+m\big)\times N^n}\times\ldots\times\R^{\big(\{k\times N\}+\{d\times N\}+m\big)\times N}
\times\R^{\{k\times N\}+\{d\times N\}+m}\,\to\,\R$$ be a nonnegative
continuous function. Furthermore assume that $v:=(\nabla
u,h,\psi)\in\mathcal{B}(\O)\cap BV\big(\O,\R^{k\times
N}\times\R^{d\times N}\times\R^m\big)\cap
L^\infty\big(\O,\R^{k\times N}\times\R^{d\times N}\times\R^m\big)$
satisfies
$$F\Big(0,\ldots,0,v(x)\Big)=0\quad\text{for a.e.}\;\;x\in\O.$$
Then for every $\{v_\varepsilon\}_{\varepsilon>0}\subset
\mathcal{B}(\Omega)\cap W^{n,1}_{loc}\big(\O,\R^{k\times
N}\times\R^{d\times N}\times\R^m\big)$, such that $v_\varepsilon\to
v$ in $L^p$ as $\e\to 0^+$, we have $$\liminf_{\varepsilon\to
0^+}\int_{\Omega}\frac{1}{\varepsilon}F\Big(\varepsilon^n\nabla^n
v_\e(x),\ldots,\varepsilon\nabla v_\e(x)\,,\,v(x)\Big)dx\geq
K^{*}(v).$$
Here $K^{*}(\cdot)$ is defined by \er{hfighfighfihhioh}
with respect to $L^p$ topology.
\end{theorem}

As we saw there is a natural question: whether in general
$K^*(\cdot)\equiv K_{per}(\cdot)\,$? The answer yes will mean that,
in the case when $G,W$ are $C^1$ functions which don't depend on $x$
explicitly, the upper bound in Theorem \ref{ffgvfgfhthjghgjhg} will
coincide with the lower bound of Theorem
\ref{dehgfrygfrgygenjklhhjkghhjggjfjkh} and therefore we will find
the full $\Gamma$-limit in the case of $BV$ limiting functions. The
equivalent question is whether
\begin{equation*}
E_{abst}(v^+,v^-,\vec\nu,x)= E_{per}(v^+,v^-,\vec\nu,x),
\end{equation*}
where $E_{per}(\cdot)$ is defined in \er{Energia2} and
$E_{abst}(\cdot)$ is defined by \er{Energia3}. In \cite{PII} we
formulate and prove some partial results that refer to this
important question. In particular we prove that this is indeed the
case for the general problem \er{b2..} i.e. when we have no
prescribed differential constraint. More precisely, we have the
following Theorem:
\begin{theorem}\label{dehgfrygfrgygenbgggggggggggggkgkgthtjtfnewbhjhjkgj}
Let $G\in C^1\big(\R^{m\times N^n}\times\R^{m\times
N^{(n-1)}}\times\ldots\times\R^{m\times N}\times \R^m,\R\big)$ and
$W\in C^1(\R^m,\R)$ be nonnegative functions such that
$G\big(0,0,\ldots,0,b)= 0$ for every $b\in\R^m$ and there exist
$C>0$ and $p\geq 1$ satisfying
\begin{multline}\label{hgdfvdhvdhfvjjjjiiiuyyyjitghujtrnewkhjklhkl}
\frac{1}{C}|A|^p
\leq F\Big(A,a_1,\ldots,a_{n-1},b\Big) \leq
C\bigg(|A|^p+\sum_{j=1}^{n-1}|a_j|^{p}+|b|^p+1\bigg)\quad \text{for
every}\;\;\big(A,a_1,a_2,\ldots,a_{n-1},b\big),
\end{multline}
where we denote
$$F\Big(A,a_1,\ldots,a_{n-1},b\Big):=G\Big(A,a_1,\ldots,a_{n-1},b\Big)+W(b)$$
Next let $\psi\in BV(\R^N,\R^{m})\cap L^\infty$ be such that $\|D
\psi\|(\partial\Omega)=0$ and $W\big(\psi(x)\big)=0$ for a.e.
$x\in\O$.
Then $K^*(\psi)=K_{per}(\psi)$ and for every
$\{\varphi_\e\}_{\e>0}\subset W^{n,p}_{loc}(\O,\R^m)$ such that
$\varphi_\e\to \psi$ in $L^p_{loc}(\O,\R^m)$ as $\e\to 0^+$, we have
\begin{multline}\label{a1a2a3a4a5a6a7s8hhjhjjhjjjjjjkkkkgenhjhhhhjtjurtnewjgvjhgv}
\liminf_{\e\to 0^+}I_\e(\varphi_\e):=\liminf_{\e\to
0^+}\frac{1}{\e}\int_\O F\bigg(\,\e^n\nabla^n
\varphi_\e(x),\,\e^{n-1}\nabla^{n-1}\varphi_\e(x),\,\ldots,\,\e\nabla \varphi_\e(x),\, \varphi_\e(x)\bigg)dx\\
\geq
\int_{\O\cap J_\psi}\bar
E_{per}\Big(\psi^+(x),\psi^-(x),\vec \nu(x)\Big)d \mathcal
H^{N-1}(x)\,,
\end{multline}
where
\begin{multline}\label{L2009hhffff12kkkhjhjghghgvgvggcjhggghtgjutnewjgkjgjk}
\bar E_{per}\Big(\psi^+,\psi^-,\vec \nu\Big)\;:=\;\\
\inf\Bigg\{\int_{Q_{\vec \nu}}\frac{1}{L} F\bigg(L^n\,\nabla^n
\zeta,\,L^{n-1}\,\nabla^{n-1} \zeta,\,\ldots,\,L\,\nabla
\zeta,\,\zeta\bigg)\,dx:\;\; L\in(0,+\infty)\,,\;\zeta\in
\mathcal{\tilde D}_{per}(\psi^+,\psi^-,\vec \nu)\Bigg\}\,,
\end{multline}
with
\begin{multline}\label{L2009Ddef2hhhjjjj77788hhhkkkkllkjjjjkkkhhhhffggdddkkkgjhikhhhjjddddkkjkjlkjlkintuukgkggk}
\mathcal{\tilde D}_{per}(\psi^+,\psi^-,\vec \nu):=
\bigg\{\zeta\in C^n(\R^N,\R^m):\;\;\zeta(y)=\psi^-\;\text{ if }\;y\cdot\vec\nu\leq-1/2,\\
\zeta(y)=\psi^+\;\text{ if }\; y\cdot\vec\nu(x)\geq 1/2\;\text{ and
}\;\zeta\big(y+\vec k_j\big)=\zeta(y)\;\;\forall j=2,3,\ldots,
N\bigg\}\,.
\end{multline}
Here $Q_{\vec \nu}:=\{y\in\R^N:\;|y\cdot \vec
\nu_j|<1/2\;\;\;\forall j=1,2\ldots N\}$ where $\{\vec \nu_1,\vec
\nu_2,\ldots,\vec \nu_N\}\subset\R^N$ is an orthonormal base in
$\R^N$ such that $\vec \nu_1:=\vec \nu$. Moreover, there exists e
sequence $\{\psi_\e\}_{\e>0}\subset C^\infty(\R^N,\R^m)$ such that
$\int_\O\psi_\e(x)dx=\int_\O \psi(x)dx$, for every $q\geq 1$ we have
$\psi_\e\to \psi$ in $L^q(\O,\R^m)$ as $\e\to 0^+$, and we have
\begin{multline}\label{a1a2a3a4a5a6a7s8hhjhjjhjjjjjjkkkkgenhjhhhhjtjurtgfhfhfjfjfjnewjkggujk}
\lim_{\e\to 0^+}I_\e(\psi_\e):=\lim_{\e\to 0^+}\frac{1}{\e}\int_\O
F\bigg(\,\e^n\nabla^n
\psi_\e(x),\,\e^{n-1}\nabla^{n-1}\psi_\e(x),\,\ldots,\,\e\nabla \psi_\e(x),\, \psi_\e(x)\bigg)dx\\
= \int_{\O\cap J_\psi}\bar
E_{per}\Big(\psi^+(x),\psi^-(x),\vec \nu(x)\Big)d \mathcal
H^{N-1}(x)\,.
\end{multline}
\end{theorem}
\begin{remark}\label{vyuguigiugbuikkk}
In what follows we use some special notations and apply some basic
theorems about $BV$ functions. For the convenience of the reader we
put these notations and theorems in Appendix.
\end{remark}

\section{Preliminary results}
\begin{definition}\label{defkub}
Let $\vec T\in \mathcal{L}(\R^{d\times N};\R^m)$. For any
$A_1,A_2\in\big\{A\in\R^{m}:\exists M\in\R^{d\times N},\;A=\vec
T\cdot M\big\}$, for every $\vec\nu\in\R^N$ satisfying $|\vec\nu|=1$
and for every system of $N-1$ linearly independent vectors $\{\vec
a_1,\vec a_2,\ldots,\vec a_{N-1}\}$ in $\R^N$ satisfying $\vec
a_j\cdot\vec\nu=0$ for all $j=1,2,\ldots,(N-1)$ set
\begin{multline}\label{L2009Ddef2hhhjjjj77788}
\mathcal{D}\big(\vec T,A_2,A_1,\vec\nu,\vec a_1,\vec a_2,\ldots,\vec
a_{N-1}\big):= \bigg\{ u\in C^\infty(\R^N,\R^d):\;\vec T\cdot\nabla
u(y)=A_1\;\text{ if }\;y\cdot\vec\nu\leq-1/2,\\
\vec T\cdot\nabla u(y)=A_2\;\text{ if }\; y\cdot\vec\nu\geq
1/2\;\text{ and }\;\vec T\cdot\nabla u(y+\vec a_j)=\vec T\cdot\nabla
u(y)\;\;\forall j=1,2,\ldots, (N-1)\bigg\}\,,
\end{multline}
and
\begin{multline}\label{L2009Ddef2hhhjjjj77788kkktttttt8888}
\mathcal{\tilde D}\big(\vec T,A_2,A_1,\vec\nu,\vec a_1,\vec
a_2,\ldots,\vec a_{N-1}\big):= \bigg\{u\in
\mathcal{D}'(\R^N,\R^d):\\
\vec T\cdot\nabla u\in C^1(\R^N,\R^m),\;\;\vec T\cdot\nabla
u(y)=A_1\;\text{ if }\;y\cdot\vec\nu\leq-1/2,\\
\vec T\cdot\nabla u(y)=A_2\;\text{ if }\; y\cdot\vec\nu\geq
1/2\;\text{ and }\;\vec T\cdot\nabla u(y+\vec a_j)=\vec T\cdot\nabla
u(y)\;\;\forall j=1,2,\ldots, (N-1)\bigg\}\,.
\end{multline}
Set also
\begin{multline}\label{L2009Ddef2hhhjjjj77788null}
\mathcal{D}_0\big(\vec\nu,\vec a_1,\vec a_2,\ldots,\vec
a_{N-1}\big):= \Big\{ u\in C^\infty(\R^N,\R^d):\;
u(y)=0\;\text{ if }\; |y\cdot\vec\nu|\geq 1/2\;\\ \text{ and
}\;u(y+\vec a_j)=u(y)\;\;\forall j=1,2,\ldots, (N-1)\Big\}\,.
\end{multline}
and
\begin{multline}\label{L2009Ddef2hhhjjjj77788nullkkk}
\mathcal{D}_1\big(\vec T,\vec\nu,\vec a_1,\vec a_2,\ldots,\vec
a_{N-1}\big):=\\ \bigg\{ u\in C^\infty(\R^N,\R^d)\cap
L^\infty(\R^N,\R^d)\cap Lip(\R^N,\R^d):\;
\vec T\cdot\nabla u(y)=0\;\text{ if }\; |y\cdot\vec\nu|\geq 1/2\;\\
\text{ and }\;
u(y+\vec a_j)=u(y) \;\;\forall j=1,2,\ldots, (N-1)\bigg\}\,.
\end{multline}
\end{definition}
\begin{proposition}\label{L2009cffgfProp8hhhjjj8887788}
Let $\vec T\in \mathcal{L}(\R^{d\times N};\R^m)$ and $G\in
C^1(\R^{m\times N}\times\R^{m}\times\R)$, satisfying $G\geq 0$, and
let $A_1,A_2\in \big\{A\in\R^{m}:\exists M\in\R^{d\times N},\;A=\vec
T\cdot M\big\}$ satisfy $G(0,A_1,-1)=0$ and $G(0,A_2,1)=0$.
Furthermore, let $\vec\nu\in\R^N$ satisfies $|\vec\nu|=1$. Given a
system $\{\vec p_1,\vec p_2,\ldots,\vec p_{N-1}\}$ of $(N-1)$
linearly independent vectors in $\R^N$, satisfying $\vec
p_j\cdot\vec\nu=0$ for all $j=1,2,\ldots,(N-1)$, set
\begin{multline*}\Theta(\vec p_1,\ldots,\vec p_{N-1}):=\\
\inf\Bigg\{ \frac{1}{L}\int\limits_{I_{N}}G\bigg(L\,\nabla\{\vec
T\cdot\nabla u\}\Big(s_1\vec\nu+\Sigma_{j=1}^{N-1}s_{j+1}\vec
p_j\Big),\,\vec T\cdot\nabla
u\Big(s_1\vec\nu+\Sigma_{j=1}^{N-1}s_{j+1}\vec p_j\Big),\,
s_1/|s_1|\bigg)\,ds
:\\ \quad L>0,\, u\in \mathcal{D}\big(\vec T,A_2,A_1,\vec\nu,\vec
p_1,\ldots,\vec p_{N-1}\big)\Bigg\}\quad\quad\text{(see Definition
\ref{defkub})},
\end{multline*}
where
\begin{equation}\label{cubidepggh}
I_N:= \Big\{s\in\R^N:\; -1/2<s_j<1/2\quad\forall j=1,2,\ldots,
N\Big\}\,.
\end{equation}
Then for every two systems $\{\vec a_1,\vec a_2,\ldots,\vec
a_{N-1}\}$ and $\{\vec b_1,\vec b_2,\ldots,\vec b_{N-1}\}$ of
$(N-1)$ linearly independent vectors in $\R^N$, satisfying $\vec
a_j\cdot\vec\nu=0$ and $\vec b_j\cdot\vec\nu=0$ for all
$j=1,2,\ldots,(N-1)$, we have
\begin{equation}\label{nezbas}
\Theta(\vec a_1,\ldots,\vec a_{N-1})=\Theta(\vec b_1,\ldots,\vec
b_{N-1})\,.
\end{equation}
\end{proposition}
\begin{proof}
%
%
%
%
%
%
%
%
%
%
First of all we observe that if $\{\vec a_1,\vec a_2,\ldots,\vec
a_{N-1}\}$ and $\{\vec b_1,\vec b_2,\ldots,\vec b_{N-1}\}$ are
equal, up to a permutation of vectors, then \er{nezbas} will follow
by the definition. Next since for every $(N-1)$ positive integer
numbers $K_1,
\ldots,K_{N-1}$
we have $\mathcal{D}(\vec T,A_2,A_1,\vec\nu,\vec a_1,\ldots,\vec
a_{N-1})\subset\mathcal{D}(\vec T,A_2,A_1,\vec\nu,K_1\vec
a_1,\ldots,K_{N-1}\vec a_{N-1})$ then
\begin{multline}\label{esheodnoner}
\Theta(K_1\vec a_1,\ldots,K_{N-1}\vec a_{N-1})=\\ \inf\Bigg\{
\frac{1}{L}\int\limits_{I_{N}}G\bigg(L\,\nabla\{\vec T\cdot\nabla
u\}\Big(s_1\vec\nu+\Sigma_{j=1}^{N-1}s_{j+1}K_j\vec a_j\Big),\,\vec
T\cdot\nabla u\Big(s_1\vec\nu+\Sigma_{j=1}^{N-1}s_{j+1}K_j\vec
a_j\Big),\, s_1/|s_1|\bigg)\,ds
:\\ \quad L>0,\, u\in \mathcal{D}\big(\vec T,A_2,A_1,\vec\nu,K_1\vec
a_1,\ldots,K_{N-1}\vec a_{N-1}\big)\Bigg\}\leq\\
\inf\Bigg\{ \frac{1}{L}\int\limits_{I_{N}}G\bigg(L\,\nabla\{\vec
T\cdot\nabla u\}\Big(s_1\vec\nu+\Sigma_{j=1}^{N-1}s_{j+1}K_j\vec
a_j\Big),\,\vec T\cdot\nabla
u\Big(s_1\vec\nu+\Sigma_{j=1}^{N-1}s_{j+1}K_j\vec a_j\Big),\,
s_1/|s_1|\bigg)\,ds
:\\
\quad L>0,\, u\in \mathcal{D}\big(\vec T,A_2,A_1,\vec\nu,\vec
a_1,\ldots,\vec a_{N-1}\big)\Bigg\}\,.
\end{multline}
Thus changing variables of the integration $\bar s_1=s_1$ and $\bar
s_j=K_{j-1}s_j$ for every $j=2,\ldots,N$ in the r.h.s. of
\er{esheodnoner} and using the periodicity condition in the
definition of $\mathcal{D}\big(\vec T,A_2,A_1,\vec\nu,\vec
a_1,\ldots,\vec a_{N-1}\big)$ gives
\begin{equation}\label{nezbasner}
\Theta(K_1\vec a_1,\ldots,K_{N-1}\vec a_{N-1})\leq \Theta(\vec
a_1,\ldots,\vec a_{N-1})\,.
\end{equation}
On the other hand for every positive integer number $K$ we have that
if $u\in \mathcal{D}\big(\vec T,A_2,A_1,\vec\nu,K\vec
a_1,\ldots,K\vec a_{N-1}\big)$ then the function $\bar
u(y):=\frac{1}{K}u(Ky)\in\mathcal{D}\big(\vec T,A_2,A_1,\vec\nu,\vec
a_1,\ldots,\vec a_{N-1}\big)$ and therefore,
\begin{multline}\label{esheodnonervdrug}
\Theta(K\vec a_1,\ldots,K\vec a_{N-1})=\\ \inf\Bigg\{
\frac{1}{L}\int\limits_{I_{N}}G\bigg(L\,\nabla\{\vec T\cdot\nabla
u\}\Big(s_1\vec\nu+\Sigma_{j=1}^{N-1}s_{j+1}K\vec a_j\Big),\,\vec
T\cdot\nabla u\Big(s_1\vec\nu+\Sigma_{j=1}^{N-1}s_{j+1}K\vec
a_j\Big),\, s_1/|s_1|\bigg)\,ds
:\\ \quad L>0,\, u\in \mathcal{D}\big(\vec T,A_2,A_1,\vec\nu,K\vec
a_1,\ldots,K\vec a_{N-1}\big)\Bigg\}\geq\\
\inf\Bigg\{ \frac{1}{L}\int\limits_{I_{N}}G\bigg(L\,\nabla\{\vec
T\cdot\nabla u\}\Big(s_1\vec\nu+\Sigma_{j=1}^{N-1}s_{j+1}K\vec
a_j\Big),\,\vec T\cdot\nabla
u\Big(s_1\vec\nu+\Sigma_{j=1}^{N-1}s_{j+1}K\vec a_j\Big),\,
s_1/|s_1|\bigg)\,ds
:\\ \quad L>0,\, \frac{1}{K}u(Ky)\in\mathcal{D}\big(\vec
T,A_2,A_1,\vec\nu,\vec a_1,\ldots,\vec a_{N-1}\big)\Bigg\}\,.
\end{multline}
Therefore, changing variables $\bar s_1:=s_1/K$, $\bar s_j=s_j$ for
$j=2,\ldots,N$ in the r.h.s. of \er{esheodnonervdrug} we obtain
\begin{multline}\label{esheodnonervdrugprod}
\Theta(K\vec a_1,\ldots,K\vec a_{N-1})\geq\\ \inf\Bigg\{
\frac{K}{L}\int\limits_{I_{N}}G\bigg((L/K)\,\nabla\{\vec
T\cdot\nabla \bar u\}\Big(s_1\vec\nu+\Sigma_{j=1}^{N-1}s_{j+1}\vec
a_j\Big),\,\vec T\cdot\nabla \bar
u\Big(s_1\vec\nu+\Sigma_{j=1}^{N-1}s_{j+1}\vec a_j\Big),\,
s_1/|s_1|\bigg)\,ds
:\\ \quad L>0,\, \bar u\in\mathcal{D}\big(\vec
T,A_2,A_1,\vec\nu,\vec a_1,\ldots,\vec
a_{N-1}\big)\Bigg\}=\Theta(\vec a_1,\ldots,\vec a_{N-1})\,.
\end{multline}
Plugging \er{esheodnonervdrugprod} to \er{nezbasner} with $K_j:=K$
we deduce
\begin{equation}\label{nezbasnerrav}
\Theta(K\vec a_1,\ldots,K\vec a_{N-1})=\Theta(\vec a_1,\ldots,\vec
a_{N-1})\,.
\end{equation}
Thus from \er{nezbasner} and \er{nezbasnerrav} we deduce
\begin{equation}\label{nezbasnerratz223}
\Theta\big((K_1/K)\vec a_1,\ldots,(K_{N-1}/K)\vec a_{N-1}\big)\leq
\Theta(\vec a_1,\ldots,\vec a_{N-1})\,.
\end{equation}
So for every $(N-1)$ positive rational numbers $r_1,\ldots ,r_{N-1}$
we have
\begin{equation}\label{nezbasnerratz224}
\Theta(r_1\vec a_1,\ldots,r_{N-1}\vec a_{N-1})\leq \Theta(\vec
a_1,\ldots,\vec a_{N-1})\,.
\end{equation}
Then also
\begin{equation}\label{nezbasnerratz225}
\Theta(\vec a_1,\ldots,\vec a_{N-1})=\Theta\big((1/r_1)r_1\vec
a_1,\ldots,(1/r_{N-1})r_{N-1}\vec a_{N-1}\big)\leq \Theta(r_1\vec
a_1,\ldots,r_{N-1}\vec a_{N-1})\,.
\end{equation}
Therefore for every $(N-1)$ positive rational numbers $r_1,\ldots,
r_{N-1}$ we must have
\begin{equation}\label{nezbasnerratz222888899}
\Theta(r_1\vec a_1,\ldots,r_{N-1}\vec a_{N-1})=\Theta(\vec
a_1,\ldots,\vec a_{N-1})\,.
\end{equation}
Finally since $$\mathcal{D}\big(A_2,A_1,\vec\nu,\vec a_1,\ldots,\vec
a_j\ldots,\vec a_{N-1}\big)=\mathcal{D}\big(A_2,A_1,\vec\nu,\vec
a_1,,\ldots,(-\vec a_j)\ldots,\vec a_{N-1}\big)\,,$$ we deduce that
the equality
\begin{equation}\label{nezbasnerratz2228888997788899}
\Theta(r_1\vec a_1,\ldots,r_{N-1}\vec a_{N-1})=\Theta(\vec
a_1,\ldots,\vec a_{N-1})
\end{equation}
is valid for every $(N-1)$ different from zero rational numbers
$r_1,\ldots, r_{N-1}$ (without any restriction on their sign).

 Next since obviously
$$\mathcal{D}\big(\vec T,A_2,A_1,\vec\nu,\vec
a_1,\vec a_2,\ldots,\vec a_{N-1}\big)=\mathcal{D}\big(\vec
T,A_2,A_1,\vec\nu,\vec a_1,(\vec a_1+\vec a_2),\vec a_3,\ldots,\vec
a_{N-1}\big)$$ we have
\begin{equation}
\label{esheodnonervdrugprodsum}
\begin{split}
&\Theta\big(\vec a_1,(\vec a_1+\vec a_2),\vec a_3,\ldots,\vec
a_{N-1}\big)=\inf\Bigg\{ \frac{1}{L}\int\limits_{I_{N}}\times
\\G\bigg(L\nabla &\{\vec T\cdot\nabla
u\}\Big(s_1\vec\nu+(s_2+s_3)\vec a_1+\Sigma_{j=2}^{N-1}s_{j+1}\vec
a_j\Big),\vec T\cdot\nabla u\Big(s_1\vec\nu+(s_2+s_3)\vec
a_1+\Sigma_{j=2}^{N-1}s_{j+1}\vec a_j\Big), s_1/|s_1|\bigg)\\
\times ds&:
\quad\quad\quad L>0,\, u\in\mathcal{D}\big(\vec
T,A_2,A_1,\vec\nu,\vec a_1,\vec a_2,\ldots,\vec
a_{N-1}\big)\Bigg\}\,.
\end{split}
\end{equation}
Therefore, changing variables $\bar s_j=s_j$ if $j\neq 2$ and $\bar
s_2=s_2+s_3$ in the r.h.s. of \er{esheodnonervdrugprodsum} and using
the periodicity condition of the Definition \ref{defkub} gives
\begin{equation}\label{esheodnonervdrugprodsumsled}
\Theta\big(\vec a_1,(\vec a_1+\vec a_2),\vec a_3,\ldots,\vec
a_{N-1}\big)=\Theta(\vec a_1,\vec a_2,\ldots,\vec a_{N-1})\,.
\end{equation}
Next let $\vec b_j=\Sigma_{j=1}^{N-1}Q_{jk}\vec a_k$, where
$\big\{Q_{jk}\big\}\in\R^{(N-1)\times(N-1)}$ is a non-degenerate
matrix with rational coefficients. Then we can obtain the system
$\{\vec b_1,\vec b_2,\ldots,\vec b_{N-1}\}$ from the system $\{\vec
a_1,\vec a_2,\ldots,\vec a_{N-1}\}$ step by step by applying the
following three types of operations
\begin{itemize}
\item Multiplying the vectors of the system with different from zero
rational numbers.
\item Permutation of the vectors of the system.
\item Adding the first vector of the system to the second one.
\end{itemize}
Since by \er{nezbasnerratz2228888997788899} and
\er{esheodnonervdrugprodsumsled} every step keeps the same
$\Theta(\cdot,\ldots,\cdot)$ we deduce that the equality
\begin{equation}\label{nezbasraz}
\Theta(\vec a_1,\ldots,\vec a_{N-1})=\Theta(\vec b_1,\ldots,\vec
b_{N-1})\,,
\end{equation}
is valid for every $\vec b_j=\Sigma_{j=1}^{N-1}Q_{jk}\vec a_k$,
where $\big\{Q_{jk}\big\}$
is a non-degenerate matrix with rational coefficients.

 Finally consider the general situation where $\vec b_j=\Sigma_{j=1}^{N-1}R_{jk}\vec a_k$, where
$\big\{R_{jk}\big\}\in\R^{(N-1)\times(N-1)}$ is a non-degenerate
matrix with real coefficients. Let $\delta>0$ be an arbitrary small
positive number. By the definition of $\Theta(\cdot,\ldots,\cdot)$
there exists $L_\delta>0$ and $u_\delta\in\mathcal{D}\big(\vec
T,A_2,A_1,\vec\nu,\vec a_1,\vec a_2,\ldots,\vec a_{N-1}\big)$ such
that
\begin{multline}\label{hvvLinfgen7888999}
\frac{1}{L_\delta}\int\limits_{I_{N}}G\bigg(L_\delta\,\nabla\{\vec
T\cdot\nabla u_\delta\}\Big(s_1\vec\nu+\Sigma_{j=1}^{N-1}s_{j+1}\vec
a_j\Big),\,\vec T\cdot\nabla
u_\delta\Big(s_1\vec\nu+\Sigma_{j=1}^{N-1}s_{j+1}\vec a_j\Big),\,s_1/|s_1|\bigg)\,ds
\\<\Theta(\vec a_1,\ldots,\vec
a_{N-1})+\frac{\delta}{2}\,.
\end{multline}
There exists a sequence of non-degenerate matrices with rational
coefficients
$\big\{\{Q^{(n)}_{jk}\}\big\}_{n=1}^{\infty}\subset\R^{(N-1)\times(N-1)}$
with the property that
$\lim_{n\to\infty}\{Q^{(n)}_{jk}\}=\{R_{jk}\}$ and
$\lim_{n\to\infty}\{Q^{(n)}_{jk}\}^{-1}=\{R_{jk}\}^{-1}$. In
particular if we set
$\{P^{(n)}_{jk}\}:=\{Q^{(n)}_{jk}\}^{-1}\cdot\{R_{jk}\}$ then
$\lim_{n\to\infty}\{P^{(n)}_{jk}\}=I_{N-1}$ where
$I_{N-1}\in\R^{(N-1)\times(N-1)}$ is the identity matrix. For every
$n$ consider $(N-1)$ linearly independent vectors $\vec
c^{(n)}_1,\ldots,\vec c^{(n)}_{N-1}$ satisfying $\vec
c^{(n)}_j=\Sigma_{k=1}^{N-1}P^{(n)}_{jk}\vec a_k$ for all $j$. Thus
$\vec b_j=\Sigma_{k=1}^{N-1}Q^{(n)}_{jk}\vec c^{(n)}_k$. Next let
$u_n(y):\R^N\to\R^d$ be defined by
$$u_n\Big(s_1\vec\nu+\Sigma_{j=1}^{N-1}s_{j+1}\vec
c^{(n)}_j\Big)=u_\delta\Big(s_1\vec\nu+\Sigma_{j=1}^{N-1}s_{j+1}\vec
a_j\Big)\quad\quad\forall s=(s_1,s_2,\ldots,s_N)\in\R^N\,.$$ Thus by
\er{hvvLinfgen7888999} and by the equality
$\lim_{n\to\infty}\{P^{(n)}_{jk}\}=I_{N-1}$, for sufficiently large
$n$ we must have
\begin{multline}\label{hvvLinfgen7888999prod9999}
\frac{1}{L_\delta}\int\limits_{I_{N}}G\bigg(L_\delta\,\nabla\{\vec
T\cdot\nabla u_n\}\Big(s_1\vec\nu+\Sigma_{j=1}^{N-1}s_{j+1}\vec
c^{(n)}_j\Big),\,\vec T\cdot\nabla
u_n\Big(s_1\vec\nu+\Sigma_{j=1}^{N-1}s_{j+1}\vec c^{(n)}_j\Big)\bigg)\,ds
\\<\Theta(\vec a_1,\ldots,\vec
a_{N-1})+\delta\,.
\end{multline}
However, we have $u_n\in\mathcal{D}\big(\vec T,A_2,A_1,\vec\nu,\vec
c^{(n)}_1,\vec c^{(n)}_2,\ldots,\vec c^{(n)}_{N-1}\big)$. Thus by
\er{hvvLinfgen7888999prod9999} we obtain
\begin{equation}\label{yhffjgjk7788hhhpppp88899999}\Theta(\vec
c^{(n)}_1,\ldots,\vec c^{(n)}_{N-1})<\Theta(\vec a_1,\ldots,\vec
a_{N-1})+\delta\,.
\end{equation}
On the other hand, by \er{nezbasraz} and the equality $\vec
b_j=\Sigma_{k=1}^{N-1}Q^{(n)}_{jk}\vec c^{(n)}_k$ we have
$$\Theta(c^{(n)}_1,\ldots,\vec c^{(n)}_{N-1})=\Theta(\vec b_1,\ldots,\vec
b_{N-1})\,.$$ Therefore, by \er{yhffjgjk7788hhhpppp88899999}
\begin{equation}\label{yhffjgjk7788hhhpppp88899999888999}\Theta(\vec b_1,\ldots,\vec
b_{N-1})<\Theta(\vec a_1,\ldots,\vec a_{N-1})+\delta\,,
\end{equation}
and since $\delta>0$ was arbitrary small we finally get
\begin{equation}\label{yhffjgjk7788hhhpppp88899999888999ner9999}\Theta(\vec b_1,\ldots,\vec
b_{N-1})\leq\Theta(\vec a_1,\ldots,\vec a_{N-1})\,.
\end{equation}
Interchanging the roles of $\{\vec a_1,\ldots,\vec a_{N-1}\}$ and
$\{\vec b_1,\ldots,\vec b_{N-1}\}$ we obtain the opposite
inequality. So in fact we have the equality
$$\Theta(\vec b_1,\ldots,\vec
b_{N-1})=\Theta(\vec a_1,\ldots,\vec a_{N-1})$$
and the result
follows.
\end{proof}

\begin{lemma}\label{lijkkhnjjhjh}
Let $\vec T\in \mathcal{L}(\R^{d\times N};\R^m)$ and let $G\in
C^1(\R^{m\times N}\times\R^{m}\times\R)$, satisfying $G\geq 0$, and
let $A_1,A_2,\vec\nu$, $\{\vec a_1,\vec a_2,\ldots,\vec a_{N-1}\}$
be the same as in Definition \ref{defkub} and satisfy
$G(0,A_1,-1)=0$ and $G(0,A_2,1)=0$. Furthermore, set
\begin{multline}\label{vgjgkhjkcgswqekklll}
R_L:=\inf\Bigg\{
\frac{1}{L}\int\limits_{I_{N}}G\bigg(L\,\nabla\{\vec T\cdot\nabla
u\}\Big(s_1\vec\nu+\Sigma_{j=1}^{N-1}s_{j+1}\vec a_j\Big),\,\vec
T\cdot\nabla u\Big(s_1\vec\nu+\Sigma_{j=1}^{N-1}s_{j+1}\vec
a_j\Big),\,s_1/|s_1|\bigg)\,ds
:\\ \quad u\in \mathcal{D}\big(\vec T,A_2,A_1,\vec\nu,\vec
a_1,\ldots,\vec a_{N-1}\big)\Bigg\}\,,
\end{multline}
and
\begin{multline}\label{vgjgkhjkcgswqekklllkgdklhjgppp}
\tilde R_L:=\inf\Bigg\{
\frac{1}{L}\int\limits_{I_{N}}G\bigg(L\,\nabla\{\vec T\cdot\nabla
u\}\Big(s_1\vec\nu+\Sigma_{j=1}^{N-1}s_{j+1}\vec a_j\Big),\,\vec
T\cdot\nabla u\Big(s_1\vec\nu+\Sigma_{j=1}^{N-1}s_{j+1}\vec
a_j\Big),\,s_1/|s_1|\bigg)\,ds
:\\ \quad u\in \mathcal{\tilde D}\big(\vec T,A_2,A_1,\vec\nu,\vec
a_1,\ldots,\vec a_{N-1}\big)\Bigg\}
\end{multline}
(see Definition \ref{defkub}). Then
\begin{equation}\label{fgggnullreslslbhhhhhhhhh}
R_L=\tilde R_L\quad\quad\forall L>0\,.
\end{equation}
and
\begin{equation}\label{fgggnullreslslb}
\inf\limits_{L>0}R_L=\inf\limits_{L\in(0,1)}R_L=\lim\limits_{L\to
0^+}R_L\,.
\end{equation}
\end{lemma}
\begin{proof}
Firstly we will prove \er{fgggnullreslslbhhhhhhhhh}. By Definition
\ref{defkub} we clearly have $$\mathcal{D}\big(\vec
T,A_2,A_1,\vec\nu,\vec a_1,\ldots,\vec
a_{N-1}\big)\subset\mathcal{\tilde D}\big(\vec
T,A_2,A_1,\vec\nu,\vec a_1,\ldots,\vec a_{N-1}\big).$$ Therefore for
every $L>0$ we clearly deduce $R_L\geq\tilde R_L$. Next fix an
arbitrary $u\in \mathcal{\tilde D}\big(\vec T,A_2,A_1,\vec\nu,\vec
a_1,\ldots,\vec a_{N-1}\big)$ and consider $\zeta(z)\in
C^\infty_c(\R^N,\R)$ such that $\supp\zeta\subset\subset B_1(0)$,
$\zeta\geq 0$ and $\int_{\R^N}\zeta(z)dz=1$. For any $\e>0$ and any
fixed $x\in\R^N$ set
\begin{equation}\label{a4543chejhgufydfdtklljkjjk}
\bar
u_\e(x):=\frac{1}{\e^N}\Big<\zeta\Big(\frac{y-x}{\e}\Big),u(y)\Big>
\end{equation}
(see notations and definitions in the Appendix). Then $\bar u_\e\in
C^\infty(\R^N,\R^d)$. Moreover, clearly
\begin{equation}\label{a4543chehlkjkklkk}
\vec T\cdot \{\nabla \bar
u_\e(x)\}=\frac{1}{\e^N}\int_{\R^N}\zeta\Big(\frac{y-x}{\e}\Big)
\cdot\{\vec T\cdot\nabla u\}(y)dy=
\int_{\R^N}\zeta(z)\cdot\big\{\vec T\cdot\nabla u\big\}(x+\e z)dz,
\end{equation}
and
\begin{equation}\label{a4543chehlkjkklkkytju}
\nabla\big(\vec T\cdot \{\nabla \bar
u_\e(x)\}\big)=\frac{1}{\e^N}\int_{\R^N}\zeta\Big(\frac{y-x}{\e}\Big)
\cdot\nabla\{\vec T\cdot\nabla u\}(y)dy=
\int_{\R^N}\zeta(z)\cdot\nabla\big\{\vec T\cdot\nabla u\big\}(x+\e
z)dz,
\end{equation}
Then by the definition of $\mathcal{\tilde D}\big(\vec
T,A_2,A_1,\vec\nu,\vec a_1,\ldots,\vec a_{N-1}\big)$ and by
\er{a4543chehlkjkklkk} we have
\begin{multline}\label{hdsudsgcudggullljjjlll}
\vec T\cdot\nabla \bar u_\e(y)=A_1\;\text{ if
}\;y\cdot\vec\nu\leq-(1/2+\e),\; \vec T\cdot\nabla
\bar u_\e(y)=A_2\;\text{ if }\; y\cdot\vec\nu\geq (1/2+\e)\;\text{ and }\\
\vec T\cdot\nabla \bar u_\e(y+\vec a_j)=\vec T\cdot\nabla \bar
u_\e(y)\;\;\forall j=1,2,\ldots, (N-1)\,.
\end{multline}
Moreover, $\vec T\cdot\nabla \bar u_\e\to \vec T\cdot\nabla u$ as
$\e\to 0^+$ in $C^1(\R^N,\R^m)$ i.e. $\vec T\cdot\nabla \bar u_\e\to
\vec T\cdot\nabla u$ uniformly in $\R^N$ and $\nabla\{\vec
T\cdot\nabla \bar u_\e\}\to \nabla\{\vec T\cdot\nabla u\}$ uniformly
in $\R^N$. Finally define $u_\e\in C^\infty(\R^N,\R^d)$ by the
formula
\begin{equation}\label{a4543chejhgufydfdtklljkjjkujsdguyg}
u_\e\Big(s_1\vec\nu+\Sigma_{j=1}^{N-1}s_{j+1}\vec a_j\Big):=\bar
u_\e\Big((2\e+1)s_1\vec\nu+\Sigma_{j=1}^{N-1}s_{j+1}\vec a_j\Big)\,.
\end{equation}
Then using \er{hdsudsgcudggullljjjlll} we deduce that $u_\e\in
\mathcal{D}\big(\vec T,A_2,A_1,\vec\nu,\vec a_1,\ldots,\vec
a_{N-1}\big)$. Finally, $\vec T\cdot\nabla u_\e\to \vec T\cdot\nabla
u$ as $\e\to 0^+$ in $C^1(\R^N,\R^m)$. Therefore,
\begin{multline}\label{vgjgkhjkcgswqekklllkukukkkk}
R_L\leq\lim\limits_{\e\to
0^+}\Bigg\{\frac{1}{L}\int\limits_{I_{N}}G\bigg(L\,\nabla\{\vec
T\cdot\nabla u_\e\}\Big(s_1\vec\nu+\Sigma_{j=1}^{N-1}s_{j+1}\vec
a_j\Big),\,\vec T\cdot\nabla
u_\e\Big(s_1\vec\nu+\Sigma_{j=1}^{N-1}s_{j+1}\vec
a_j\Big),\,s_1/|s_1|\bigg)\,ds
\Bigg\}\\ =\frac{1}{L}\int\limits_{I_{N}}G\bigg(L\,\nabla\{\vec
T\cdot\nabla u\}\Big(s_1\vec\nu+\Sigma_{j=1}^{N-1}s_{j+1}\vec
a_j\Big),\,\vec T\cdot\nabla
u\Big(s_1\vec\nu+\Sigma_{j=1}^{N-1}s_{j+1}\vec
a_j\Big),\,s_1/|s_1|\bigg)\,ds
\,.
\end{multline}
Thus since $u\in \mathcal{\tilde D}\big(\vec T,A_2,A_1,\vec\nu,\vec
a_1,\ldots,\vec a_{N-1}\big)$ was chosen arbitrary, from
\er{vgjgkhjkcgswqekklllkukukkkk} we deduce $R_L\leq\tilde R_L$,
which together with the reverse inequality, established before,
gives \er{fgggnullreslslbhhhhhhhhh}.

We are going to prove \er{fgggnullreslslb} now. For every positive
integer number $K$ we have that if $$u\in \mathcal{D}\big(\vec
T,A_2,A_1,\vec\nu,\vec a_1,\ldots,\vec a_{N-1}\big)$$ then the
function
\begin{equation}\label{opreresfunkap}
w_K(y):=\frac{1}{K}u(Ky)\in\mathcal{D}\big(\vec
T,A_2,A_1,\vec\nu,\vec a_1,\ldots,\vec a_{N-1}\big)\,,
\end{equation}
Moreover, changing variables in the integration and using the
periodicity conditions of the definition of $$\mathcal{D}\big(\vec
T,A_2,A_1,\vec\nu,\vec a_1,\ldots,\vec a_{N-1}\big)$$ gives
the following equality
\begin{multline}\label{esheodnonervdrugnnnk}
\frac{K}{L}\int\limits_{I_{N}}G\bigg((L/K)\,\nabla\{\vec
T\cdot\nabla w_K\}\Big(s_1\vec\nu+\Sigma_{j=1}^{N-1}s_{j+1}\vec
a_j\Big),\,\vec T\cdot\nabla
w_K\Big(s_1\vec\nu+\Sigma_{j=1}^{N-1}s_{j+1}\vec
a_j\Big),\,s_1/|s_1|\bigg)\,ds
\\=\frac{1}{L}\int\limits_{I_{N}}G\bigg(L\,\nabla\{\vec
T\cdot\nabla u\}\Big(s_1\vec\nu+\Sigma_{j=1}^{N-1}s_{j+1}\vec
a_j\Big),\,\vec T\cdot\nabla
u\Big(s_1\vec\nu+\Sigma_{j=1}^{N-1}s_{j+1}\vec
a_j\Big),\,s_1/|s_1|\bigg)\,ds\,.
\end{multline}
Then in
particular we observe that, for
$L>0$ and a positive
integer $K$, we have
\begin{equation}\label{dihghigbh}
R_L\leq R_{KL}\,.
\end{equation}
%
%
%
%
%
%
%
%
%
%
and thus clearly we have
\begin{equation}\label{fgggnullreslslbjkjkjkkllklk}
\inf\limits_{L>0}R_L=\inf\limits_{L\in(0,1)}R_L=\liminf\limits_{L\to
0^+}R_L\,.
\end{equation}
Finally assume, by the contradiction, that $\limsup\limits_{L\to
0^+}R_L> \liminf\limits_{L\to 0^+}R_L$. Then there exists $\delta>0$
and a sequence $L_n\to 0^+$ as $n\to +\infty$ such that
\begin{equation*}
\lim\limits_{n\to +\infty}R_{L_n}>\inf\limits_{L>0}R_L+2\delta\,.
\end{equation*}
Thus in particular, there exists $L_0>0$ and $u_0\in
\mathcal{D}\big(\vec T,A_2,A_1,\vec\nu,\vec a_1,\ldots,\vec
a_{N-1}\big)$ such that
\begin{equation}
\label{fgggnullreslslbjkjkjkkllklkhjhhjkjolkuolio} \lim\limits_{n\to
+\infty}R_{L_n}>\delta+
\frac{1}{L_0}\int\limits_{I_{N}}G\bigg(L_0\nabla\{\vec T\cdot\nabla
u_0\}\Big(s_1\vec\nu+\Sigma_{j=1}^{N-1}s_{j+1}\vec a_j\Big),\vec
T\cdot\nabla u_0\Big(s_1\vec\nu+\Sigma_{j=1}^{N-1}s_{j+1}\vec
a_j\Big),s_1/|s_1|\bigg)ds.
\end{equation}
Set
\begin{equation}
\label{fgggnullreslslbjkjkjkkllklkhjhhjkjolkkj}
K_n:=\max\big\{K\in\mathbb{N}:K\leq (L_0/L_n)\big\}\,.
\end{equation}
Then by the definition $K_n\leq(L_0/L_n)<K_n+1$ and $\lim_{n\to
+\infty}K_n=+\infty$. Thus if we set $d_n:=L_n\cdot K_n$ then
$\lim_{n\to +\infty}d_n=L_0$. On the other hand by \er{dihghigbh}
and \er{fgggnullreslslbjkjkjkkllklkhjhhjkjolkuolio} we obtain
\begin{multline}\label{fgggnullreslslbjkjkjkkllklkhjhhjkjolkuoliojkjkklklk}
\liminf\limits_{n\to +\infty}R_{d_n}:=\liminf\limits_{n\to
+\infty}R_{K_nL_n}>\delta+\\
\frac{1}{L_0}\int\limits_{I_{N}}G\bigg(L_0\,\nabla\{\vec
T\cdot\nabla u_0\}\Big(s_1\vec\nu+\Sigma_{j=1}^{N-1}s_{j+1}\vec
a_j\Big),\,\vec T\cdot\nabla
u_0\Big(s_1\vec\nu+\Sigma_{j=1}^{N-1}s_{j+1}\vec
a_j\Big),\,s_1/|s_1|\bigg)\,ds
\,.
\end{multline}
Thus in particular by the definition of $R_{d_n}$
\begin{multline}\label{fgggnullreslslbjkjkjkkllklkhjhhjkjolkuoliojkjkklklkkjlll}
\liminf\limits_{n\to
+\infty}\Bigg\{\frac{1}{d_n}\int\limits_{I_{N}}G\bigg(d_n\,\nabla\{\vec
T\cdot\nabla u_0\}\Big(s_1\vec\nu+\Sigma_{j=1}^{N-1}s_{j+1}\vec
a_j\Big),\,\vec T\cdot\nabla
u_0\Big(s_1\vec\nu+\Sigma_{j=1}^{N-1}s_{j+1}\vec
a_j\Big),\,s_1/|s_1|\bigg)\,ds
\Bigg\}>\\ \delta+
\frac{1}{L_0}\int\limits_{I_{N}}G\bigg(L_0\,\nabla\{\vec
T\cdot\nabla u_0\}\Big(s_1\vec\nu+\Sigma_{j=1}^{N-1}s_{j+1}\vec
a_j\Big),\,\vec T\cdot\nabla
u_0\Big(s_1\vec\nu+\Sigma_{j=1}^{N-1}s_{j+1}\vec
a_j\Big),\,s_1/|s_1|\bigg)\,ds
\,,
\end{multline}
which contradicts to the identity $\lim_{n\to +\infty}d_n=L_0$. So
we have \er{fgggnullreslslb}.
\end{proof}

\begin{lemma}\label{resatnulo}
Let $\vec T\in \mathcal{L}(\R^{d\times N};\R^m)$ and let $G\in
C^1(\R^{m\times N}\times\R^{m}\times\R)$, satisfying $G\geq 0$, and
let $A_1,A_2,\vec\nu$, $\{\vec a_1,\vec a_2,\ldots,\vec a_{N-1}\}$
be the same as in Definition \ref{defkub} and satisfy
$G(0,A_1,-1)=0$ and $G(0,A_2,1)=0$.
Let $\theta(t)\in C^1(\R,\R)$ satisfies $\theta(t)=0$ if $t\leq
-1/2$ and $\theta(t)=1$ if $t\geq 1/2$. For every $L\in(0,1)$ define
the function $m_L:\R^N\to\R^m$ by
\begin{equation}\label{odnomer}
m_L(y):=\big(1-\theta(y\cdot\vec\nu/L)\big)A_1
+\theta(y\cdot\vec\nu/L)A_2\,.
\end{equation}
Then
\begin{equation}\label{fgggnullresl}
\inf\limits_{L\in(0,1)}\tilde P_L
%
%
%
=\lim\limits_{L\to 0^+}P_L\,,
%
%
\end{equation}
where
\begin{multline}\label{fgggnullreslslbmmmmnkfjgjlllkkkllll}
P_L :=\inf\Bigg\{
\frac{1}{L}\int\limits_{I_{N}}G\bigg(L\nabla(m_L+\vec T\cdot\nabla
v)\Big(s_1\vec\nu+\Sigma_{j=1}^{N-1}s_{j+1}\vec a_j\Big),(m_L+\vec
T\cdot\nabla v)\Big(s_1\vec\nu+\Sigma_{j=1}^{N-1}s_{j+1}\vec
a_j\Big),\frac{s_1}{|s_1|}\bigg)ds:
\\ \quad v\in \mathcal{D}_0\big(\vec\nu,\vec a_1,\ldots,\vec
a_{N-1}\big)\Bigg\}\,,
\end{multline}
and
\begin{multline}\label{fgggnullreslslbmmmmnkfjgjlllkkkllllhoihiooh}
\tilde P_L :=\inf\Bigg\{
\frac{1}{L}\int\limits_{I_{N}}G\bigg(L\nabla(m_L+\vec T\cdot\nabla
v)\Big(s_1\vec\nu+\Sigma_{j=1}^{N-1}s_{j+1}\vec a_j\Big),(m_L+\vec
T\cdot\nabla v)\Big(s_1\vec\nu+\Sigma_{j=1}^{N-1}s_{j+1}\vec
a_j\Big),\frac{s_1}{|s_1|}\bigg)ds:
\\ \quad v\in \mathcal{D}_1\big(\vec\nu,\vec a_1,\ldots,\vec
a_{N-1}\big)\Bigg\}\,,
\end{multline}
with
$\mathcal{D}_1(\cdot)$ and $\mathcal{D}_0(\cdot)$ defined by
Definition \ref{defkub}.
\end{lemma}
\begin{proof}
Since $\mathcal{D}_0\big(\vec\nu,\vec a_1,\ldots,\vec
a_{N-1}\big)\subset\mathcal{D}_1\big(\vec T,\vec\nu,\vec
a_1,\ldots,\vec a_{N-1}\big)$
we deduce
\begin{equation}\label{fgggnullreslslbner}
\inf\limits_{L\in(0,1)}\tilde P_L \leq\liminf\limits_{L\to
0^+}P_L\,.
\end{equation}

Next let $L\in(0,1)$ and $v\in \mathcal{D}_1\big(\vec T,\vec\nu,\vec
a_1,\ldots,\vec a_{N-1}\big)$ that we now fix. As before for every
positive integer $K\geq 2$ define
\begin{equation}\label{opreresfunkapjjj}
v_K(y):=\frac{1}{K}v(Ky)\in\mathcal{D}_1\big(\vec T,\vec\nu,\vec
a_1,\ldots,\vec a_{N-1}\big)\,.
\end{equation}
Then, changing variables in the integration and using the
periodicity condition in the definition of $\mathcal{D}_1(\cdot)$
together with the definition of $m_L$,
we obtain
\begin{multline}\label{esheodnonervdrugnnnvvvvuuuk}
\frac{K}{L}\int\limits_{I_{N}}G\bigg(\frac{L}{K}\nabla (m_{L/K}+\vec
T\cdot\nabla v_K)\Big(s_1\vec\nu+\Sigma_{j=1}^{N-1}s_{j+1}\vec
a_j\Big),(m_{L/K}+\vec T\cdot\nabla
v_K)\Big(s_1\vec\nu+\Sigma_{j=1}^{N-1}s_{j+1}\vec
a_j\Big),\frac{s_1}{|s_1|}\bigg)ds
\\=\frac{1}{L}\int\limits_{I_{N}}G\bigg(L\,\nabla(m_{L}+\vec T\cdot\nabla
v)\Big(s_1\vec\nu+\Sigma_{j=1}^{N-1}s_{j+1}\vec a_j\Big),\,
(m_{L}+\vec T\cdot\nabla
v)\Big(s_1\vec\nu+\Sigma_{j=1}^{N-1}s_{j+1}\vec a_j\Big),\,\frac
{s_1}{|s_1|}\bigg)\,ds\,.
\end{multline}
Next let $\zeta(t)\in C^\infty_c(\R,\R)$ satisfies
$\supp\zeta\subset[-1/2,1/2]$, $\zeta(t)=1$ if $t\in [-1/4,1/4]$ and
$0\leq\zeta(t)\leq 1$ for every $t\in\R$. For every integer $K\geq
2$ set
\begin{equation}\label{opreresfunkapjjjtild}
\tilde
v_K(y):=v_K(y)\zeta(y\cdot\vec\nu)\in\mathcal{D}_0\big(\vec\nu,\vec
a_1,\ldots,\vec a_{N-1}\big)\,.
\end{equation}
Then for $K\geq 2$, we have
\begin{equation}\label{gradopreresfunkapjjjtild}
\vec T\cdot\nabla \tilde v_K(y)=\vec T\cdot\bigg(\nabla
v_K(y)\zeta(y\cdot\vec\nu)+
\zeta'(y\cdot\vec\nu)v_K(y)\otimes\vec\nu\bigg)=\vec T\cdot\nabla
v(K y)+ \frac{1}{K}\zeta'(y\cdot\vec\nu)\vec
T\cdot\big\{v(Ky)\otimes\vec\nu\big\}\,,
\end{equation}
and furthermore, for the same case $K\geq 2$,
\begin{multline}\label{gradvtopreresfunkapjjjtild}
\nabla\{\vec T\cdot\nabla \tilde v_K\}(y)=\nabla\{\vec T\cdot\nabla
v_K\}(y)+ \zeta'(y\cdot\vec\nu)\nabla\Big(\vec T\cdot
\{v_K(y)\otimes\vec\nu\}\Big) +\zeta''(y\cdot\vec\nu)\Big(\vec
T\cdot\{v_K(y)\otimes\vec\nu\}\Big)\otimes\vec\nu
\\=K\nabla\{\vec T\cdot\nabla
v\}(Ky)+ \zeta'(y\cdot\vec\nu)\nabla\Big(\vec T\cdot
\{v\otimes\vec\nu\}\Big)(Ky)
+\frac{1}{K}\zeta''(y\cdot\vec\nu)\Big(\vec
T\cdot\{v(Ky)\otimes\vec\nu\}\Big)\otimes\vec\nu\,.
\end{multline}
Thus
\begin{multline}\label{gradvtopreresfunkapjjjtildkl}
(L/K)\nabla\{\vec T\cdot\nabla \tilde v_K\}(y)=\\ L\nabla\{\vec
T\cdot\nabla v\}(Ky)+\frac{L}{K}\zeta'(y\cdot\vec\nu)\nabla\big(\vec
T\cdot \{v\otimes\vec\nu\}\big)(Ky)+
\frac{L}{K^2}\zeta''(y\cdot\vec\nu)\Big(\vec
T\cdot\{v(Ky)\otimes\vec\nu\}\Big)\otimes\vec\nu\,.
\end{multline}
On the other hand by \er{odnomer} we have
\begin{equation}\label{odnomergradkkk}
\begin{split}
m_{L/K}(y)=\big(1-\theta(y\cdot\vec\nu\, K/L)\big)A_1
+\theta(y\cdot\vec\nu\, K/L)A_2\quad\forall y\in\R^N\,,\\
(L/K)\nabla m_{L/K}(y)=\theta'(y\cdot\vec\nu\,
K/L)\,(A_2-A_1)\otimes\vec\nu\quad\forall y\in\R^N\,.
\end{split}
\end{equation}
However, we have $\zeta,\zeta',\zeta'',\theta,\theta'\in L^\infty$
and $v,\nabla v,\nabla\{\vec T\cdot\nabla v\}\in L^\infty$.
Moreover, $G\big(0,A_r,(-1)^r\big)=0$ and since $G\geq 0$ also
$\nabla G\big(0,A_r,(-1)^r\big)=0$ for every $r\in\{1,2\}$.
Therefore, by
\er{odnomergradkkk}, \er{gradopreresfunkapjjjtild} and
\er{gradvtopreresfunkapjjjtildkl}
we have
\begin{multline}\label{esheodnonervdrugnnnvvvvuuuknew}
\frac{K}{L}\int\limits_{I_{N}}G\bigg(\frac{L}{K}\nabla (m_{L/K}+\vec
T\cdot\nabla\tilde v_K)\Big(s_1\vec\nu+\Sigma_{j=1}^{N-1}s_{j+1}\vec
a_j\Big),(m_{L/K}+\vec T\cdot\nabla\tilde
v_K)\Big(s_1\vec\nu+\Sigma_{j=1}^{N-1}s_{j+1}\vec
a_j\Big),\frac{s_1}{|s_1|}\bigg)ds\\
%
%
%
%
%
-\frac{K}{L}\int\limits_{I_{N}}G\bigg(\frac{L}{K}\nabla
(m_{L/K}+\vec T\cdot\nabla
v_K)\Big(s_1\vec\nu+\Sigma_{j=1}^{N-1}s_{j+1}\vec
a_j\Big),(m_{L/K}+\vec T\cdot\nabla
v_K)\Big(s_1\vec\nu+\Sigma_{j=1}^{N-1}s_{j+1}\vec
a_j\Big),\frac{s_1}{|s_1|}\bigg)ds
%
\\ \to 0\quad\quad\quad\quad\quad
\text{as}\quad K\to +\infty\,.
\quad\quad\quad\quad\quad\quad\quad\quad\quad\quad\quad\quad\quad
\quad\quad\quad\quad\quad\quad\quad\quad\quad\quad\quad\quad\quad
\end{multline}
Thus by
\er{esheodnonervdrugnnnvvvvuuuk} we have
\begin{multline}\label{esheodnonervdrugnnnvvvvuuuknewnnn}
\lim\limits_{K\to+\infty}\frac{K}{L}\int\limits_{I_{N}}\Bigg\{\\G\bigg(\frac{L}{K}\nabla
(m_{L/K}+\vec T\cdot\nabla\tilde
v_K)\Big(s_1\vec\nu+\Sigma_{j=1}^{N-1}s_{j+1}\vec
a_j\Big),(m_{L/K}+\vec T\cdot\nabla\tilde
v_K)\Big(s_1\vec\nu+\Sigma_{j=1}^{N-1}s_{j+1}\vec
a_j\Big),\frac{s_1}{|s_1|}\bigg)\Bigg\}ds\\
=\frac{1}{L}\int\limits_{I_{N}}G\bigg(L\nabla (m_{L}+\vec
T\cdot\nabla v)\Big(s_1\vec\nu+\Sigma_{j=1}^{N-1}s_{j+1}\vec
a_j\Big),(m_{L}+\vec T\cdot\nabla
v)\Big(s_1\vec\nu+\Sigma_{j=1}^{N-1}s_{j+1}\vec
a_j\Big),\frac{s_1}{|s_1|}\bigg)ds\,.
\end{multline}
However, since $\tilde v_K(y)\in\mathcal{D}_0\big(\vec\nu,\vec
a_1,\ldots,\vec a_{N-1}\big)$,
we obtain
\begin{multline}\label{fgggnullreslslbnernnbbb}
\lim\limits_{K\to +\infty}\Bigg\{\\
\frac{K}{L}\int\limits_{I_{N}}G\bigg(\frac{L}{K}\nabla (m_{L/K}+\vec
T\cdot\nabla\tilde v_K)\Big(s_1\vec\nu+\Sigma_{j=1}^{N-1}s_{j+1}\vec
a_j\Big),(m_{L/K}+\vec T\cdot\nabla\tilde
v_K)\Big(s_1\vec\nu+\Sigma_{j=1}^{N-1}s_{j+1}\vec
a_j\Big),\frac{s_1}{|s_1|}\bigg)ds
\\ \Bigg\}\geq \limsup\limits_{l\to 0^+}P_l
\end{multline}
Therefore, by \er{esheodnonervdrugnnnvvvvuuuknewnnn} and
\er{fgggnullreslslbnernnbbb} we deduce
\begin{multline}\label{fgggnullreslslbnernnbbblim}
\frac{1}{L}\int\limits_{I_{N}}G\bigg(L\nabla (m_{L}+ \vec
T\cdot\nabla v)\Big(s_1\vec\nu+\Sigma_{j=1}^{N-1}s_{j+1}\vec
a_j\Big),(m_{L}+ \vec T\cdot\nabla
v)\Big(s_1\vec\nu+\Sigma_{j=1}^{N-1}s_{j+1}\vec
a_j\Big),\frac{s_1}{|s_1|}\bigg)ds\geq \limsup\limits_{l\to 0^+}P_l
\end{multline}
Then since $L\in(0,1)$ and $v\in \mathcal{D}_1\big(\vec
T,\vec\nu,\vec a_1,\ldots,\vec a_{N-1}\big)$ were arbitrary by
\er{fgggnullreslslbnernnbbblim} we deduce
\begin{equation*}
\inf\limits_{L\in(0,1)}\tilde P_L\geq \limsup\limits_{L\to 0^+}P_L
\end{equation*}
%
%
%
%
%
%
%
%
This inequality together with the reverse inequality
\er{fgggnullreslslbner} gives equality \er{fgggnullresl}.
\end{proof}

\section{Upper bound construction}
\subsection{Primary approximating sequence}
We define a special class of mollifiers that we shall use in the
upper bound construction.
\begin{definition}\label{defper} The class $\mathcal V_0$ consists of
all functions $\eta\in C^\infty_c(\R^N,\R)$ such that $\eta$ is
radial, $\eta\geq 0$, $\supp\eta\subset \bar B_{1/2}(0)$ and
$\int_{\R^N}\eta(z)dz=1$.
\end{definition}
Next let  $\eta(z)\in \mathcal V_0$, $\vec A\in
\mathcal{L}(\R^{d\times N};\R^m)$, $\vec B\in
\mathcal{L}(\R^{d};\R^k)$ and $v\in\mathcal{D}'(\R^N,\R^d)$ such
that $\vec A\cdot\nabla v\in
BV(\R^N,\R^m)\cap L^\infty$ and $\vec B\cdot v\in
W^{1,1}(\R^N,\R^k)\cap L^\infty\cap Lip$. For any $\e>0$ and any
fixed $x\in\R^N$ set
\begin{equation}\label{a4543chejhgufydfdt}
\psi_\e(x):=\frac{1}{\e^N}\Big<\eta\Big(\frac{y-x}{\e}\Big),v(y)\Big>
\end{equation}
(see notations and definitions in the Appendix). Then $\psi_\e\in
C^\infty(\R^N,\R^d)$. Moreover clearly
\begin{equation}\label{a4543che}
\vec A\cdot \{\nabla
\psi_\e(x)\}=\frac{1}{\e^N}\int_{\R^N}\eta\Big(\frac{y-x}{\e}\Big)
\cdot\{\vec A\cdot\nabla v\}(y)dy= \int_{\R^N}\eta(z)\cdot\big\{\vec
A\cdot\nabla v\big\}(x+\e z)dz,
\end{equation}
\begin{equation}\label{a4543chekkk}
\vec
B\cdot\{\psi_\e(x)\}=\frac{1}{\e^N}\int_{\R^N}\eta\Big(\frac{y-x}{\e}\Big)\cdot
\{\vec B\cdot v\}(y)dy= \int_{\R^N}\eta(z)\cdot\{\vec B\cdot
v\}(x+\e z)dz.
\end{equation}
and since $\eta$ is radial, and therefore $\int_{\R^N}\eta(z)zdz=0$,
we have
\begin{multline}\label{a4543chekkkhhhjjkk}
\frac{1}{\e}\Big(\vec B\cdot\{\psi_\e(x)\}-\{\vec B\cdot
v\}(x)\Big)= \int_{\R^N}\eta(z)\frac{\{\vec B\cdot v\}(x+\e
z)-\{\vec B\cdot v\}(x)}{\e}dz=\\
\int_0^1\bigg(\int_{\R^N}\eta(z)z\cdot\nabla\{\vec B\cdot v\}(x+\e
tz)dz\bigg)dt=\int_0^1\bigg(\int_{\R^N}\eta(z)z\cdot\Big(\nabla\{\vec
B\cdot v\}(x+\e tz)-\nabla\{\vec B\cdot v\}(x)\Big)dz\bigg)dt.
\end{multline}

Then by the results of \cite{polgen} and \cite{polmag} we have
$\lim_{\e\to0^+} \vec A\cdot \nabla\psi_\e=\vec A\cdot \nabla v$ in
$L^{p}(\R^N,\R^m)$, $\lim_{\e\to0^+} \e\nabla\{\vec A\cdot
\nabla\psi_\e\}=0$ in $L^{p}(\R^N,\R^{m\times N})$,
$\lim_{\e\to0^+}\vec B\cdot \psi_\e=\vec B\cdot v$ in
$W^{1,p}(\R^N,\R^k)$ and $\lim_{\e\to0^+} (\vec B\cdot \psi_\e-\vec
B\cdot v)/\e=0$ in $L^{p}(\R^N,\R^k)$ for every $p\geq 1$. Moreover,
$\vec A\cdot\nabla \psi_\e$, $\e\nabla\{\vec A\cdot\nabla
\psi_\e\}$, $\vec B\cdot \psi_\e$ and $\nabla\{\vec B\cdot
\psi_\e\}$ are bounded in $L^\infty$,
\begin{equation}\label{bjdfgfghljjklkjkllllkkll}
\limsup\limits_{\e\to 0^+}\frac{1}{\e}\int_\O\Big|\vec A\cdot\nabla
\psi_\e(x)-\{\vec A\cdot \nabla v\}(x)\Big|\,dx<+\infty\,,
\end{equation}
and the following Theorem holds.
\begin{theorem}\label{vtporbound4}
Let $\Omega\subset\R^N$ be an open set.
Furthermore, let $\vec A\in \mathcal{L}(\R^{d\times N};\R^m)$, $\vec
B\in \mathcal{L}(\R^{d};\R^k)$ and let $F\in C^1(\R^{m\times
N}\times \R^m\times \R^k\times\R^q,\R)$, satisfying $F\geq 0$. Let
$f\in BV_{loc}(\R^N,\R^q)\cap L^\infty$ and
$v\in\mathcal{D}'(\R^N,\R^d)$ be such that $\vec A\cdot\nabla v\in
BV(\R^N,\R^m)\cap L^\infty(\R^N,\R^m)$ and $\vec B\cdot v\in
Lip(\R^N,\R^k)\cap W^{1,1}(\R^N,\R^k)\cap L^\infty(\R^N,\R^k)$,
$\|D(\vec A\cdot \nabla v)\|(\partial\Omega)=0$ and $F\big(0,\{\vec
A\cdot\nabla v\}(x),\{\vec B\cdot v\}(x),f(x)\big)=0$ a.e.~in
$\Omega$.
Consider $\psi_\e$, defined by \er{a4543chejhgufydfdt}. Then,
\begin{multline}\label{a1a2a3a4a5a6a7s8}
\lim_{\e\to 0}\frac{1}{\e}\int_\Omega F\Big(\,\e\nabla\big\{\vec
A\cdot\nabla\psi_\e(x)\big\},\, \vec A\cdot\nabla\psi_\e(x),\,\vec
B\cdot\psi_\e(x),\,f(x)\,\Big)dx=\int_{\O\cap J_{\vec A\cdot\nabla
v}}\Bigg\{\\ \int_\R F\bigg(p(t,x)\Big(\{\vec A\cdot\nabla
v\}^+(x)-\{\vec A\cdot \nabla
v\}^-(x)\Big)\otimes\vec\nu(x),\Gamma(t,x),\{\vec B\cdot
v\}(x),\frac{(|t|+t)f^-(x)+(|t|-t)f^+(x)}{2|t|}
\bigg)dt\\
\Bigg\}d \mathcal H^{N-1}(x),
\end{multline}
(with $\vec\nu(x)$ denoting the orientation vector of $J_{\vec
A\cdot\nabla v}$) where
\begin{equation}\label{matrixgamadef7845389yuj9}
\Gamma(t,x):=\bigg(\int_{-\infty}^t p(s,x)ds\bigg)\{\vec
A\cdot\nabla v\}^-(x)+\bigg(\int_t^{+\infty} p(s,x)ds\bigg)\{\vec
A\cdot\nabla v\}^+(x)\,,
\end{equation}
with $p(t,x)$ is defined by
\begin{equation}\label{defp}
p(t,x):=\int_{H^0_{\vec \nu(x)}}\eta(t\vec \nu(x)+y)\,d \mathcal
H^{N-1}(y)\,,
\end{equation}
and we assume that the
 orientation of $J_f$ coincides with the orientation of $J_{\vec A\cdot\nabla v}$ $\mathcal{H}^{N-1}$
a.e. on $J_f\cap J_{\vec A\cdot\nabla v}$.
\end{theorem}
\subsection{Modification of the primary sequence near the single elementary surface}
Next assume we are in the settings of Theorem \ref{vtporbound4}. Let
$S$ be $(N-1)$-dimensional hypersurface satisfying $\ov
S\subset\subset\O$. Moreover, assume that for some function
$g(x')\in C^1(\R^{N-1},\R)$ and a bounded open set
$\mathcal{U}\subset\R^{N-1}$ we have
\begin{equation} \label{L2009aprplmingg23128}
S=\{x=(x_1,x'):\,x'\in\mathcal{U},\;x_1=g(x')\}\,,
\end{equation}
and
\begin{equation} \label{L2009aprplmingg23128normal}
\vec n(x')=(1,-\nabla_{x'}g(x'))/\sqrt{1+|\nabla_{x'}g(x')|^2}\,,
\end{equation}
where $\vec n(x')$ is a normal vector to $S$ at the point
$\big(g(x'),x'\big)$.
Using Theorem \ref{petTh} we deduce that there exists a set
$D\subset S$ such that $\mathcal{H}^{N-1}(S\setminus D)=0$ and for
every $x\in D$ we have
\begin{multline}\label{L2009surfhh8128odno888} \lim\limits_{\rho\to
0^+}\frac{\int_{B_\rho^+(x,\vec n(x'))}\Big(\big|\{\vec A\cdot\nabla
v\}(y)-\{\vec A\cdot\nabla
v\}^+(x)\big|+\big|f(y)-f^+(x)\big|\Big)\,dy}
{\mathcal{L}^N\big(B_\rho(x)\big)}=0\,,\\ \lim\limits_{\rho\to
0^+}\frac{\int_{B_\rho^-(x,\vec n(x'))}\Big(\big|\{\vec A\cdot\nabla
v\}(y)-\{\vec A\cdot\nabla
v\}^-(x)\big|+\big|f(y)-f^-(x)\big|\Big)\,dy}
{\mathcal{L}^N\big(B_\rho(x)\big)}=0\,,
\end{multline}
where we use the convention $\nabla v^+(x)=\nabla v^-(x)=
\nabla\tilde{v}(x)$ and $f^+(x)=f^-(x)= \tilde{f}(x)$ at a point of
approximate continuity $x$.

Consider a radial function $\theta_0(z')=\kappa(|z'|)\in
C^\infty_c(\R^{N-1},\R)$ such that $\supp \theta_0\subset\subset
B_1(0)$, $\theta_0\geq 0$ and $\int_{\R^{N-1}}\theta_0(z')dz'=1$.
Then for any $\e>0$ define the function
$g_{\e}(x'):\R^{N-1}\to\R$ by
\begin{equation}
\label{L2009psije12387878}
g_{\e}(x'):=\frac{1}{\e^{N-1}}\int_{\R^{N-1}}\theta_0\Big(\frac{y'-x'}{\e}\Big)
g(y')dy'= \int_{\R^{N-1}}\theta_0(z')g(x'+\e z')\,dz',\quad\forall
x'\in\R^{N-1}\,.
\end{equation}
Since $\theta_0$ is radial, we obtain
$\int_{\R^{N-1}}\theta_0(z')z'\,dz'=0$. Therefore, since $g\in C^1$,
clearly we have
\begin{multline}
\label{L2009psijebn12387878}
\frac{1}{\e}\,\sup_{x'\in\mathcal{U}}\,|g_{\e}(x')-g(x')|\,\to
0\quad\text{as }\e\to 0^+\,,
\\\sup_{x'\in\mathcal{U}}\,|\nabla g_{\e}(x')-\nabla g(x')|\,\to
0\quad\text{as }\e\to 0^+\,,
\\
\sup_{x'\in\mathcal{U}}\,|\e\nabla^2g_{\e}(x')|\,\to 0\quad\text{as
}\e\to 0^+\,.
\end{multline}
Let $\vec e_1,\vec e_2,\ldots,\vec e_N$ be a standard orthonormal
base in $\R^N$ and let $h_0(y,x')$ be an arbitrary
$C^\infty(\R^N\times\R^{N-1},\R^d)$ function satisfying
\begin{equation}\label{obnulenie}
h_0(y,x')=0\;\text{ if }\;|y_1|\geq 1/2,\text{ and }\;h_0(y+\vec
e_j,x')=h_0(y,x')\;\;\forall j=2,\ldots, N\,,
\end{equation}
and
\begin{equation}\label{nositel}
\ov{\supp h_0(y,x')}\subset\R^N\times \mathcal{U}\,.
\end{equation}
Denote the set of such functions by $\mathcal{P}(\mathcal{U})$. Let
$L>0$ be an arbitrary number and let
\begin{equation}\label{nulltoone}
h(y,x'):=\frac{1}{L}h_0(Ly,x')\,.
\end{equation}
Then $h\in C^\infty(\R^N\times\R^{N-1},\R^d)$ satisfies
\begin{equation}\label{obnulenie1}
h(y,x')=0\;\text{ if }\;|y_1|\geq 1/(2L),\text{ and
}\;h\big(y+(1/L)\vec e_j,x'\big)=h(y,x')\;\;\forall j=2,\ldots, N\,,
\end{equation}
and
\begin{equation}\label{nositel1}
\ov{\supp h(y,x')}\subset\R^N\times \mathcal{U}\,.
\end{equation}
For any $\e>0$ define the function $u_{\e}(x)\in
C^\infty(\R^N,\R^d)$ by
\begin{equation}
\label{L2009test1288}
u_{\e}(x):=\psi_{\e}(x)+\e
h\bigg(\Big(\frac{x_1-g_{\e}(x')}{\e},\frac{x'}{\e}\Big),x'\bigg)\,,
\end{equation}
where $\psi_{\e}$ is defined  by \er{a4543che}.
\begin{lemma}\label{mnogohypsi88hkkhgkjgj}
Let $S$, $g$, $\mathcal{U}$, $\vec n$, $\theta_0$, $g_\e$,
$\mathcal{P}(\mathcal{U})$, $h_0\in\mathcal{P}(\mathcal{U})$, $L$,
$h$ and $u_\e$ be as above. Then
$\lim_{\e\to0^+}\vec A\cdot\nabla u_\e=\vec A\cdot \nabla v$ in
$L^p(\R^N,\R^m)$, $\lim_{\e\to0^+} \e\nabla\{\vec A\cdot \nabla
u_\e\}=0$ in $L^{p}(\R^N,\R^{m\times N})$, $\lim_{\e\to0^+}\vec
B\cdot u_\e=\vec B\cdot v$ in $W^{1,p}(\R^N,\R^k)$ and
$\lim_{\e\to0^+}(\vec B\cdot u_\e-\vec B\cdot v)/\e=0$ in
$L^{p}(\R^N,\R^k)$ for every $p\geq 1$. Moreover, $\vec A\cdot\nabla
u_\e$, $\e\nabla\{\vec A\cdot\nabla u_\e\}$, $\vec B\cdot u_\e$ and
$\nabla\{\vec B\cdot u_\e\}$ are bounded in $L^\infty$, and since
for small $\e>0$ we have $\int_\O \vec A\cdot\nabla u_\e
dx=\int_\O\vec A\cdot\nabla\psi_\e dx$ by
\er{bjdfgfghljjklkjkllllkkll}, we obtain
\begin{equation}\label{bjdfgfghljjklkjkllllkkllkkkkkkllllkkll}
\limsup\limits_{\e\to 0^+}\Bigg|\frac{1}{\e}\bigg(\int_\O\vec
A\cdot\nabla u_\e(x)\,dx-\int_\O\{\vec A\cdot \nabla
v\}(x)\,dx\bigg)\Bigg|<+\infty\,,
\end{equation}
\end{lemma}
\begin{proof}
Denote
\begin{equation}\label{L2009test1288kllhguitttttu}
r_{\e}(x):=\e
h\bigg(\Big(\frac{x_1-g_{\e}(x')}{\e},\frac{x'}{\e}\Big),x'\bigg)\,.
\end{equation}
It is clear that there exists $M>0$ such that
\begin{equation}\label{iugiuiuhuigyffffftyf}
\bigg|\frac{r_\e(x)}{\e}\bigg|+\Big|\nabla
r_\e(x)\Big|+\Big|\e\nabla^2 r_\e(x)\Big|\leq M\quad\forall
x\in\R^N,\;\,\forall\e>0.
\end{equation}
It is sufficient to prove that $\lim_{\e\to0^+} \e\nabla r^2_\e=0$
in $L^{p}(\R^N,\R^{d\times N\times N})$, $\lim_{\e\to0^+}\nabla
r_\e=0$ in $L^{p}(\R^N,\R^{d\times N})$ and
$\lim_{\e\to0^+}r_\e/\e=0$ in $L^{p}(\R^N,\R^d)$ for every $p\geq
1$. Indeed,
%
%
%
%
%
%
by the first equality in \er{obnulenie1} we have
$$r_{\e}(x)=\e h\bigg(\Big(\frac{x_1-g_{\e}(x')}{\e},\frac{x'}{\e}\Big),x'\bigg)=0\quad\text{if}\quad\big|x_1-g_{\e}(x')\big|>\frac{\e}{2L}.$$
Therefore,
\begin{equation}\label{gyugughhhgghjghjghj}
r_\e(x)=0,\;\, \nabla r_\e(x)=0\;\,\text{and}\;\, \nabla^2
r_\e(x)=0\quad\text{if}\quad\big|x_1-g_{\e}(x')\big|>\frac{\e}{2L}.
\end{equation}
Thus, by \er{nositel1}, \er{iugiuiuhuigyffffftyf} and
\er{gyugughhhgghjghjghj} we obtain that there exists a compact
$\widetilde K\subset\subset\R^{N-1}$, independent on $\e$, such that
\begin{multline*}\int\limits_{\R^N}\Bigg(\bigg|\frac{r_\e(x)}{\e}\bigg|^p+\Big|\nabla
r_\e(x)\Big|^p+\Big|\e\nabla^2 r_\e(x)\Big|^p\Bigg)dx=\\
\int\limits_{\widetilde
K}\int\limits_{g_{\e}(x')-\frac{\e}{2L}}^{g_{\e}(x')+\frac{\e}{2L}}\Bigg(\bigg|\frac{r_\e(x)}{\e}\bigg|^p+\Big|\nabla
r_\e(x)\Big|^p+\Big|\e\nabla^2 r_\e(x)\Big|^p\Bigg)dx_1dx' \leq
\frac{3M^p}{L}\,\e\mathcal{L}^{(N-1)}(\widetilde K)\to
0\quad\text{as}\quad\e\to 0^+.
\end{multline*}
\end{proof}
\begin{proposition}\label{mnogohypsi88}
Let $S$, $g$, $\mathcal{U}$, $\vec n$, $\theta_0$, $g_\e$,
$\mathcal{P}(\mathcal{U})$, $h_0\in\mathcal{P}(\mathcal{U})$, $L$,
$h$ and $u_\e$ be as above and let $v$ $F$, $f$, $\vec A$, $\vec B$,
$\psi_\e$, $p$ and $\Gamma$ be the same as in Theorem
\ref{vtporbound4}. Then
\begin{multline}
\label{L2009limew03zeta71288888Contggiuuggyyyynew88789999vprop78899shtrih}
\lim\limits_{\e\to 0}\int\limits_{\O}\frac{1}{\e}
\bigg\{F\Big(\e\nabla\{\vec A\cdot\nabla \psi_{\e}\},\,\vec
A\cdot\nabla \psi_{\e},\,\vec
B\cdot\psi_{\e},\,f\Big)-F\Big(\e\nabla\{\vec A\cdot \nabla
u_{\e}\},\,\vec A\cdot\nabla u_{\e},\,\vec B\cdot
u_{\e},\,f\Big)\bigg\}\,dx=
\int\limits_{S}\int\limits_{\R}\int\limits_{I_1^{N-1}}\Bigg\{\\
\frac{1}{L}F\bigg(p(-s_1/L,x) \Big\{(\vec A\cdot\nabla v)^+(x)-(\vec
A\cdot\nabla v)^-(x)\Big\}\otimes\vec
n(x'),\,\Gamma(-s_1/L,x),\,\vec B\cdot
v(x),\,\zeta\big(s_1,f^+(x),f^-(x)\big)\bigg)
\\-\frac{1}{L}F\bigg(p(-s_1/L,x)\Big\{(\vec A\cdot\nabla v)^+(x)-(\vec A\cdot\nabla
v)^-(x)\Big\}\otimes\vec n(x')+L\nabla_y\{\vec A\cdot\nabla_y\tilde h\}\big(Q_{x'}(s),x'\big),\\
\,\Gamma(-s_1/L,x)+\vec A\cdot\nabla_y\tilde
h\big(Q_{x'}(s),x'\big),\,\vec B\cdot
v(x),\,\zeta\big(s_1,f^+(x),f^-(x)\big) \bigg)
\Bigg\}\,ds'ds_1\,d\mathcal{H}^{N-1}(x)\,,
\end{multline}
where $(s_1,s'):=s\in\R\times\R^{N-1}$,
$$I_1^{N-1}:=\Big\{s'=(s_2,\ldots s_N)\in\R^{N-1}:-1/2\leq s_j\leq
1/2\;\;\text{if}\;\;2\leq j\leq N\Big\}\,$$ $\tilde h(y,x')\in
C^\infty(\R^N\times\R^{N-1},\R^d)$ is given by
\begin{equation}\label{tildedefttshtrih}
\tilde h(y,x'):=h_0\Big(\big\{y_1-\nabla_{x'}g(x')\cdot
y',y'\big\},x'\Big)\,,
\end{equation}
\begin{equation}\label{defpgytyut}
\zeta(t,a,b):=\frac{(|t|+t)a+(|t|-t)b}{2|t|}=\begin{cases} a\quad\text{if}\;\;t>0 \,,\\
b\quad\text{if}\;\;t<0\,,
\end{cases}
\end{equation}
and the linear transformation $Q_{x'}(s)=Q_{x'}(s_1,s_2,\ldots,
s_N):\R^N\to\R^N$ is defined by
\begin{equation}\label{chfhh2009y888shtrih}
Q_{x'}(s):=\sum\limits_{j=1}^{N}s_j\,\vec q_j(x')\,,
\end{equation}
with
\begin{equation}\label{tildedefteeeshtrih}
\vec q_j(x'):=\begin{cases}\vec
n(x')\quad\quad\quad\quad\text{if}\quad j=1\,,\\
\big(\nabla_{x}g(x')\cdot \vec e_j\big)\,\vec e_1+\vec
e_j\quad\quad\text{if}\quad 2\leq j\leq N\,.\end{cases}\,.
\end{equation}
\end{proposition}
\begin{proof}
Observe that
\begin{multline}
\label{L2009testgrad12888} \nabla_{x}
u_{\e}(x)=\nabla_{x}\psi_{\e}(x)+\nabla_{y}
h\bigg(\Big(\frac{x_1-g_{\e}(x')}{\e},\frac{x'}{\e}\Big),x'\bigg)\\-\partial_{y_1}h\bigg(\Big(\frac{x_1-g_{\e}(x')}{\e},\frac{x'}{\e}\Big),x'\bigg)
\otimes\nabla_{x}g_{\e}(x')+\e\nabla_{2}
h\bigg(\Big(\frac{x_1-g_{\e}(x')}{\e},\frac{x'}{\e}\Big),x'\bigg)\,,
\end{multline}
where $\nabla_2h(y,x'):=\big(0,\nabla_{x'}h(y,x')\big)$ (the
$\R^N$-gradient by the second variable) and
\begin{multline}
\label{L2009testgradgrad212888} \e\nabla^2_{x}
u_{\e}(x)=\e\nabla^2_{x}\psi_{\e}(x)+\nabla^2_{y}
h\bigg(\Big(\frac{x_1-g_{\e}(x')}{\e},\frac{x'}{\e}\Big),x'\bigg)-\nabla_y\partial_{y_1}h\bigg(\Big(\frac{x_1-g_{\e}(x')}{\e},\frac{x'}{\e}\Big),x'\bigg)
\otimes\nabla_{x}g_{\e}(x')\\-\nabla_{x}g_{\e}(x')\otimes\nabla_y\partial_{y_1}h\bigg(\Big(\frac{x_1-g_{\e}(x')}{\e},\frac{x'}{\e}\Big),x'\bigg)
+\partial^2_{y_1y_1}h\bigg(\Big(\frac{x_1-g_{\e}(x')}{\e},\frac{x'}{\e}\Big),x'\bigg)
\otimes\nabla_{x}g_{\e}(x')\otimes\nabla_{x}g_{\e}(x')\\-\e\partial_{y_1}h\bigg(\Big(\frac{x_1-g_{\e}(x')}{\e},\frac{x'}{\e}\Big),x'\bigg)
\otimes\nabla^2_{x}g_{\e}(x')+ \e\nabla_2\nabla_{y}
h\bigg(\Big(\frac{x_1-g_{\e}(x')}{\e},\frac{x'}{\e}\Big),x'\bigg)\\+\e\nabla_{y}\nabla_{2}
h\bigg(\Big(\frac{x_1-g_{\e}(x')}{\e},\frac{x'}{\e}\Big),x'\bigg)-\e\nabla_2\partial_{y_1}h\bigg(\Big(\frac{x_1-g_{\e}(x')}{\e},\frac{x'}{\e}\Big),x'\bigg)
\otimes\nabla_{x}g_{\e}(x')\\-\e\nabla_{x}g_{\e}(x')\otimes\nabla_2\partial_{y_1}h\bigg(\Big(\frac{x_1-g_{\e}(x')}{\e},\frac{x'}{\e}\Big),x'\bigg)
+\e^2\nabla^2_{2}
h\bigg(\Big(\frac{x_1-g_{\e}(x')}{\e},\frac{x'}{\e}\Big),x'\bigg)\,.
\end{multline}
Then by \er{L2009test1288}, \er{L2009testgrad12888} and
\er{L2009testgradgrad212888}, and by \er{a4543che} and
\er{a4543chekkk} we have
\begin{equation}
\label{L2009testgradprod128888ggjkgjkgjk} \vec B\cdot
u_{\e}(x)=\int_{\R^N}\eta(z)\{\vec B\cdot v\}(x+\e z)\,dz
+\e\vec B\cdot
h\bigg(\Big(\frac{x_1-g_{\e}(x')}{\e},\frac{x'}{\e}\Big),x'\bigg)\,,
\end{equation}
\begin{multline}
\label{L2009testgradprod128888} \vec A\cdot\nabla_{x}
u_{\e}(x)=\int_{\R^N}\eta(z)\{\vec A\cdot\nabla v\}(x+\e z)\,dz
+\vec A\cdot\nabla_{y}
h\bigg(\Big(\frac{x_1-g_{\e}(x')}{\e},\frac{x'}{\e}\Big),x'\bigg)\\
-\vec
A\cdot\bigg\{\partial_{y_1}h\bigg(\Big(\frac{x_1-g_{\e}(x')}{\e},\frac{x'}{\e}\Big),x'\bigg)
\otimes\nabla_{x}g_{\e}(x')\bigg\}+\e\vec A\cdot\nabla_{2}
h\bigg(\Big(\frac{x_1-g_{\e}(x')}{\e},\frac{x'}{\e}\Big),x'\bigg)\,,
\end{multline}
and
\begin{multline}
\label{L2009testgradgradprod2128888} \e\nabla_{x}\{\vec
A\cdot\nabla_{x} u_{\e}\}(x)=-\int_{\R^N}\{\vec A\cdot\nabla
v\}(x+\e z)\otimes\nabla_z\eta(z)\,dz+
\nabla_{y}\{\vec A\cdot\nabla_{y}
h\}\bigg(\Big(\frac{x_1-g_{\e}(x')}{\e},\frac{x'}{\e}\Big),x'\bigg)\\-\{\vec
A\cdot\nabla_y\partial_{y_1}h\}\bigg(\Big(\frac{x_1-g_{\e}(x')}{\e},\frac{x'}{\e}\Big),x'\bigg)
\otimes\nabla_{x}g_{\e}(x') -\vec
A\cdot\bigg\{\nabla_{x}g_{\e}(x')\otimes\nabla_y\partial_{y_1}h\bigg(\Big(\frac{x_1-g_{\e}(x')}{\e},\frac{x'}{\e}\Big),x'\bigg)
\bigg\}
\\+ \vec
A\cdot\bigg\{\partial^2_{y_1y_1}h\bigg(\Big(\frac{x_1-g_{\e}(x')}{\e},\frac{x'}{\e}\Big),x'\bigg)
\otimes\nabla_{x}g_{\e}(x')\bigg\}\otimes\nabla_{x}g_{\e}(x')\\-\e\vec
A\cdot\bigg\{\partial_{y_1}h\bigg(\Big(\frac{x_1-g_{\e}(x')}{\e},\frac{x'}{\e}\Big),x'\bigg)
\otimes\nabla^2_{x}g_{\e}(x')\bigg\}+\e\nabla_{y}\{\vec
A\cdot\nabla_2
h\}\bigg(\Big(\frac{x_1-g_{\e}(x')}{\e},\frac{x'}{\e}\Big),x'\bigg)\\+\e\nabla_{2}\{\vec
A\cdot\nabla_y
h\}\bigg(\Big(\frac{x_1-g_{\e}(x')}{\e},\frac{x'}{\e}\Big),x'\bigg)-\e\{\vec
A\cdot\nabla_2\partial_{y_1}h\}\bigg(\Big(\frac{x_1-g_{\e}(x')}{\e},\frac{x'}{\e}\Big),x'\bigg)
\otimes\nabla_{x}g_{\e}(x')\\-\e \vec A\cdot\bigg\{
\nabla_{x}g_{\e}(x')\otimes\nabla_2\partial_{y_1}h\bigg(\Big(\frac{x_1-g_{\e}(x')}{\e},\frac{x'}{\e}\Big),x'\bigg)\bigg\}
+\e^2\nabla_{2}\{\vec A\cdot\nabla_{2}
h\}\bigg(\Big(\frac{x_1-g_{\e}(x')}{\e},\frac{x'}{\e}\Big),x'\bigg)\,,
\end{multline}
where we denote
\begin{equation*}
\vec
A\cdot\bigg\{
\sigma\otimes\nabla^2_{x}g_{\e}(x')\bigg\}:= \Bigg\{\Bigg(\vec
A\cdot\bigg\{
\sigma\otimes\nabla_{x}\partial_{x_j}g_{\e}(x')\bigg\}\Bigg)_i\Bigg\}_{1\leq
i\leq m,\,1\leq j\leq N}\quad\quad\forall\sigma\in\R^d\,,
\end{equation*}
and
\begin{equation*}
\vec A\cdot\bigg\{\nabla_{x}g_{\e}(x')\otimes\nabla\varpi
\bigg\}:= \Bigg\{\Bigg(\vec A\cdot\bigg\{\partial_j\varpi
\otimes\nabla_{x}g_{\e}(x')\bigg\}\Bigg)_i\Bigg\}_{1\leq i\leq
m,\,1\leq j\leq N}\quad\quad\forall\varpi:\R^N\to\R^d\,.
\end{equation*}
Note also that
\begin{equation}
\label{L2009difflll212888} \big\{x\in\O:u_\e(x)\neq
\psi_\e(x)\,\big\}\subset\{x:\,x'\in
\mathcal{U}\,\;|x_1-g_{\e}(x')|<\e/(2L)\}\,.
\end{equation}
%
%
%
%
%
Thus
\begin{multline}
\label{L2009limew03zeta71288888} \int\limits_{\O}\frac{1}{\e}
\bigg\{F\Big(\e\nabla \{\vec A\cdot\nabla\psi_{\e}\},\,\vec
A\cdot\nabla \psi_{\e},\,\vec
B\cdot\psi_{\e},\,f\Big)-F\Big(\e\nabla\{\vec A\cdot\nabla
u_{\e}\},\,\vec A\cdot\nabla
u_{\e},\,\vec B\cdot u_{\e},\,f\Big)\bigg\}\,dx=\\
\int\limits_{\mathcal{U}}\int\limits_{g_{\e}(x')-\e/(2L)}^{g_{\e}(x')+\e/(2L)}\frac{1}{\e}
\bigg\{F\Big(\e\nabla \{\vec A\cdot\nabla\psi_{\e}\},\,\vec
A\cdot\nabla \psi_{\e},\,\vec
B\cdot\psi_{\e},\,f\Big)-F\Big(\e\nabla\{\vec A\cdot\nabla
u_{\e}\},\,\vec A\cdot\nabla u_{\e},\,\vec B\cdot
u_{\e},\,f\Big)\bigg\}\,dx_1dx'\,.
\end{multline}
Then changing variables gives
\begin{multline}
\label{L2009limew03zeta71288888Cont} \int\limits_{\O}\frac{1}{\e}
\bigg\{F\Big(\e\nabla \{\vec A\cdot\nabla\psi_{\e}\},\,\vec
A\cdot\nabla \psi_{\e},\,\vec
B\cdot\psi_{\e},\,f\Big)-F\Big(\e\nabla\{\vec A\cdot\nabla
u_{\e}\},\,\vec A\cdot\nabla
u_{\e},\,\vec B\cdot u_{\e},\,f\Big)\bigg\}\,dx=\\
\int\limits_{\mathcal{U}}\int\limits_{-\e/(2L)}^{\e/(2L)}\frac{1}{\e}\Bigg\{
F\bigg(\e\nabla \{\vec A\cdot\nabla\psi_{\e}\}\big(x+g_{\e}(x')\vec
e_1\big),\,\vec A\cdot\nabla \psi_{\e}\big(x+g_{\e}(x')\vec
e_1\big),\,\vec B\cdot\psi_{\e}\big(x+g_{\e}(x')\vec
e_1\big),\,f\big(x+g_{\e}(x')\vec
e_1\big)\bigg)\\-F\bigg(\nabla\{\vec A\cdot\nabla
u_{\e}\}\big(x+g_{\e}(x')\vec e_1\big),\,\vec A\cdot\nabla
u_{\e}(x+g_{\e}\big(x')\vec e_1\big),\,\vec B\cdot
u_{\e}\big(x+g_{\e}(x')\vec e_1\big),\,f\big(x+g_{\e}(x')\vec
e_1\big)\bigg)\Bigg\}\,dx_1dx'\,.
\end{multline}
We also observe that
for any small $\e>0$ we have
\begin{multline}
\label{L2009testshift1288} \vec B\cdot\psi_{\e}\big(x+g_{\e}(x')\vec
e_1\big)=
\int_{\R^N}\eta(z)\{\vec B\cdot v\}\big(x+g_{\e}(x')\vec e_1+\e z\big)\,dz=\\
\int_{\R^N}\eta\Big(z_1-x_1/\e+\big(g(x')-g_\e(x')\big)/\e,z'\Big)
\{\vec B\cdot v\}\big(g(x')+\e z_1,x'+\e z'\big)\,dz\,,
\end{multline}
and then
\begin{multline}
\label{L2009testshift1288m} \vec B\cdot u_{\e}\big(x+g_{\e}(x')\vec
e_1\big)=\vec B\cdot\psi_{\e}(x+g_{\e}(x')\vec e_1\big)+\e\,\vec B\cdot h(x/\e,x')=\\
\int_{\R^N}\eta\Big(z_1-x_1/\e+\big(g(x')-g_\e(x')\big)/\e,z'\Big)
\{\vec B\cdot v\}\big(g(x')+\e z_1,x'+\e z'\big)\,dz+\e\,\vec B\cdot
h(x/\e,x')\,.
\end{multline}
Moreover,
\begin{multline}
\label{L2009testgradshift12888} \vec A\cdot\nabla
\psi_{\e}\big(x+g_{\e}(x')\vec e_1\big)= \int_{\R^N}\eta(z)\{\vec
A\cdot\nabla v\}\big(x+g_{\e}(x')\vec e_1+\e z\big)\,dz
\\=\int_{\R^N}\eta\Big(z_1-x_1/\e+\big(g(x')-g_\e(x')\big)/\e,z'\Big)\{\vec A\cdot\nabla
v\}\big(g(x')+\e z_1,x'+\e z'\big)\,dz\,,
\end{multline}
and then by
\er{L2009testgradprod128888} we infer
\begin{multline}
\label{L2009testgradshift12888m} \vec A\cdot\nabla
u_{\e}\big(x+g_{\e}(x')\vec e_1\big)=
\int_{\R^N}\eta\Big(z_1-x_1/\e+\big(g(x')-g_\e(x')\big)/\e,z'\Big)\{\vec
A\cdot\nabla v\}\big(g(x')+\e z_1,x'+\e
z'\big)\,dz\\
+\vec A\cdot\nabla_y h(x/\e,x')-\vec
A\cdot\Big\{\partial_{y_1}h(x/\e,x')
\otimes\nabla_{x}g_{\e}(x')\Big\}+\e\vec A\cdot\nabla_2
h(x/\e,x')\,.
\end{multline}
Finally
\begin{multline}
\label{L2009testgradgradshift12888} \e\nabla\{\vec A\cdot\nabla
\psi_{\e}\}\big(x+g_{\e}(x')\vec e_1\big)=-\int_{\R^N}\{\vec
A\cdot\nabla v\}\big(x+g_{\e}(x')\vec e_1+\e
z\big)\otimes\nabla\eta(z)\,dz
=\\-\int_{\R^N} \{\vec A\cdot\nabla v\}\big(g(x')+\e z_1,x'+\e
z'\big)\otimes\nabla\eta\Big(z_1-x_1/\e+\big(g(x')-g_\e(x')\big)/\e,z'\Big)\,dz\,,
\end{multline}
and then by \er{L2009testgradgradprod2128888} we have
\begin{multline}
\label{L2009testgradgradshift12888m} \e\nabla\{\vec A\cdot\nabla
u_{\e}\}\big(x+g_{\e}(x')\vec
e_1\big)=\\-\int_{\R^N}\{\vec
A\cdot\nabla v\}\big(g(x')+\e z_1,x'+\e z'\big)\otimes\nabla\eta\Big(z_1-x_1/\e+\big(g(x')-g_\e(x')\big)/\e,z'\Big)\,dz
\\+\nabla_y\{\vec A\cdot\nabla_y
h\}(x/\e,x')-\{\vec A\cdot\nabla_y\partial_{y_1}h\}(x/\e,x')
\otimes\nabla_{x}g_{\e}(x')-\vec
A\cdot\Big\{\nabla_{x}g_{\e}(x')\otimes\nabla_y\partial_{y_1}h(x/\e,x')
\Big\}\\+\vec A\cdot\Big\{\partial^2_{y_1y_1}h(x/\e,x')
\otimes\nabla_{x}g_{\e}(x')\Big\}\otimes\nabla_{x}g_{\e}(x')-\vec
A\cdot\Big\{\partial_{y_1}h(x/\e,x')
\otimes\{\e\nabla^2_{x}g_{\e}(x')\}\Big\}\\+\e\nabla_{y}\{\vec
A\cdot\nabla_2 h\}(x/\e,x')+\e\nabla_{2}\{\vec A\cdot\nabla_{y}
h\}(x/\e,x')-\e\{\vec A\cdot\nabla_2\partial_{y_1}h\}(x/\e,x')
\otimes\nabla_{x}g_{\e}(x')\\-\e\vec
A\cdot\Big\{\nabla_{x}g_{\e}(x')\otimes\nabla_2\partial_{y_1}h(x/\e,x')\Big\}
+\e^2\nabla_{2}\{\vec A\cdot\nabla_{2} h\}(x/\e,x')\,.
\end{multline}
Set
\begin{multline}
\label{L2009testgradshiftlim12888kp} \delta(x_1,x'):=
\bigg(\int\limits_{H_+\big((g(x'),x'),\vec
n(x')\big)}\eta\big(z_1-x_1,z'\big)dz\bigg) \{\vec A\cdot\nabla
v\}^+\big(g(x'),x'\big)\\+\bigg(\int\limits_{H_-\big((g(x'),x'),\vec
n(x')\big)}\eta\big(z_1-x_1,z'\big)dz\bigg) \{\vec A\cdot\nabla
v\}^-\big(g(x'),x'\big)
\,,
\end{multline}
and
\begin{multline}
\label{L2009testgradgradshiftlim12888kp} \theta(x_1,x'):=-\{\vec
A\cdot\nabla v\}^+\big(g(x'),x'\big)\otimes
\bigg(\int\limits_{H_+\big((g(x'),x'),\vec
n(x')\big)}\nabla\eta\big(z_1-x_1,z'\big)dz\bigg)\\-\{\vec
A\cdot\nabla
v\}^-\big(g(x'),x'\big)\otimes\bigg(\int\limits_{H_-\big((g(x'),x'),\vec
n(x')\big)}\nabla_z\eta\big(z_1-x_1,z'\big)dz\bigg)
\,,
\end{multline}
%
%
%
%
%
%
where $H_+(x,\vec n)=\{y\in\R^N:(y-x)\cdot\vec n>0\}$ and
$H_-(x,\vec n)=\{y\in\R^N:(y-x)\cdot\vec n<0\}$,
and define
\begin{equation}
\label{L2009testgradshiftaddlim12888}
\Lambda_\e(x):=\delta(x_1/\e,x')+
\vec A\cdot\nabla_y h(x/\e,x')-\vec
A\cdot\Big\{\partial_{y_1}h(x/\e,x') \otimes\nabla_{x}g(x')\Big\}
\,,
\end{equation}
and
\begin{multline}
\label{L2009testgradgradshiftaddlim12888} \Theta_\e(x):=
\theta(x_1/\e,x')+ \nabla_y\{\vec A\cdot\nabla_y h(x/\e,x')\}-\{\vec
A\cdot\nabla_y\partial_{y_1}h\}(x/\e,x')\otimes\nabla_{x}g(x')
\\-\vec
A\cdot\Big\{\nabla_{x}g(x')\otimes\nabla_y\partial_{y_1}h(x/\e,x')\Big\}
+\vec
A\cdot\Big\{\partial^2_{y_1y_1}h(x/\e,x')
\otimes\nabla_{x}g(x')\Big\}\otimes\nabla_{x}g(x')
\,.
\end{multline}
Then by the fact that $\{\vec B\cdot v\}\in Lip\,$ and by
\er{L2009surfhh8128odno888}, \er{L2009psijebn12387878},
\er{L2009testshift1288}, \er{L2009testgradshift12888} and
\er{L2009testgradgradshift12888} we deduce
\begin{multline}
\label{L2009limew03zetalim71288888Cont}
\int\limits_{\mathcal{U}}\;\int\limits_{-\e/(2L)}^{\e/(2L)}\frac{1}{\e}\bigg\{\Big|
\e\nabla\{\vec A\cdot\nabla \psi_{\e}\}\big(x+g_{\e}(x')\vec
e_1\big)-\theta(x_1/\e,x')\Big|+\Big|\vec A\cdot\nabla
\psi_{\e}\big(x+g_{\e}(x')\vec
e_1\big)-\delta(x_1/\e,x')\Big|\\+\Big|\vec
B\cdot\psi_{\e}\big(x+g_{\e}(x')\vec e_1\big)-\{\vec B\cdot
v\}\big(g(x'),x'\big)\Big|\bigg\}\,dx_1dx'\to 0\quad\text{as }\e\to
0\,,
\end{multline}
and by the fact that $\{\vec B\cdot v\}\in Lip\,$ and by
\er{L2009surfhh8128odno888}, \er{L2009psijebn12387878},
\er{L2009testshift1288m}, \er{L2009testgradshift12888m} and
\er{L2009testgradgradshift12888m} we deduce
\begin{multline}
\label{L2009limew03zetalim71288888Contu}
\int\limits_{\mathcal{U}}\;\int\limits_{-\e/(2L)}^{\e/(2L)}\frac{1}{\e}\bigg\{\Big|
\e\nabla\{\vec A\cdot\nabla u_{\e}\}\big(x+g_{\e}(x')\vec
e_1\big)-\Theta_\e(x)\Big|+\Big|\vec A\cdot\nabla
u_{\e}\big(x+g_{\e}(x')\vec e_1\big)-\Lambda_\e(x)\Big|\\+\Big|\vec
B\cdot u_{\e}\big(x+g_{\e}(x')\vec e_1\big)-\{\vec B\cdot
v\}\big(g(x'),x'\big)\Big|\bigg\}\,dx_1dx'\to 0\quad\text{as }\e\to
0\,,
\end{multline}
Therefore, by \er{L2009limew03zeta71288888Cont},
\er{L2009limew03zetalim71288888Cont} and
\er{L2009limew03zetalim71288888Contu}
we infer
\begin{multline}
\label{L2009limew03zeta71288888Contprod12888}
\int\limits_{\O}\frac{1}{\e} \bigg\{F\Big(\e\nabla\{\vec
A\cdot\nabla \psi_{\e}\},\,\vec A\cdot\nabla \psi_{\e},\,\vec
B\cdot\psi_{\e},\,f\Big)-F\Big(\e\nabla\{\vec A\cdot\nabla
u_{\e}\},\,\vec
A\cdot\nabla u_{\e},\,\vec
B\cdot u_{\e},\,f\Big)\bigg\}\,dx=\\
\int\limits_{\mathcal{U}}\;\int\limits_{-\e/(2L)}^{\e/(2L)}\frac{1}{\e}\bigg\{
F\Big(\theta(x_1/\e,x'),\,\delta(x_1/\e,x'),\,\{\vec B\cdot
v\}\big(g(x'),x'\big),\,f\big(x+g_{\e}(x')\vec
e_1\big)\Big)\\-F\Big(\Theta_\e(x),\,\Lambda_\e(x),\,\{\vec B\cdot
v\}\big(g(x'),x'\big),\,f\big(x+g_{\e}(x')\vec
e_1\big)\Big)\bigg\}\,dx_1dx'+l_\e
\,,
\end{multline}
where $\lim_{\e\to 0}l_\e=0$. Next by Theorem 3.108 and Remark 3.109
from \cite{amb} we deduce that
\begin{equation}
\label{shrppppp8}
\lim\limits_{\rho\to
0^+}\frac{1}{\rho}\int_{-\rho}^\rho\bigg|f\big(g(x')+s,x'\big)
-\zeta\Big(s,f^+\big(g(x'),x'\big),f^-\big(g(x'),x'\big)\Big)\bigg|\,ds=0\quad\quad\text{for
}\mathcal{L}^{N-1}\text{ a.e. }x'\in \mathcal{U}\,,
\end{equation}
where $\zeta(t,a,b)$ is defined by \er{defpgytyut}.
Then, since $f\in L^\infty$, by \er{shrppppp8} and
\er{L2009psijebn12387878} we obtain
\begin{equation}
\label{L2009limew03zetalim71288888shrrrrrpCont8}
\int\limits_{\mathcal{U}}\;\int\limits_{-\e/(2L)}^{\e/(2L)}\frac{1}{\e}\bigg|f\big(x+g_{\e}(x')\vec
e_1\big)-\zeta\Big(x_1,f^+\big(g(x'),x'\big),f^-\big(g(x'),x'\big)\Big)\bigg|\,dx_1dx'\to
0\quad\text{as
}\e\to 0\,.
\end{equation}
Thus, using \er{L2009limew03zetalim71288888shrrrrrpCont8}, by
\er{L2009limew03zeta71288888Contprod12888} we obtain
\begin{multline}
\label{L2009limew03zeta71288888Contprod12888shrrrp8}
\int\limits_{\O}\frac{1}{\e} \bigg\{F\Big(\e\nabla\{\vec
A\cdot\nabla \psi_{\e}\},\,\vec A\cdot\nabla \psi_{\e},\,\vec
B\cdot\psi_{\e},\,f\Big)-F\Big(\e\nabla\{\vec A\cdot\nabla
u_{\e}\},\,\vec A\cdot\nabla u_{\e},\,\vec
B\cdot u_{\e},\,f\Big)\bigg\}\,dx=\\
\int\limits_{\mathcal{U}}\;\int\limits_{-\e/(2L)}^{\e/(2L)}\frac{1}{\e}\Bigg\{
F\bigg(\theta(x_1/\e,x'),\,\delta(x_1/\e,x'),\,\{\vec B\cdot
v\}\big(g(x'),x'\big),\,\zeta\Big(x_1,f^+\big(g(x'),x'\big),f^-\big(g(x'),x'\big)\Big)
\bigg)\\-F\bigg(\Theta_\e(x),\,\Lambda_\e(x),\,\{\vec B\cdot
v\}\big(g(x'),x'\big),\,\zeta\Big(x_1,f^+\big(g(x'),x'\big),f^-\big(g(x'),x'\big)\Big)
\bigg)\Bigg\}\,dx_1dx'+\bar l_\e=\\
\int\limits_{\mathcal{K}}\;\int\limits_{-\e/(2L)}^{\e/(2L)}\frac{1}{\e}\Bigg\{
F\bigg(\theta(x_1/\e,x'),\,\delta(x_1/\e,x'),\,\{\vec B\cdot
v\}\big(g(x'),x'\big),\,\zeta\Big(x_1,f^+\big(g(x'),x'\big),f^-\big(g(x'),x'\big)\Big)
\bigg)\\-F\bigg(\Theta_\e(x),\,\Lambda_\e(x),\,\{\vec B\cdot
v\}\big(g(x'),x'\big),\,\zeta\Big(x_1,f^+\big(g(x'),x'\big),f^-\big(g(x'),x'\big)\Big)
\bigg)\Bigg\}\,dx_1dx'+\bar l_\e\,,
\end{multline}
where $\lim_{\e\to 0}\bar l_\e=0$ and
$\mathcal{K}\subset\subset\mathcal{U}$ is a compact set, such that
$$\supp h(y,x')\subset\R^N\times \mathcal{K}\,.$$
Next for every $y'\in\R^{N-1}$ set
\begin{multline}
\label{L2009testgradshiftaddlim12888up}
\bar\Lambda_\e(x,y'):=\delta(x_1/\e,x')+\vec A\cdot \nabla_1
h\Big(\{x_1/\e,x'/\e+y'\},x'\Big)-\vec
A\cdot\Big\{\partial_{y_1}h\Big(\{x_1/\e,x'/\e+y'\},x'\Big)
\otimes\nabla_{x}g(x')\Big\}
\,,
\end{multline}
and
\begin{multline}
\label{L2009testgradgradshiftaddlim12888upr}
\bar\Theta_\e(x,y'):=\theta(x_1/\e,x')+\nabla_1\{\vec A\cdot\nabla_1
h\}\Big(\{x_1/\e,x'/\e+y'\},x'\Big)-\{\vec
A\cdot\nabla_1\partial_{y_1}h\}\Big(\{x_1/\e,x'/\e+y'\},x'\Big)\otimes\nabla_{x}g(x')\\-\vec
A\cdot\Big\{\nabla_{x}g(x')\otimes\nabla_1\partial_{y_1}h\Big(\{x_1/\e,x'/\e+y'\},x'\Big)\Big\}
+\vec
A\cdot\Big\{\partial^2_{y_1y_1}h\Big(\{x_1/\e,x'/\e+y'\},x'\Big)
\otimes\nabla_{x}g(x')\Big\}\otimes\nabla_{x}g(x')
\,,
\end{multline}
where we denote $\nabla_1 h(y,x'):=\nabla_y h(y,x')$  (the partial
gradient of $h$ by the first variable). Then
$\bar\Lambda_\e(x,0)=\Lambda_\e(x)$ and
$\bar\Theta_\e(x,0)=\Theta_\e(x)$. Moreover, for every
$y'\in\R^{N-1}$ and for $\e>0$ sufficiently small, changing
variables in the integral gives
\begin{multline}
\label{L2009limew03zeta71288888Contprod12888shrrrp8upplus}
\int\limits_{\mathcal{K}}\;\int\limits_{-\e/(2L)}^{\e/(2L)}F\bigg(\bar\Theta_\e(x,y'),\,\bar\Lambda_\e(x,y'),\,\{\vec
B\cdot
v\}\big(g(x'),x'\big),\,\zeta\Big(x_1,f^+\big(g(x'),x'\big),f^-\big(g(x'),x'\big)\Big)
\bigg)\,dx_1dx'=\\
\int\limits_{\mathcal{K}+\e
y'}\;\int\limits_{-\e/(2L)}^{\e/(2L)}F\bigg(\tilde\Theta_\e(x,y'),\,\tilde\Lambda_\e(x,y'),\,\{\vec
B\cdot v\}\big(g(x'-\e y'),x'-\e y'\big),\,
\zeta\Big(x_1,f^+\big(g(x'-\e y'),x'-\e y'\big)\Big)
\bigg)\,dx_1dx'\\=
\int\limits_{\mathcal{K}}\;\int\limits_{-\e/(2L)}^{\e/(2L)}F\bigg(\tilde\Theta_\e(x,y'),\,\tilde\Lambda_\e(x,y'),\,\{\vec
B\cdot v\}\big(g(x'-\e y'),x'-\e
y'\big),\,\zeta\Big(x_1,f^+\big(g(x'-\e y'),x'-\e y'\big)\Big)
\bigg)\,dx_1dx'+\e\tilde l^{(0)}_\e(y')\,,
\end{multline}
where,
\begin{multline}
\label{L2009testgradshiftaddlim12888tilde} \tilde\Lambda_\e(x,y'):=
\delta(x_1/\e,x'-\e y')+
\vec A\cdot\nabla_1 h(x/\e,x'-\e y')-\vec
A\cdot\Big\{\partial_{y_1}h(x/\e,x'-\e y') \otimes\nabla_{x}g(x'-\e
y')\Big\}
\,,
\end{multline}
\begin{multline}
\label{L2009testgradgradshiftaddlim12888tilde}
\tilde\Theta_\e(x,y'):=\theta(x_1/\e,x'-\e y')+
\nabla_1 \{\vec A\cdot\nabla_1 h\}(x/\e,x'-\e y')-\{\vec
A\cdot\nabla_1\partial_{y_1}h\}(x/\e,x'-\e
y')\otimes\nabla_{x}g(x'-\e y')\\-\vec
A\cdot\Big\{\nabla_{x}g(x'-\e y')\otimes\nabla_1\partial_{y_1}h(x/\e,x'-\e
y')\Big\}+\vec
A\cdot\Big\{\partial^2_{y_1y_1}h(x/\e,x'-\e y')
\otimes\nabla_{x}g(x'-\e y')\Big\}\otimes\nabla_{x}g(x'-\e y')
\,,
\end{multline}
and $\tilde l^{(0)}_\e(y')\to 0$ as $\e\to 0^+$ and $\|\tilde
l^{(0)}_\e(y')\|_{L^\infty(\mathcal{D})}<C$ for every bounded set
$\mathcal{D}\subset\subset\R^{N-1}$.
On the other hand, since $\delta,\theta\in L^1_{loc}$, we deduce
that
\begin{multline}
\label{L2009testgradgradshiftlim12888kplim}
\int\limits_{\mathcal{U}}\;\int\limits_{-\e/(2L)}^{\e/(2L)}\frac{1}{\e}\Big(\big|\theta(x_1/\e,x'-\e
y')-\theta(x_1/\e,x')\big|+\big|\delta(x_1/\e,x'-\e
y')-\delta(x_1/\e,x')\big|\Big)dx_1 dx'=\\
\int\limits_{\mathcal{U}}\;\int\limits_{-1/(2L)}^{1/(2L)}\Big(\big|\theta(x_1,x'-\e
y')-\theta(x_1,x')\big|+\big|\delta(x_1,x'-\e
y')-\delta(x_1,x')\big|\Big)dx_1dx'\to 0\,,
\end{multline}
as $\e\to 0^+$. Moreover, since $\{\vec B\cdot
v\}\big(g(x'),x'\big),f^+\big(g(x'),x'\big),f^-\big(g(x'),x'\big),
\nabla g(x')\in L^1_{loc}(\R^{N-1})$ we also have
\begin{multline}
\label{L2009testgradgradshiftlim12888kplimplos}
\int\limits_{\mathcal{U}}\bigg(\Big|\{\vec B\cdot v\}\big(g(x'-\e
y'),x'-\e y'\big)-\{\vec B\cdot
v\}\big(g(x'),x'\big)\Big|+\Big|f^+\big(g(x'-\e y'),x'-\e
y'\big)-f^+\big(g(x'),x'\big)\Big|\\ +\Big|f^-\big(g(x'-\e y'),x'-\e
y'\big)-f^-\big(g(x'),x'\big)\Big|+\Big|\nabla g(x'-\e y')-\nabla
g(x')\Big|\bigg)dx'\to 0\,,
\end{multline}
as $\e\to 0^+$. Therefore, by
\er{L2009testgradgradshiftlim12888kplim},
\er{L2009testgradgradshiftlim12888kplimplos} and
\er{L2009limew03zeta71288888Contprod12888shrrrp8upplus}
we deduce
\begin{multline}
\label{L2009limew03zeta71288888Contprod12888shrrrp8upplusminuslim}
\Bigg|\int\limits_{\mathcal{K}}\;\int\limits_{-\e/(2L)}^{\e/(2L)}\frac{1}{\e}
F\bigg(\bar\Theta_\e(x,y'),\,\bar\Lambda_\e(x,y'),\,\{\vec B\cdot
v\}\big(g(x'),x'\big),\,\zeta\Big(x_1,f^+\big(g(x'),x'\big),f^-\big(g(x'),x'\big)\Big)
\bigg)\,dx_1dx'\\-\int\limits_{\mathcal{K}}\;\int\limits_{-\e/(2L)}^{\e/(2L)}
\frac{1}{\e}F\bigg(\Theta_\e(x),\,\Lambda_\e(x),\,\{\vec B\cdot
v\}\big(g(x'),x'\big),\,\zeta\Big(x_1,f^+\big(g(x'),x'\big),f^-\big(g(x'),x'\big)\Big)
\bigg)\,dx_1dx'\Bigg|
=\tilde l_\e(y')\,,
\end{multline}
where $\tilde l_\e(y')\to 0$ as $\e\to 0^+$ and $\|\tilde
l_\e(y')\|_{L^\infty(\mathcal{D})}<C$ for every bounded set
$\mathcal{D}\subset\subset\R^{N-1}$. Then by
\er{L2009limew03zeta71288888Contprod12888shrrrp8} and
\er{L2009limew03zeta71288888Contprod12888shrrrp8upplusminuslim} we
infer
\begin{multline}
\label{L2009limew03zeta71288888Contprod12888shrrrp8sled}
\int\limits_{\O}\frac{1}{\e} \bigg\{F\Big(\e\nabla\{\vec
A\cdot\nabla \psi_{\e}\},\,\vec A\cdot\nabla \psi_{\e},\,\vec
B\cdot\psi_{\e},\,f\Big)-F\Big(\e\nabla\{\vec A\cdot\nabla
u_{\e}\},\,\vec A\cdot\nabla u_{\e},\,\vec
B\cdot u_{\e},\,f\Big)\bigg\}\,dx=\\
\int\limits_{\mathcal{K}}\;\int\limits_{-\e/(2L)}^{\e/(2L)}\int_{I_L^{N-1}}\frac{L^{N-1}}{\e}\Bigg\{
F\bigg(\theta(x_1/\e,x'),\,\delta(x_1/\e,x'),\,\{\vec B\cdot
v\}\big(g(x'),x'\big),\,\zeta\Big(x_1,f^+\big(g(x'),x'\big),f^-\big(g(x'),x'\big)\Big)\bigg)
\\-F\bigg(\bar\Theta_\e(x,y'),\,\bar\Lambda_\e(x,y'),\,\{\vec B\cdot
v\}\big(g(x'),x'\big),\,\zeta\Big(x_1,f^+\big(g(x'),x'\big),f^-\big(g(x'),x'\big)\Big)
\bigg)\Bigg\}\,dy'dx_1dx'
+\tilde l_\e=\\
\int\limits_{\mathcal{U}}\;\int\limits_{-\e/(2L)}^{\e/(2L)}\int_{I_L^{N-1}}\frac{L^{N-1}}{\e}\Bigg\{
F\bigg(\theta(x_1/\e,x'),\,\delta(x_1/\e,x'),\,\{\vec B\cdot
v\}\big(g(x'),x'\big),\,\zeta\Big(x_1,f^+\big(g(x'),x'\big),f^-\big(g(x'),x'\big)\Big)\bigg)
\\-F\bigg(\bar\Theta_\e(x,y'),\,\bar\Lambda_\e(x,y'),\,\{\vec B\cdot
v\}\big(g(x'),x'\big),\,\zeta\Big(x_1,f^+\big(g(x'),x'\big),f^-\big(g(x'),x'\big)\Big)
\bigg)\Bigg\}\,dy'dx_1dx' +\tilde l_\e\,,
\end{multline}
where $\lim_{\e\to 0}\tilde l_\e=0$ and
$I_L^{N-1}:=\{y'\in\R^{N-1}:-1/(2L)\leq y'_j\leq
1/(2L)\;\text{if}\;1\leq j\leq (N-1)\}$.
%
%
%
%
%
%
Next since for every locally integrable function $P:\R^{N-1}\to\R$
satisfying $P\big(y'_1,y'_2,\ldots (y'_j+1/L),\ldots
y'_{N-1}\big)=P\big(y'_1,y'_2,\ldots y'_j,\ldots y'_{N-1}\big)$ for
every $1\leq j\leq N-1$ we have
$$\int_{I_L^{N-1}}P(y'+z')dy'=\int_{I_L^{N-1}}P(y')dy'\quad\forall z'\in\R^{N-1}\,,$$
by \er{obnulenie1} and
\er{L2009limew03zeta71288888Contprod12888shrrrp8sled} we deduce
\begin{multline}
\label{L2009limew03zeta71288888Contprod12888shrrrp8sledsledush}
\int\limits_{\O}\frac{1}{\e} \bigg\{F\Big(\e\nabla\{\vec
A\cdot\nabla \psi_{\e}\},\,\vec A\cdot\nabla \psi_{\e},\,\vec
B\cdot\psi_{\e},\,f\Big)-F\Big(\e\nabla\{\vec A\cdot\nabla
u_{\e}\},\,\vec A\cdot\nabla u_{\e},\,\vec
B\cdot u_{\e},\,f\Big)\bigg\}\,dx=\\
\int\limits_{\mathcal{U}}\;\int\limits_{-\e/(2L)}^{\e/(2L)}\int_{I_L^{N-1}}\frac{L^{N-1}}{\e}\Bigg\{
F\bigg(\theta(x_1/\e,x'),\,\delta(x_1/\e,x'),\,\{\vec B\cdot
v\}\big(g(x'),x'\big),\,\zeta\Big(x_1,f^+\big(g(x'),x'\big),f^-\big(g(x'),x'\big)\Big)\bigg)
\\-F\bigg(\bar\Theta_\e(x,y'-x'/\e),\,\bar\Lambda_\e(x,y'-x'/\e),\,\{\vec B\cdot
v\}\big(g(x'),x'\big),\,\zeta\Big(x_1,f^+\big(g(x'),x'\big),f^-\big(g(x'),x'\big)\Big)
\bigg)\Bigg\}\,dy'dx_1dx'
+\tilde l_\e\,.
\end{multline}
Therefore, changing the variables $z_1:=Lx_1/\e$, $z':=Ly'$ in
\er{L2009limew03zeta71288888Contprod12888shrrrp8sledsledush}
together with \er{nulltoone} gives
%
%
%
%
%
%
%
%
\begin{multline}
\label{L2009limew03zeta71288888Contprod12888shrrrp8sledsledushrbb}
\lim\limits_{\e\to 0}\int\limits_{\O}\frac{1}{\e}
\bigg\{F\Big(\e\nabla\{\vec A\cdot\nabla \psi_{\e}\},\,\vec
A\cdot\nabla \psi_{\e},\,\vec
B\cdot\psi_{\e},\,f\Big)-F\Big(\e\nabla\{\vec A\cdot\nabla
u_{\e}\},\,\vec A\cdot\nabla u_{\e},\,\vec B\cdot
u_{\e},\,f\Big)\bigg\}\,dx=\\
\int\limits_{\mathcal{U}}\;\int\limits_{-1/2}^{1/2}\int_{I_1^{N-1}}\frac{1}{L}\Bigg\{
F\bigg(\theta(z_1/L,x'),\delta(z_1/L,x'),\{\vec B\cdot
v\}\big(g(x'),x'\big),\zeta\Big(z_1,f^+\big(g(x'),x'\big),f^-\big(g(x'),x'\big)\Big)\bigg)-
\\ F\bigg(\theta(z_1/L,x')+\sigma(z,x'),\delta(z_1/L,x')+\kappa(z,x'),\{\vec B\cdot
v\}\big(g(x'),x'\big),\zeta\Big(z_1,f^+\big(g(x'),x'\big),f^-\big(g(x'),x'\big)\Big)
\bigg)\Bigg\}dz'dz_1dx'
\\=
\int\limits_{\mathcal{U}}\;\int\limits_{\R}\int_{I_1^{N-1}}\frac{1}{L}\Bigg\{
F\bigg(\theta(z_1/L,x'),\delta(z_1/L,x'),\{\vec B\cdot
v\}\big(g(x'),x'\big),\zeta\Big(z_1,f^+\big(g(x'),x'\big),f^-\big(g(x'),x'\big)\Big)\bigg)-
\\ F\bigg(\theta(z_1/L,x')+\sigma(z,x'),\delta(z_1/L,x')+\kappa(z,x'),\{\vec B\cdot
v\}\big(g(x'),x'\big),\zeta\Big(z_1,f^+\big(g(x'),x'\big),f^-\big(g(x'),x'\big)\Big)
\bigg)\Bigg\}dz'dz_1dx',
\end{multline}
where
\begin{equation}
\label{L2009testgradshiftaddlim12888upchtpnew} \kappa(z,x'):=
\vec A\cdot \nabla_z h_0(z,x')-\vec
A\cdot\Big\{\partial_{z_1}h_0(z,x') \otimes\nabla_{x}g(x')\Big\}
\,,
\end{equation}
and
\begin{multline}
\label{L2009testgradgradshiftaddlim12888rupchtpppnew} \sigma(z,x'):=
L\bigg(\nabla_z\{\vec A\cdot\nabla_z h_0\}(z,x')-\{\vec
A\cdot\nabla_z\partial_{z_1}h_0\}(z,x')\otimes\nabla_{x}g(x')\\-\vec
A\cdot\Big\{\nabla_{x}g(x')\otimes\nabla_z\partial_{z_1}h_0(z,x')\Big\} +\vec
A\cdot\Big\{\partial^2_{z_1z_1}h_0(z,x')
\otimes\nabla_{x}g(x')\Big\}\otimes\nabla_{x}g(x')\bigg)
\,.
\end{multline}
On the other hand we have
\begin{multline}
\label{L2009testgradshiftlim12888kp2sld} \delta(t,x')=
\bigg(\int\limits_{H_+\big((g(x'),x'),\vec
n(x')\big)}\eta\Big(z-\{t\vec n_1(x')\}\vec n(x')\Big)dz\bigg)
\{\vec A\cdot \nabla
v\}^+\big(g(x'),x'\big)+\\\bigg(\int\limits_{H_-\big((g(x'),x'),\vec
n(x')\big)}\eta\Big(z-\{t\vec n_1(x')\}\vec n(x')\Big)dz\bigg)\{\vec
A\cdot \nabla v\}^-\big(g(x'),x'\big)= \Gamma\Big(-t\vec
n_1(x'),\{g(x'),x'\}\Big)
\,,
\end{multline}
and
\begin{multline}
\label{L2009testgradgradshiftlim12888kp3sld} \theta(t,x')=-\{\vec A\cdot \nabla
v\}^+\big(g(x'),x'\big)\otimes
\bigg(\int\limits_{H_+\big((g(x'),x'),\vec
n(x')\big)}\nabla\eta\Big(z-\{t\vec n_1(x')\}\vec
n(x')\Big)dz\bigg)\\-\{\vec A\cdot \nabla
v\}^-\big(g(x'),x'\big)\otimes\bigg(\int\limits_{H_-\big((g(x'),x'),\vec
n(x')\big)}\nabla\eta\Big(z-\{t\vec n_1(x')\}\vec
n(x')\Big)dz\bigg)\\=p\Big(-t\vec n_1(x'),\{g(x'),x'\}\Big)\bigg(\{\vec A\cdot \nabla
v\}^+\big(g(x'),x'\big)-\{\vec A\cdot \nabla
v\}^-\big(g(x'),x'\big)\bigg)\otimes\vec
n(x')
\,,
\end{multline}
where as in \er{defp} and \er{matrixgamadef7845389yuj9},
\begin{equation}\label{defpnnn}
p(t,x):=\int_{H^0_{\vec n(x')}}\eta(t\vec n(x')+y)\,d \mathcal
H^{N-1}(y)\,,
\end{equation}
\begin{equation}\label{matrixgamadef7845389yuj9nnn}
\Gamma(t,x):=\bigg(\int_{-\infty}^t p(s,x)ds\bigg)\cdot \{\vec
A\cdot \nabla v\}^-\big(g(x'),x'\big)+\bigg(\int_t^\infty
p(s,x)ds\bigg)\cdot\{\vec A\cdot \nabla v\}^+\big(g(x'),x'\big)\,,
\end{equation}
and by $\vec n_1$ we denote the first coordinate of $\vec n$.

Next define $\tilde h(y,x')\in C^\infty(\R^N\times\R^{N-1},\R^d)$ by
\begin{equation}\label{tildedeftt}
\tilde h(y,x'):=h_0\Big(\big\{y_1-\nabla_{x'}g(x')\cdot
y',y'\big\},x'\Big)\,,
\end{equation}
where $y=(y_1,y')\in\R\times\R^{N-1}$, and set
\begin{equation}\label{tildedefteee}
\vec q_j(x'):=\begin{cases}\vec
n(x')\quad\quad\quad\quad\text{if}\quad j=1\,,\\
\big(\nabla_{x}g(x')\cdot \vec e_j\big)\,\vec e_1+\vec
e_j\quad\quad\text{if}\quad 2\leq j\leq N\,.\end{cases}\,.
\end{equation}
Then using \er{L2009aprplmingg23128normal} by definition we have
\begin{equation}\label{tildedefteeeort}
\vec q_j(x')\cdot\vec q_1(x')=\vec q_j(x')\cdot\vec
n(x')=0\quad\quad\text{for every}\quad 2\leq j\leq N\,.
\end{equation}
Moreover, by \er{obnulenie} we have
\begin{multline}\label{obnulenietilde1}
\tilde h(y,x')=0\;\text{ if }\;\big|\big(y\,\cdot \vec
q_1(x')\big)\big|\geq \frac{1}{2}\vec n_1(x'),\;\text{ and }\;
\tilde h\big(y+\vec q_j(x'),x'\big)=\tilde h(y,x')\;\;\forall
j=2,\ldots, N\,,
\end{multline}
and by \er{nositel},
\begin{equation}\label{nositeltilde1}
\supp \tilde h(y,x')\subset\subset\R^N\times \mathcal{U}\,.
\end{equation}
Furthermore, by the definitions \er{tildedeftt},
\er{L2009testgradshiftaddlim12888upchtpnew}, and
\er{L2009testgradgradshiftaddlim12888rupchtpppnew} we deduce
\begin{equation}\label{tildedefttgrad}
\begin{split}
\vec A\cdot\nabla_y\tilde
h(y,x')=\kappa\Big(\big\{y_1-\nabla_{x'}g(x')\cdot
y',y'\big\},x'\Big)\,,\\ L\nabla_y\{\vec A\cdot\nabla_y\tilde
h\}(y,x')=\sigma\Big(\big\{y_1-\nabla_{x'}g(x')\cdot
y',y'\big\},x'\Big)\,.
\end{split}
\end{equation}
Then, since by
\er{L2009limew03zeta71288888Contprod12888shrrrp8sledsledushrbb} we
have
\begin{multline}
\label{L2009limew03zeta71288888Contggiuugg777899887787888899}
\lim\limits_{\e\to 0}\int\limits_{\O}\frac{1}{\e}
\bigg\{F\Big(\e\nabla\{\vec A\cdot\nabla \psi_{\e}\},\,\vec
A\cdot\nabla \psi_{\e},\,\vec
B\cdot\psi_{\e},\,f\Big)-F\Big(\e\nabla\{\vec A\cdot\nabla
u_{\e}\},\,\vec A\cdot\nabla u_{\e},\,\vec
B\cdot u_{\e},\,f\Big)\bigg\}\,dx=\\
\int\limits_{\mathcal{U}}\;\int\limits_{\R}\int_{I_1^{N-1}}\frac{1}{L}\Bigg\{
F\bigg(\theta(z_1/L,x'),\delta(z_1/L,x'),\{\vec B\cdot
v\}\big(g(x'),x'\big),\zeta\Big(z_1,f^+\big(g(x'),x'\big),f^-\big(g(x'),x'\big)\Big)\bigg)
\\ -F\bigg(\theta(z_1/L,x')+\sigma(z,x'),\delta(z_1/L,x')+\kappa(z,x'),\{\vec B\cdot
v\}\big(g(x'),x'\big),\zeta\Big(z_1,f^+\big(g(x'),x'\big),f^-\big(g(x'),x'\big)\Big)
\bigg)\\ \Bigg\}dz'dz_1dx'\,,
\end{multline}
changing variables $(z_1,z')=\big(y_1-\nabla_{x'}g(x')\cdot
y',y'\big)$ of the internal integration in all places in the r.h.s.
of \er{L2009limew03zeta71288888Contggiuugg777899887787888899}
together with \er{L2009testgradshiftaddlim12888upchtpnew},
\er{L2009testgradgradshiftaddlim12888rupchtpppnew},
\er{L2009testgradshiftlim12888kp2sld},
\er{L2009testgradgradshiftlim12888kp3sld} and \er{tildedefttgrad}
gives
\begin{multline}
\label{L2009limew03zeta71288888Contggiuuggyyyy8878999}
\lim\limits_{\e\to 0}\int\limits_{\O}\frac{1}{\e}
\bigg\{F\Big(\e\nabla\{\vec A\cdot\nabla \psi_{\e}\},\,\vec
A\cdot\nabla \psi_{\e},\,\vec
B\cdot\psi_{\e},\,f\Big)-F\Big(\e\nabla\{\vec A\cdot\nabla
u_{\e}\},\,\vec A\cdot\nabla u_{\e},\,\vec B\cdot
u_{\e},\,f\Big)\bigg\}\,dx=\\ \int\limits_{S}\;\int\limits_
{\{y\in\R^N:\,y'\in
I_1^{N-1}\}}\frac{1}{L\sqrt{\big(1+|\nabla_{x'}g(x')|^2\big)}}\Bigg\{
F\bigg(p\big(-\vec n(x')\cdot y\,/L,x\big) \big(\{\vec A\cdot\nabla
v\}^+(x)-\{\vec A\cdot\nabla v\}^-(x)\big)\otimes\vec n(x'),\\
\Gamma\big(-\vec n(x')\cdot y\,/L,x\big),\,\{\vec B\cdot
v\}(x),\,\zeta\big(\vec n(x')\cdot y,f^+(x),f^-(x)\big)\bigg)
\\-F\bigg(p\big(-\vec
n(x')\cdot y\,/L,x\big)\big(\{\vec A\cdot\nabla
v\}^+(x)-\{\vec A\cdot\nabla
v\}^-(x)\big)\otimes\vec n(x')+L\nabla_y\{\vec A\cdot\nabla_y\tilde h\}(y,x'),\\
\,\Gamma\big(-\vec n(x')\cdot y\,/L,x\big)+\vec A\cdot\nabla_y\tilde
h(y,x'),\,\{\vec B\cdot v\}(x),\,\zeta\big(\vec n(x')\cdot
y,f^+(x),f^-(x)\big) \bigg) \Bigg\}\,dy\,d\mathcal{H}^{N-1}(x)\,,
\end{multline}
Consider the linear transformation $Q_{x'}(s)=Q_{x'}(s_1,s_2,\ldots,
s_N):\R^N\to\R^N$ by
\begin{equation}\label{chfhh2009y888}
Q_{x'}(s):=\sum\limits_{j=1}^{N}s_j\,\vec q_j(x')\,,
\end{equation}
where $\vec q_j$ is defined by \er{tildedefteee}. Then by
\er{tildedefteee}
\begin{equation}\label{chfhh2009y888det88989}
\det\{Q_{x'}\}=\vec
n_1(x')\Big(1+\big|\nabla_{x'}g(x')\big|^2\Big)=\sqrt{\Big(1+\big|\nabla_{x'}g(x')\big|^2\Big)}\,.
\end{equation}
Moreover, we have
\begin{multline}\label{chfhh2009y888img889899}
Q_{x'}\bigg(\Big\{s\in\R^N:\;s_1>0,\;s_j\in
\big(-1/2+\alpha_{x'}s_1,1/2+\alpha_{x'}s_1\big)\;\forall j\geq
2\Big\}\bigg)\\=\{\vec n(x')\cdot y>0\}\cap\{y'\in I_1^{N-1}\}\,,
\\
Q_{x'}\bigg(\Big\{s\in\R^N:\;s_1<0,\;s_j\in
\big(-1/2+\alpha_{x'}s_1,1/2+\alpha_{x'}s_1\big)\;\forall j\geq
2\Big\}\bigg)\\=\{\vec n(x')\cdot y<0\}\cap\{y'\in I_1^{N-1}\}\,,
\end{multline}
where $\alpha_{x'}:=\vec n_1(x')\big(\nabla_{x}g(x')\cdot \vec
e_j\big)$ and by \er{obnulenietilde1} we deduce
\begin{equation}\label{chfhh2009y888set8898999}
\tilde h\big(Q_{x'}(s+\vec e_j),x'\big)=\tilde
h\big(Q_{x'}(s),x'\big)\;\;\forall j=2,\ldots, N\,.
\end{equation}
Therefore, changing variables from $y$ to $s$ in
\er{L2009limew03zeta71288888Contggiuuggyyyy8878999}
and using \er{tildedefteeeort}, \er{chfhh2009y888det88989},
\er{chfhh2009y888img889899} and a periodicity condition
\er{chfhh2009y888set8898999} we deduce
\er{L2009limew03zeta71288888Contggiuuggyyyynew88789999vprop78899shtrih}.
\end{proof}

\begin{lemma}\label{mnogohypsi88kkkll88}
Let $S$, $g$, $\mathcal{U}$, $\vec n$, $\theta_0$, $g_\e$,
$\mathcal{P}(\mathcal{U})$, $h$, $u_\e$, $v$ $F$, $f$, $\vec A$,
$\vec B$, $\psi_\e$, $p$ and $\Gamma$ be the same as in
Proposition \ref{mnogohypsi88} and let $L>0$. Then
\begin{equation}
\label{L2009limew03zeta71288888Contggiuuggyyyynew88789999vprop78899shtrihkkkll}
\inf\limits_{h_0\in\mathcal{P}(\mathcal{U})}\bigg\{\int\limits_{S}P_x\big(\tilde
h(\cdot,x')\big)\,d\mathcal{H}^{N-1}(x)\bigg\}
= \int\limits_{S}\Big\{\inf\limits_{\sigma\in
\mathcal{R}(x')}P_x\big(\sigma(\cdot)\big)
\Big\}\,d\mathcal{H}^{N-1}(x) \,,
\end{equation}
where
\begin{multline}
\label{L2009limew03zeta71288888Contggiuuggyyyynew88789999vprop78899shtrihkkkllyhjyukjkkmmm}
P_x(\sigma):=
\int\limits_{\R}\int\limits_{I_1^{N-1}}\frac{1}{L}F\bigg(p(-s_1/L,x)\Big\{(\vec
A\cdot\nabla v)^+(x)-(\vec A\cdot\nabla
v)^-(x)\Big\}\otimes\vec n(x')+L\nabla_y\{\vec A\cdot\nabla_y\sigma\}\big(Q_{x'}(s)\big),\\
\,\Gamma(-s_1/L,x)+\vec A\cdot\nabla_y\sigma\big(Q_{x'}(s)\big),\,\vec B\cdot v(x),\,\zeta\big(s_1,f^+(x), f^-(x)\big) \bigg)\,ds'ds_1\,,
\end{multline}
with $(s_1,s'):=s\in\R\times\R^{N-1}$, $\tilde h(y,x')\in
C^\infty(\R^N\times\R^{N-1},\R^d)$ is given by
\begin{equation}\label{tildedefttshtrihkkkll}
\tilde h(y,x'):=h_0\Big(\big\{y_1-\nabla_{x'}g(x')\cdot
y',y'\big\},x'\Big)\,,
\end{equation}
\begin{multline}\label{khhukgghfhgkkkkkll}
\mathcal{R}(x'):= \bigg\{\sigma(y)\in C^\infty(\R^N,\R^d):\\
\sigma(y)=0\;\text{ if }\;\big|\big(y\,\cdot \vec
q_1(x')\big)\big|\geq \frac{1}{2}\vec n_1(x'),\;\text{ and }\;
\sigma\big(y+\vec q_j(x')\big)=\sigma(y)\;\;\forall j=2,\ldots,
N\,,\bigg\}
\end{multline}
and the linear transformation $Q_{x'}(s)=Q_{x'}(s_1,s_2,\ldots,
s_N):\R^N\to\R^N$ is defined by
\begin{equation}\label{chfhh2009y888shtrihkkkll}
Q_{x'}(s):=\sum\limits_{j=1}^{N}s_j\,\vec q_j(x')\,,
\end{equation}
with
\begin{equation}\label{tildedefteeeshtrihkkkll}
\vec q_j(x'):=\begin{cases}\vec
n(x')\quad\quad\quad\quad\text{if}\quad j=1\,,\\
\big(\nabla_{x}g(x')\cdot \vec e_j\big)\,\vec e_1+\vec
e_j\quad\quad\text{if}\quad 2\leq j\leq N\,.\end{cases}\,.
\end{equation}
\end{lemma}
\begin{proof}
First of all by \er{obnulenietilde1}, for every
$h_0\in\mathcal{P}(\mathcal{U})$ we have
\begin{multline}\label{obnulenietilde1kkkll}
\tilde h(y,x')=0\;\text{ if }\;\big|\big(y\,\cdot \vec
q_1(x')\big)\big|\geq \frac{1}{2}\vec n_1(x'),\;\text{ and }\;
\tilde h\big(y+\vec q_j(x'),x'\big)=\tilde h(y,x')\;\;\forall
j=2,\ldots, N\,.
\end{multline}
In particular for every $x'\in \mathcal{U}$ we have
$\sigma_{x'}(\cdot):=\tilde h(\cdot,x')\in \mathcal{R}(x')$. Thus
\begin{equation}
\label{L2009limew03zeta71288888Contggiuuggyyyynew88789999vprop78899shtrihkkkllkkklklklkkkk}
\inf\limits_{h_0\in\mathcal{P}(\mathcal{U})}\bigg\{\int\limits_{S}P_x\big(\tilde
h(\cdot,x')\big)\,d\mathcal{H}^{N-1}(x)\bigg\} \geq
\int\limits_{S}\Big\{\inf\limits_{\sigma\in
\mathcal{R}(x')}P_x\big(\sigma(\cdot)\big)\Big\}\,d\mathcal{H}^{N-1}(x)\,.
\end{equation}
Therefore, we need only to prove the reverse inequality. Next
observe that for every $x\in S$ we have
\begin{multline}
\label{L2009limew03zeta71288888Contggiuuggyyyynew88789999vprop78899shtrihkkkllkkklklklkkkkjkuykolllkkll}
\inf\limits_{\sigma\in \mathcal{R}(x')}P_x(\sigma)\leq P_x(0):=
\int\limits_{\R}\int\limits_{I_1^{N-1}}\frac{1}{L}\times
\\ F\bigg(p(-s_1/L,x)\Big\{(\vec A\cdot\nabla v)^+(x)-(\vec
A\cdot\nabla v)^-(x)\Big\}\otimes\vec n(x'), \Gamma(-s_1/L,x),\vec
B\cdot v(x),\zeta\big(s_1,f^+(x),f^-(x)\big) \bigg)ds'ds_1
\\= \int\limits_{\R}F\bigg(p(t,x)\Big\{(\vec
A\cdot\nabla v)^+(x)-(\vec A\cdot\nabla v)^-(x)\Big\}\otimes\vec
n(x'), \Gamma(t,x),\vec B\cdot v(x),\zeta\big(t,,f^-(x),f^+(x)\big)
\bigg)\,dt
\leq D\,,
\end{multline}
where $D\in(0,+\infty)$ is a constant, not depending on $x$.

Next we prove that the function
\begin{equation}\label{fyuyfuigjhjjjkkkkjykukjjhhjk}
\zeta(x):=\inf\limits_{\sigma\in
\mathcal{R}(x')}P_{x}(\sigma)\quad\forall x\in S
\end{equation}
is Borel measurable. Indeed, consider $\mathcal{O}$, the subspace of
the metric space $C^2_{loc}(\R^N,\R^d)$ (with the family of
semi-norms $\{\|\cdot\|_{W^{2,\infty}(K)}\}_{K\subset\subset\R^N}$),
containing all the functions $\bar\sigma\in C^\infty(\R^N,\R^d)$,
that satisfy
\begin{equation}\label{obnuleniejkjkhhh}
\bar\sigma(y)=0\;\text{ if }\;|y_1|\geq 1/2,\text{ and
}\;\bar\sigma(y+\vec e_j)=\bar\sigma(y)\;\;\forall j=2,\ldots, N\,.
\end{equation}
Then, since the metric space $C^2_{loc}(\R^N,\R^d)$ is separable,
the subspace $\mathcal{O}$ is also separable and therefore there
exists a countable subset $\mathcal{O}_r\subset\mathcal{O}$ which is
dense in the topology of $C^2_{loc}(\R^N,\R^d)$. On the other hand
for every $\bar\sigma\in\mathcal{O}$ and every $x'\in\mathcal{U}$
the function $\sigma_{x'}:=\bar\sigma\big(y_1-\nabla_{x'}g(x')\cdot
y',y'\big)$ belongs to $\mathcal{R}(x')$ and moreover we have
\begin{equation}\label{fyuyfuigjhjjjkkkk}
\zeta(x)=\inf\limits_{\sigma\in
\mathcal{R}(x')}P_{x}(\sigma)=\inf\limits_{\bar\sigma\in\mathcal{O}
}P_{x}(\sigma_{x'})=\inf\limits_{\bar\sigma\in\mathcal{O}_r
}P_{x}(\sigma_{x'})\,.
\end{equation}
Thus since $\mathcal{O}_r$ is countable and since
$P_{x}(\sigma_{x'})$ is Borel measurable on $S$, by
\er{fyuyfuigjhjjjkkkk} we deduce that $\zeta(x)$ is Borel
measurable.
 Next fix any $\e>0$. By Lusin's Theorem there exists a compact set
$K\subset S$ such that $(\vec A\cdot\nabla v)^+(x)$, $(\vec
A\cdot\nabla v)^-(x)$, $f^+(x)$, $f^-(x)$ and $\zeta(x)$ are
continuous functions on $K$ and
\begin{equation}\label{bububup}
\mathcal{H}^{N-1}\big(S\setminus K\big)\leq \frac{\e}{2D}.
\end{equation}
Here $\zeta$ is the function defined by
\er{fyuyfuigjhjjjkkkkjykukjjhhjk}. For any $x\in K$ there exists
$\varphi_x\in\mathcal{R}(x')$ such that
\begin{equation}\label{parovoz3}
P_{x}(\varphi_{x})-\zeta(x)<\frac{\e}{4+4\mathcal{H}^{N-1}(S)}\,.
\end{equation}
Then set
\begin{equation}\label{parovoz3bnghjykjukiu}
\gamma_{z,x}(y):=\varphi_{x}\bigg(y_1+\big(\nabla_{z'}g(z')-\nabla_{x'}g(x')\big)\cdot
y',y'\bigg).
\end{equation}
Using the continuity by $z$ of $\nabla_{z'}g(z')$, $\zeta(z)$ and
$P_{z}(\gamma_{z,x})$ on $K$, we infer that for any $x\in K$ there
exists $\delta_x>0$ such that
\begin{equation}\label{sparovoz4}
P_{z}(\gamma_{z,x})-\zeta(z)<\frac{\e}{2+2\mathcal{H}^{N-1}(S)}\,,\quad
\forall z\in K\cap B_{\delta_x}(x)\,.
\end{equation}
Since the set $K$ is compact, there exists a finite number of points
$x_1,x_2,\ldots,x_l\in K$ such that
$K\subset\bigcup\limits_{j=1}^{l} B_{\delta_{x_j}}(x_j)$. Define the
function $\bar p(y,x')$ on $\R^N\times\mathcal{U}$ by
\begin{equation}
\label{qqqw3po} \bar p(y,x')=\begin{cases}
\gamma_{x,x_i}(y)\quad\forall x=(x_1,x')\in \big(K\bigcap
B_{\delta_{x_i}}(x)\big)\setminus\bigcup\limits_{1\leq j\leq
i-1}B_{\delta_{x_j}}(x_j)\quad\quad 1\leq i\leq l,\\
0\quad\quad\quad\; \forall x\in S\setminus K\,.
\end{cases}
\end{equation}
Then from \eqref{bububup}, \eqref{sparovoz4} and
\er{L2009limew03zeta71288888Contggiuuggyyyynew88789999vprop78899shtrihkkkllkkklklklkkkkjkuykolllkkll}
we get
\begin{equation}
\label{gyfgdgdtgdjkjbhjjjjjjjjkkkkll}
\int_{S}P_{x}
\big(\bar
p(\cdot,x')\big)\,d\mathcal{H}^{N-1}(x)-\int_{S}\zeta(x)\,d\mathcal{H}^{N-1}(x)<\e\,.
\end{equation}
Moreover, $\bar p(y,x')$ satisfies
\begin{multline}\label{khhukgghfhgkkkkklljjkkjjkk}
\bar p(y,x')=0\;\text{ if }\;\big|\big(y\,\cdot \vec
q_1(x')\big)\big|\geq \frac{1}{2}\vec n_1(x'),\;\text{ and }\; \bar
p\big(y+\vec q_j(x'),x'\big)=\bar p(y,x')\;\;\forall j=2,\ldots,
N\,,
\end{multline}
and $\bar p(y,x')\in L^\infty\big(\mathcal{U},C^k(K,\R^d)\big)$ for
every $K\subset\subset\R^N$ and every natural $k$. Next define
$p_0:\R^N\times\mathcal{U}\to\R^d$ by
\begin{equation}\label{tildedefttshtrihkkkllkkkkkkkkjkjkl}
p_0(y,x'):=\bar p\Big(\big\{y_1+\nabla_{x'}g(x')\cdot
y',y'\big\},x'\Big)\,,
\end{equation}
Then
\begin{equation}\label{tildedefttshtrihkkkllkkkkkkkk}
\bar p(y,x'):=p_0\Big(\big\{y_1-\nabla_{x'}g(x')\cdot
y',y'\big\},x'\Big)\,,
\end{equation}
and
\begin{multline}\label{khhukgghfhgkkkkklljjkkjjkkhlhhhhh}
p_0(y,x')=0\;\text{ if }\;|y_1|\geq 1/2,\;\text{ and }\;
p_0\big(y+\vec e_j,x'\big)=p_0(y,x')\;\;\forall j=2,\ldots, N\,.
\end{multline}
Moreover, $p_0(y,x')\in L^\infty\big(\mathcal{U},C^k(K,\R^d)\big)$
for every $K\subset\subset\R^N$ and every natural $k$. Next let
$\omega\in C^\infty_c(\R^{N-1},\R)$ be such that $\omega\geq 0$ and
$\int_{\R^{N-1}}\omega(y')dy'=1$. For any $0<\rho<1$ define
\begin{equation}\label{cvvtgjjlkk}
p_\rho(y,x')=\frac{1}{\rho^{N-1}}\int_{\R^{N-1}}
\omega\Big(\frac{x'-z'}{\rho}\Big)p_0 (y,z')\,dz'=\int_{\R^{N-1}}
\omega(z')p_0 (y,x'+\rho z')\,dz'\,.
\end{equation}
Then $p_\rho\in C^\infty(\R^N\times\mathcal{U},\R^d)$ and
\begin{multline}\label{khhukgghfhgkkkkklljjkkjjkkhlhhhhhkkkl}
p_\rho(y,x')=0\;\text{ if }\;|y_1|\geq 1/2,\;\text{ and }\;
p_\rho\big(y+\vec e_j,x'\big)=p_\rho(y,x')\;\;\forall j=2,\ldots,
N\,.
\end{multline}
Furthermore, there exists a constant $M>0$, independent of $\rho$,
$y$ and $x'$, such that for every $0<\rho<1$ we have
\begin{equation}\label{vhhdhcdhchfdhc}
|\nabla^2_y p_\rho(y,x')|+|\nabla_y p_\rho(y,x')|+|
p_\rho(y,x')|\leq M\,.
\end{equation}
Moreover, for every $y$, for a.e. $x'\in\mathcal{U}$ we have
\begin{equation}\label{vhhdhcdhchfdhckkkkllll}
\nabla^2_y p_\rho(y,x')\to\nabla^2_y p_0(y,x')\,,\quad\nabla_y
p_\rho(y,x')\to\nabla_y p_0(y,x')\,,\quad  p_\rho(y,x')\to
p_0(y,x')\quad\text{as}\quad\rho\to 0^+\,.
\end{equation}
Therefore, if we define
\begin{equation}\label{tildedefttshtrihkkkllkkkkkkkkjkjk}
\tilde p_\rho(y,x'):=p_\rho\Big(\big\{y_1-\nabla_{x'}g(x')\cdot
y',y'\big\},x'\Big)\,,
\end{equation}
then, by \er{vhhdhcdhchfdhc} and \er{vhhdhcdhchfdhckkkkllll}
\begin{equation}\label{gyfgdgdtgdjkjbhjjjjjjjjkkkkllhhh}
\lim\limits_{\rho\to 0^+}\int_{S}P_{x} \big(\tilde
p_\rho(\cdot,x')\big)\,d\mathcal{H}^{N-1}(x)=\int_{S}P_{x} \big(\bar
p(\cdot,x')\big)\,d\mathcal{H}^{N-1}(x)\,.
\end{equation}
Thus, by \er{gyfgdgdtgdjkjbhjjjjjjjjkkkkll} there exists
$\rho_0\in(0,1)$ such that
\begin{equation}\label{gyfgdgdtgdjkjbhjjjjjjjjkkkkllkllkljhhh}
\int_{S}P_{x} \big(\tilde p_{\rho_0}
(\cdot,x')\big)\,d\mathcal{H}^{N-1}(x)-\int_{S}\Big\{\inf\limits_{\sigma\in
\mathcal{R}(x')}P_x\big(\sigma(\cdot)\big)\Big\}\,d\mathcal{H}^{N-1}(x)<2\e\,.
\end{equation}
Finally consider the sequence of compact subsets
$K_n\subset\subset\mathcal{U}$, such that $K_n\subset K_{n+1}$ and
$\bigcup_{n=1}^{+\infty}K_n=\mathcal{U}$. For every $n$ consider
$\xi_n(x')\in C^\infty_c(\mathcal{U},\R)$, such that $\xi_n(x')=1$
if $x'\in K_n$ and $0\leq\xi_n(x')\leq 1$ for every
$x'\in\mathcal{U}$. For every $n$ define
$h_n(y,x'):=p_{\rho_0}(y,x')\xi_n(x')$. Then $h_n\in
C^\infty(\R^N\times\mathcal{U},\R^d)$. Moreover, by
\er{khhukgghfhgkkkkklljjkkjjkkhlhhhhhkkkl} we have
\begin{multline}\label{khhukgghfhgkkkkklljjkkjjkkhlhhhhhkkklkkok}
h_n(y,x')=0\;\text{ if }\;|y_1|\geq 1/2,\;\text{ and }\;
h_n\big(y+\vec e_j,x'\big)=h_n(y,x')\;\;\forall j=2,\ldots, N\,,
\end{multline}
and by the definition $\ov{\supp h_n}\subset\R^N\times\mathcal{U}$.
Thus $h_n\in \mathcal{P}(\mathcal{U})$. Moreover, by
\er{vhhdhcdhchfdhc}, for every $n$ we have
\begin{equation}\label{vhhdhcdhchfdhcjyhjk}
|\nabla^2_y h_n(y,x')|+|\nabla_y h_n(y,x')|+|h_n(y,x')|\leq M\,,
\end{equation}
and by the definition for every $y$ and $x'$ we have
\begin{equation}\label{vhhdhcdhchfdhckkkklllllil}
\nabla^2_y h_n(y,x')\to\nabla^2_y p_{\rho_0}(y,x')\,,\quad\nabla_y
h_n(y,x')\to\nabla_y p_{\rho_0}(y,x')\,,\quad  h_n(y,x')\to
p_{\rho_0}(y,x')\quad\text{as}\quad n\to +\infty\,.
\end{equation}
Thus if we set
\begin{equation}\label{tildedefttshtrihkkkllkkkkkkkkjkjkdjkxcjk}
\tilde h_n(y,x'):=h_n\Big(\big\{y_1-\nabla_{x'}g(x')\cdot
y',y'\big\},x'\Big)\,,
\end{equation}
then, by \er{vhhdhcdhchfdhcjyhjk} and \er{vhhdhcdhchfdhckkkklllllil}
\begin{equation}\label{gyfgdgdtgdjkjbhjjjjjjjjkkkkllhhheghekkk}
\lim\limits_{n\to +\infty}\int_{S}P_{x} \big(\tilde
h_n(\cdot,x')\big)\,d\mathcal{H}^{N-1}(x)=\int_{S}P_{x} \big(\tilde
p_{\rho_0}(\cdot,x')\big)\,d\mathcal{H}^{N-1}(x)\,.
\end{equation}
Then, by \er{gyfgdgdtgdjkjbhjjjjjjjjkkkkllkllkljhhh} we obtain
\begin{equation}\label{gyfgdgdtgdjkjbhjjjjjjjjkkkkllkllkljhhhdxhjk}
\lim\limits_{n\to +\infty}\int_{S}P_{x} \big(\tilde
h_n(\cdot,x')\big)\,d\mathcal{H}^{N-1}(x)-\int_{S}\Big\{\inf\limits_{\sigma\in
\mathcal{R}(x')}P_x\big(\sigma(\cdot)\big)\Big\}\,d\mathcal{H}^{N-1}(x)\leq
2\e\,.
\end{equation}
Therefore, since $\e>0$ was chosen arbitrary and since $h_n\in
\mathcal{P}(\mathcal{U})$ we deduce
\begin{equation}\label{gyfgdgdtgdjkjbhjjjjjjjjkkkkllkllkljhhhdxhjkhjhjlkk}
\inf\limits_{h_0\in\mathcal{P}(\mathcal{U})}\bigg\{\int_{S}P_{x}
\big(\tilde
h(\cdot,x')\big)\,d\mathcal{H}^{N-1}(x)\bigg\}\leq\int_{S}\Big\{\inf\limits_{\sigma\in
\mathcal{R}(x')}P_x\big(\sigma(\cdot)\big)\Big\}\,d\mathcal{H}^{N-1}(x)\,,
\end{equation}
which together with the reverse inequality
\er{L2009limew03zeta71288888Contggiuuggyyyynew88789999vprop78899shtrihkkkllkkklklklkkkk}
gives the desired result.
\end{proof}

\subsection{Construction of the approximating sequence in the general case}
\begin{lemma}\label{mnogohypsi88kkkll88lpppopkook}
Let
$v$, $F$, $f$, $\vec A$, $\vec B$, $\psi_\e$, $\vec\nu(x),$ $p$ and
$\Gamma$ be the same as in
Theorem \ref{vtporbound4}.
Then for every $\delta>0$ there exist $N$ Borel-measurable functions
$\vec p_j(x):J_{\vec A\cdot\nabla v}\to\R^N$ for $1\leq j\leq N$,
such that $\vec p_1(x):=\vec \nu(x)$, $\vec p_j(x)\cdot\vec\nu(x)=0$
for $2\leq j\leq N$ and $\{\vec p_j(x)\}_{1\leq j\leq N}$ is
linearly independent system of vectors for every $x$, and there
exists a sequence $\{v_\e\}_{0<\e<1}\subset C^\infty(\R^N,\R^d)$
such that $\lim_{\e\to0^+}\vec A\cdot\nabla v_\e=\vec A\cdot \nabla
v$ in $L^p(\R^N,\R^m)$, $\lim_{\e\to0^+} \e\nabla\{\vec A\cdot
\nabla v_\e\}=0$ in $L^{p}(\R^N,\R^{m\times N})$,
$\lim_{\e\to0^+}\vec B\cdot v_\e=\vec B\cdot v$ in
$W^{1,p}(\R^N,\R^k)$, $\lim_{\e\to0^+}(\vec B\cdot v_\e-\vec B\cdot
v)/\e=0$ in $L^{p}(\R^N,\R^k)$ for every $p\geq 1$ and
\begin{equation}
\label{L2009limew03zeta71288888Contggiuuggyyyynew88789999vprop78899shtrihkkklljkhkhggh}
0\leq\lim\limits_{\e\to 0}\int\limits_{\O}\frac{1}{\e}
F\Big(\e\nabla\{\vec A\cdot\nabla v_{\e}\},\,\vec A\cdot\nabla
v_{\e},\,\vec B\cdot v_{\e},\,f\Big)\,dx-\int\limits_{\O\cap J_{\vec
A\cdot\nabla v}}\bigg\{\inf\limits_{\sigma\in
\mathcal{Q}(x),L\in(0,1)}H_x\big(\sigma(\cdot),L\big)
\bigg\}\,d\mathcal{H}^{N-1}(x)<\delta\,,
\end{equation}
where
\begin{multline}
\label{L2009limew03zeta71288888Contggiuuggyyyynew88789999vprop78899shtrihkkkllyhjyukjkkmmmklklklhhhhkk}
H_x(\sigma,L):=
\int\limits_{I_{N}}\frac{1}{L}F\bigg(p(-s_1/L,x)\Big\{(\vec
A\cdot\nabla v)^+(x)-(\vec A\cdot\nabla
v)^-(x)\Big\}\otimes\vec \nu(x)+L\nabla_y\{\vec A\cdot\nabla_y\sigma\}\big(T_{x}(s)\big),\\
\,\Gamma(-s_1/L,x)+\vec
A\cdot\nabla_y\sigma\big(T_{x}(s)\big),\,\vec B\cdot
v(x),\,\zeta\big(s_1,f^+(x),f^-(x)\big) \bigg)\,ds\,,
\end{multline}
\begin{multline}\label{L2009Ddef2hhhjjjj77788nullkkkoooooll}
\mathcal{Q}(x):=\mathcal{D}_1\big(\vec A,\vec\nu(x),\vec p_2(x),\vec
p_3(x),\ldots,\vec p_{N}(x)\big):=\\ \bigg\{ \sigma\in
C^\infty(\R^N,\R^d)\cap L^\infty(\R^N,\R^d)\cap Lip(\R^N,\R^d):\;
\vec A\cdot\nabla \sigma(y)=0\;\text{ if }\; |y\cdot\vec\nu(x)|\geq 1/2\;\\
\text{ and }\;
\sigma\big(y+\vec p_j(x)\big)=\sigma(y) \;\;\forall j=2,3,\ldots,
N\bigg\}\,,
\end{multline}
\begin{equation}\label{cubidepgghkkkkllllllljjjkk}
I_N:= \Big\{s\in\R^N:\;-1/2<s_j<1/2\quad\forall
j=1,2,\ldots, N\Big\}\,,
\end{equation}
and the linear transformation $T_x(s)=T_{x}(s_1,s_2,\ldots,
s_N):\R^N\to\R^N$ is defined by
\begin{equation}\label{chfhh2009y888shtrihkkkllbjgjkll}
T_{x}(s):=\sum\limits_{j=1}^{N}s_j\,\vec p_j(x)\,,
\end{equation}
Moreover, $\vec A\cdot\nabla v_\e$, $\e\nabla\{\vec A\cdot\nabla
v_\e\}$, $\vec B\cdot v_\e$ and $\nabla\{\vec B\cdot v_\e\}$ are
bounded in $L^\infty$ sequences, there exists a compact
$K=K_\delta\subset\subset\O$ such that $v_\e(x)=\psi_\e(x)$ for
every $0<\e<1$ and every $x\in\R^N\setminus K$, and
\begin{equation}\label{bjdfgfghljjklkjkllllkkllkkkkkkllllkkllhffhhfh}
\limsup\limits_{\e\to 0^+}\Bigg|\frac{1}{\e}\bigg(\int_\O\vec
A\cdot\nabla v_\e(x)\,dx-\int_\O\{\vec A\cdot \nabla
v\}(x)\,dx\bigg)\Bigg|<+\infty\,,
\end{equation}
\end{lemma}
\begin{proof}
First of all observe that
\begin{multline}
\label{L2009limew03zeta71288888Contggiuuggyyyynew88789999vprop78899shtrihkkkllkkklklklkkkkjkuykolllkklloooohhhkkkllmmmpom}
\int\limits_{\R}F\bigg(p(t,x)\Big\{(\vec A\cdot\nabla v)^+(x)-(\vec
A\cdot\nabla v)^-(x)\Big\}\otimes\vec \nu(x),
\,\Gamma(t,x),\,\vec B\cdot v(x),\,\zeta\big(t,f^-(x),f^+(x)\big) \bigg)\,dt
\leq D_0\,,
\end{multline}
where $D_0\in(0,+\infty)$ is a constant, not depending on $x$.

 Next since the set $\O\cap J_{\vec A\cdot \nabla v}$  is a countably
$\mathcal{H}^{N-1}$-rectifiable Borel set, oriented by $\vec\nu(x)$,
$\O\cap J_{\vec A\cdot \nabla v}$ is $\sigma$-finite with respect to
$\mathcal{H}^{N-1}$, there exist countably many $C^1$ hypersurfaces
$\{S_k\}^{\infty}_{k=1}$ such that $\mathcal{H}^{N-1}\Big(\O\cap
J_{\vec A\cdot\nabla
v}\setminus\bigcup\limits_{k=1}^{\infty}S_k\Big)=0$, and for
$\mathcal{H}^{N-1}$-almost every $x\in \O\cap J_{\vec A\cdot \nabla
v}\cap S_k$, $\vec\nu(x)=\vec n_k(x)$ where $\vec n_k(x)$ is a
normal to $S_k$ at the point $x$. We also may assume that for every
$k\in\mathbb{N}$ we have $\mathcal{H}^{N-1}(S_k)<+\infty$, $\ov
S_k\subset\subset\O$ and there exists a relabeling of the axes $\bar
x:=Z_k (x)$ such that for some function $g_k(x')\in
C^1(\R^{N-1},\R)$ and a bounded open set
$\mathcal{V}_k\subset\R^{N-1}$ we have
\begin{equation}
\label{L2009surfhh8128klkjjk} S_k=\{x:\;Z_k (x)=\bar x=(\bar
x_1,\bar x'),\;\bar x'\in \mathcal{V}_k,\;\bar x_1=g_k(\bar x')\}\,.
\end{equation}
Moreover $\vec n'_k(x):=Z_k\big(\vec n_k(x)\big)=(1,-\nabla_{\bar
x'}g_k(\bar x'))/\sqrt{1+|\nabla_{\bar x'}g_k(\bar x')|^2}$.
Next clearly
there exists $k_0\in\mathbb{N}$ such that
\begin{multline}\label{ugguuggugugugggjgghghghghgh}
\int_{(\O\cap J_{\vec A\cdot\nabla
v}\setminus\cup_{k=1}^{k_0}S_k)}\Bigg\{\\
\int\limits_{\R}F\bigg(p(t,x)\Big\{(\vec A\cdot\nabla v)^+(x)-(\vec
A\cdot\nabla v)^-(x)\Big\}\otimes\vec \nu(x), \Gamma(t,x),\vec
B\cdot v(x),\zeta\big(t,f^-(x),f^+(x)\big) \bigg)dt
\Bigg\}\,d\mathcal{H}^{N-1}(x)<\frac{\delta}{8}\,.
\end{multline}
Next for every $k=1,2,\ldots k_0$ there exists an open set
$\mathcal{U}_k\subset\R^{N-1}$ such that if for every $k=1,\ldots
k_0$ we set
\begin{equation}\label{ghugtyffygdchlkkl}
S'_k:=\{x\in S_k:\;Z_k (x)=\bar x=(\bar x_1,\bar x'),\;\bar x'\in
\mathcal{U}_k,\;\bar x_1=g_k(\bar x')\}\,,
\end{equation}
then $\ov{\mathcal{U}_k}\subset\subset\mathcal{V}_k$,  $\ov{
S'_k}\subset S_k\setminus(\cup_{j=1}^{k-1}\ov{S'_j})$ and
\begin{equation}\label{ghugtyffygdch}
\mathcal{H}^{N-1}\bigg(\big\{\cup_{j=1}^{k_0}S_j\big\}\setminus\big\{\cup_{j=1}^{k_0}S'_j\big\}\bigg)<\frac{\delta}{8D_0}\,.
\end{equation}
In particular we have $\ov{ S'_j}\cap\ov{ S'_k}=\emptyset$ if $j\neq
k$.

 Next for every $j=1,2\ldots N$, for every $k=1,2,\ldots k_0$ and every $x\in S'_k$ set
\begin{equation}\label{tildedefteeeshtrihkkkllhjgkjk}
\vec l_j(x):=\begin{cases}
\vec n'_k(x)=Z_k\big(\vec n_k(x)\big)\quad\quad\quad\quad\text{if}\quad j=1\,,\\
\big(\nabla_{\bar x}g_k(\bar x')\cdot \vec e_j\big)\,\vec e_1+\vec
e_j\quad\quad\text{if}\quad 2\leq j\leq N\,.\end{cases}\,,
\end{equation}
where $\bar x:=Z_k (x)$ (the corresponding to $k$ relabeling of the
axes) and $\{\vec e_1,\vec e_2,\ldots,\vec e_N\}$ is a standard
orthonormal base in $\R^N$. Then there exist $N$ Borel-measurable
functions $\vec p_j(x):J_{\vec A\cdot\nabla v}\to\R^N$ for $1\leq
j\leq N$, such that $\vec p_1(x):=\vec \nu(x)$, $\vec
p_j(x)\cdot\vec\nu(x)=0$ for $2\leq j\leq N$, $\{\vec
p_j(x)\}_{1\leq j\leq N}$ is linearly independent system of vectors
for every $x$ and the following identity is satisfied
\begin{equation}\label{poprdefhhjh}
Z_k\big(\vec p_j(x)\big)=\begin{cases}\vec l_1(x)=Z_k\big(\vec
\nu(x)\big)\quad\text{if}\;\;j=1\\ \sqrt{1+\big|\nabla_{\bar
x'}g_k\big(\bar x'\big)\big|^2}\,\cdot\,\vec
l_j(x)\quad\text{if}\;\;j=2,3\ldots N\end{cases}\quad\quad\forall
x\in S'_k\quad\forall k=1,2,\ldots k_0\,,
\end{equation}
where again $\bar x:=Z_k (x)$. We let $\vec p_j(x)$ and
$\vec\nu(x):=\vec p_1(x)$ be defined by \er{poprdefhhjh} also for
$x\in\cup_{k=1}^{k_0}S'_k\setminus J_{\vec A\cdot\nabla v}$

Next let
\begin{multline}\label{L2009Ddef2hhhjjjj77788nullooooll}
\mathcal{Q}_0(x):=\mathcal{D}_0\big(\vec A,\vec\nu(x),\vec
p_2(x),\vec p_3(x),\ldots,\vec p_{N}(x)\big):= \Big\{ \sigma\in
C^\infty(\R^N,\R^d):\;
\sigma(y)=0\;\text{ if }\; |y\cdot\vec\nu|\geq 1/2\;\\ \text{ and
}\;\sigma\big(y+\vec p_j(x)\big)=\sigma(y)\;\;\forall j=2,3,\ldots,
N\Big\}\,,
\end{multline}
$\mathcal{Q}(x)$ be defined as in
\er{L2009Ddef2hhhjjjj77788nullkkkoooooll}, $T_x(s)$ be as in
\er{chfhh2009y888shtrihkkkllbjgjkll} and $H_x(\sigma,L)$ be as in
\er{L2009limew03zeta71288888Contggiuuggyyyynew88789999vprop78899shtrihkkkllyhjyukjkkmmmklklklhhhhkk}.

Observe that by the definition of $\sigma\in \mathcal{Q}(x)$ and by
the properties of function $p$, for every $L\leq 1/2$, we have
\begin{multline}
\label{L2009limew03zeta71288888Contggiuuggyyyynew88789999vprop78899shtrihkkkllyhjyukjkkmmmklklklhhhhkkgjkgkjkkk}
H_x(\sigma,L):=
\int\limits_{\R}\int\limits_{I_1^{N-1}}\frac{1}{L}F\bigg(p(-s_1/L,x)\Big\{(\vec
A\cdot\nabla v)^+(x)-(\vec A\cdot\nabla
v)^-(x)\Big\}\otimes\vec \nu(x)+L\nabla_y\{\vec A\cdot\nabla_y\sigma\}\big(T_{x}(s)\big),\\
\,\Gamma(-s_1/L,x)+\vec
A\cdot\nabla_y\sigma\big(T_{x}(s)\big),\,\vec B\cdot
v(x),\,\zeta\big(s_1,f^+(x), f^-(x)\big) \bigg)\,ds'ds_1\,,
\end{multline}
with $(s_1,s'):=s\in\R\times\R^{N-1}$. On the other hand for every
$\sigma\in \mathcal{Q}(x)$
\begin{multline}
\label{L2009limew03zeta71288888Contggiuuggyyyynew88789999vprop78899shtrihkkkllkkklklklkkkkjkuykolllkklloooohhhkkkll}
\inf\limits_{\sigma\in \mathcal{Q}(x)}H_x(\sigma,L)\leq H_x(0,L):=
\int\limits_{\R}\int\limits_{I_1^{N-1}}\frac{1}{L}\times\\
F\bigg(p(-s_1/L,x)\Big\{(\vec A\cdot\nabla v)^+(x)-(\vec
A\cdot\nabla v)^-(x)\Big\}\otimes\vec \nu(x),
\,\Gamma(-s_1/L,x),\,\vec B\cdot v(x),\,\zeta\big(s_1,f^+(x),
f^-(x)\big) \bigg)\,ds'ds_1
\\= \int\limits_{\R}F\bigg(p(t,x)\Big\{(\vec
A\cdot\nabla v)^+(x)-(\vec A\cdot\nabla v)^-(x)\Big\}\otimes\vec
\nu(x), \,\Gamma(t,x),\,\vec B\cdot v(x),\,\zeta\big(t,f^-(x),
f^+(x)\big) \bigg)\,dt\leq D_0,
\end{multline}
where $D_0\in(0,+\infty)$ is a constant from
\er{L2009limew03zeta71288888Contggiuuggyyyynew88789999vprop78899shtrihkkkllkkklklklkkkkjkuykolllkklloooohhhkkkllmmmpom},
which doesn't depend on $x$ and $L$. On the other hand by Lemma
\ref{resatnulo}, for every $k=1,2,\ldots k_0$ and for every $x\in
S'_k$ we deduce that
\begin{equation}\label{hgfhgjykukuioiipoookkkkkkk}
\inf_{L\in(0,1)}\bigg\{\inf\limits_{\sigma\in
\mathcal{Q}(x)}H_x(\sigma,L)\bigg\}=\lim\limits_{L\to
0^+}\Bigg\{\inf\limits_{\sigma\in
\mathcal{Q}_0(x)}H_x\bigg(\sigma\,,\,L\sqrt{1+|\nabla_{\bar
x'}g_k(\bar x')|^2}\bigg)\Bigg\}\,,
\end{equation}
where $\bar x:=Z_k(x)$. Therefore, by
\er{hgfhgjykukuioiipoookkkkkkk},
\er{L2009limew03zeta71288888Contggiuuggyyyynew88789999vprop78899shtrihkkkllkkklklklkkkkjkuykolllkklloooohhhkkkll}
and the Dominated Convergence Theorem, we obtain
\begin{multline}\label{hgfhgjykukuioiipoookkkkkkkilhkjk}
\int\limits_{\cup_{k=1}^{k_0}S'_k}\inf_{L\in(0,1)}\bigg\{\inf\limits_{\sigma\in
\mathcal{Q}(x)}H_x(\sigma,L)\bigg\}\,d\mathcal{H}^{N-1}(x)=\\
\lim\limits_{L\to
0^+}\int\limits_{\cup_{k=1}^{k_0}S'_k}\Bigg\{\inf\limits_{\sigma\in
\mathcal{Q}_0(x)}H_x\bigg(\sigma\,,\,L\sqrt{1+|\nabla_{\bar
x'}g_k(\bar x')|^2}\bigg)\Bigg\}\,d\mathcal{H}^{N-1}(x)\,,
\end{multline}
and thus there exists $L_0>0$, such that $L_0\sqrt{1+|\nabla_{\bar
x'}g_k(\bar x')|^2}<1$ and
\begin{multline}\label{hgfhgjykukuioiipoookkkkkkkilhkjkjuyk}
\int\limits_{\cup_{k=1}^{k_0}S'_k}\Bigg\{\inf\limits_{\sigma\in
\mathcal{Q}_0(x)}H_x\bigg(\sigma\,,\,L_0\sqrt{1+|\nabla_{\bar
x'}g_k(\bar
x')|^2}\bigg)\Bigg\}\,d\mathcal{H}^{N-1}(x)\\-\int\limits_{\cup_{k=1}^{k_0}S'_k}\inf_{L\in(0,1)}\bigg\{\inf\limits_{\sigma\in
\mathcal{Q}(x)}H_x(\sigma,L)\bigg\}\,d\mathcal{H}^{N-1}(x)<\frac{\delta}{8}\,,
\end{multline}
On the other hand, changing the first variable of integration in
\er{L2009limew03zeta71288888Contggiuuggyyyynew88789999vprop78899shtrihkkkllyhjyukjkkmmmklklklhhhhkkgjkgkjkkk},
for every $x\in\cup_{k=1}^{k_0}S'_k$ we obtain
\begin{equation}\label{hgfhgjykukuioiipoookkkkkkkilhkjkjuykjjkkk}
\inf\limits_{\sigma\in
\mathcal{Q}_0(x)}H_x\bigg(\sigma\,,\,L\sqrt{1+|\nabla_{\bar
x'}g_k(\bar x')|^2}\bigg)=\\ \inf\limits_{\sigma\in
\mathcal{F}(x)}P_{x,k}(\sigma,L)\,,
\end{equation}
where $P_{x,k}(\sigma,L)$ is defined by
\begin{multline}
\label{L2009limew03zeta71288888Contggiuuggyyyynew88789999vprop78899shtrihkkkllyhjyukjkkmmmjhgfkk}
P_{x,k}(\sigma,L):=\int\limits_{\R}\int\limits_{I_1^{N-1}}\frac{1}{L}F\bigg(p(-s_1/L,x)\Big\{(\vec
A\cdot\nabla v)^+(x)-(\vec A\cdot\nabla
v)^-(x)\Big\}\otimes\vec \nu(x)+L\nabla_y\{\vec A\cdot\nabla_y\sigma\}\big(Q_{x,k}(s)\big),\\
\,\Gamma(-s_1/L,x)+\vec
A\cdot\nabla_y\sigma\big(Q_{x,k}(s)\big),\,\vec B\cdot
v(x),\,\zeta\big(s_1,f^+(x), f^-(x)\big) \bigg)\,ds'ds_1\,,
\end{multline}
with $(s_1,s'):=s\in\R\times\R^{N-1}$ and the linear transformation
$Q_{x,k}(s)=Q_{x,k}(s_1,s_2,\ldots, s_N):\R^N\to\R^N$, defined by
\begin{equation}\label{chfhh2009y888shtrihkkkllbjgjklllllkkkkkkkkk}
Q_{x,k}(s):=s_1\vec\nu(x)+\sum\limits_{j=2}^{N}\frac{s_j\vec
p_j(x)}{\sqrt{1+|\nabla_{\bar x'}g_k(\bar x')|^2}}\,,
\end{equation}
and
\begin{multline}\label{L2009Ddef2hhhjjjj77788nulloooollkkkkkkkkkkkkk}
\mathcal{F}(x):=\Bigg\{ \sigma\in C^\infty(\R^N,\R^d):\;
\sigma(y)=0\;\text{ if }\; |y\cdot\vec\nu|\geq
\frac{1}{2\sqrt{1+|\nabla_{\bar x'}g_k(\bar x')|^2}}\;\\ \text{ and
}\;\sigma\bigg(y+\frac{\vec p_j(x)}{\sqrt{1+|\nabla_{\bar
x'}g_k(\bar x')|^2}}\bigg)=\sigma(y)\;\;\forall j=2,3,\ldots.
N\Bigg\}\,.
\end{multline}
Therefore, by \er{hgfhgjykukuioiipoookkkkkkkilhkjkjuykjjkkk} and by
Lemma \ref{mnogohypsi88kkkll88} we have
\begin{multline}
\label{L2009limew03zeta71288888Contggiuuggyyyynew88789999vprop78899shtrihkkklljhgfkk}
\sum\limits_{k=1}^{k_0}\inf\limits_{h_0\in\mathcal{P}(\mathcal{U}_k)}\bigg\{\int_{S'_k}P_{x,k}\big(\tilde
\lambda_k(\cdot,x),L_0\big)\,d\mathcal{H}^{N-1}(x)\bigg\} =\\
\int\limits_{\cup_{k=1}^{k_0}S'_k}\Bigg\{\inf\limits_{\sigma\in
\mathcal{Q}_0(x)}H_x\bigg(\sigma\,,\,L_0\sqrt{1+|\nabla_{\bar
x'}g_k(\bar x')|^2}\bigg)\Bigg\}\,d\mathcal{H}^{N-1}(x)\,,
\end{multline}
where $\mathcal{P}(\mathcal{U})$ is a set of all functions
$h_0(y,x')\in C^\infty(\R^N\times\R^{N-1},\R^d)$, satisfying
\begin{equation}\label{obnulenieghjgjjkk}
h_0(y,x')=0\;\text{ if }\;|y_1|\geq 1/2,\text{ and }\;h_0(y+\vec
e_j,x')=h_0(y,x')\;\;\forall j=2,\ldots, N\,,
\end{equation}
and
\begin{equation}\label{nositelkjlkjl}
\ov{\supp h_0(y,x')}\subset\R^N\times \mathcal{U}\,;
\end{equation}
and $\tilde \lambda_k(y,x)\in C^\infty(\R^N\times\R^{N-1},\R^d)$ is
given by
\begin{equation}\label{tildedefttshtrihkkkllghfkk}
\tilde \lambda_k(y,x):=h_0\Big(\big\{\bar y_1-\nabla_{\bar
x'}g_k(\bar x')\cdot \bar y',\bar y'\big\},\bar x'\Big)\,,
\end{equation}
where again $\bar x:=Z_k(x)$ and $\bar y=Z_k(y)$ (here $Z_k$ is the
appropriate relabeling).
Thus for every $k=1,2,\ldots k_0$ there exists
$h_k\in\mathcal{P}(\mathcal{U}_k)$ such that
\begin{multline}
\label{L2009limew03zeta71288888Contggiuuggyyyynew88789999vprop78899shtrihkkklljhgfkknjhkjjhjh}
\int\limits_{\cup_{k=1}^{k_0}S'_k}P_{x,k}\big(\lambda_k(\cdot,x),L_0\big)\,d\mathcal{H}^{N-1}(x)
-\int\limits_{\cup_{k=1}^{k_0}S'_k}\Bigg\{\inf\limits_{\sigma\in
\mathcal{Q}_0(x)}H_x\bigg(\sigma\,,\,L_0\sqrt{1+|\nabla_{\bar
x'}g_k(\bar
x')|^2}\bigg)\Bigg\}\,d\mathcal{H}^{N-1}(x)<\frac{\delta}{8}\,,
\end{multline}
where
\begin{equation}\label{tildedefttshtrihkkkllghfkkjjjjjjkk}
\lambda_k(y,x):=h_k\Big(\big\{\bar y_1-\nabla_{\bar x'}g_k(\bar
x')\cdot \bar y',\bar y'\big\},\bar x'\Big)\,.
\end{equation}
Plugging
\er{L2009limew03zeta71288888Contggiuuggyyyynew88789999vprop78899shtrihkkklljhgfkknjhkjjhjh}
into \er{hgfhgjykukuioiipoookkkkkkkilhkjkjuyk} we deduce
\begin{equation}
\label{L2009limew03zeta71288888Contggiuuggyyyynew88789999vprop78899shtrihkkklljhgfkknjhkjjhjhyiolkkkklkk}
\int\limits_{\cup_{k=1}^{k_0}S'_k}P_{x,k}\big(\lambda_k(\cdot,x),L_0\big)\,d\mathcal{H}^{N-1}(x)
-\int\limits_{\cup_{k=1}^{k_0}S'_k}\inf_{L\in(0,1)}\bigg\{\inf\limits_{\sigma\in
\mathcal{Q}(x)}H_x(\sigma,L)\bigg\}\,d\mathcal{H}^{N-1}(x)<\frac{\delta}{4}\,.
\end{equation}
Next consider a radial function $\theta_0(z')=\kappa(|z'|)\in
C^\infty_c(\R^{N-1},\R)$ such that $\supp \theta_0\subset\subset
B_1(0)$, $\theta_0\geq 0$ and $\int_{\R^{N-1}}\theta_0(z')dz'=1$.
Then for any $\e>0$ define the function
$g_{k,\e}(x'):\R^{N-1}\to\R$ by
\begin{equation}
\label{L2009psije12387878khkhhj}
g_{k,\e}(x'):=\frac{1}{\e^{N-1}}\int_{\R^{N-1}}\theta_0\Big(\frac{y'-x'}{\e}\Big)
g_k(y')dy'= \int_{\R^{N-1}}\theta_0(z')g_k(x'+\e
z')\,dz',\quad\forall x'\in\R^{N-1}\,.
\end{equation}
Next set
\begin{equation}\label{nulltoonellllkkkkkkkk}
h^{(L_0)}_k(y,x'):=\frac{1}{L_0}h_k(L_0y,x')\,,
\end{equation}
then $h^{(L_0)}_k\in C^\infty(\R^N\times\R^{N-1},\R^d)$ satisfy
\begin{equation}\label{obnulenie1jkljkkkk}
h^{(L_0)}_k(y,x')=0\;\text{ if }\;|y_1|\geq 1/(2L_0),\text{ and
}\;h^{(L_0)}_k\big(y+(1/L_0)\vec
e_j,x'\big)=h^{(L_0)}_k(y,x')\;\;\forall j=2,\ldots, N\,,
\end{equation}
and
\begin{equation}\label{nositel1hhuhuhhihhuh}
\ov{\supp h^{(L_0)}_k(y,x')}\subset\R^N\times \mathcal{U}_k\,.
\end{equation}
For any $\e>0$ define the function $\gamma_{k,\e}(x)\in
C^\infty(\R^N,\R^d)$ by
\begin{equation}
\label{L2009test1288kjojjjllkjjjkkk} \gamma_{k,\e}(x):=\e
h^{(L_0)}_k\bigg(\Big(\frac{\bar x_1-g_{k,\e}(\bar
x')}{\e},\frac{\bar x'}{\e}\Big),\bar x'\bigg)\,,
\end{equation}
where, as before, $\bar x:=Z_k(x)$. Next clearly for every
$k=1,2\ldots ,k_0$ there exists an open set $G_k\subset\subset\O$
such that for every $k$, $\ov{S'_k}\subset\subset G_k$ and
$\ov{G_j}\cap \ov{G_k}=\emptyset$ if $j\neq k$. Thus there exists
$\e_0\in(0,1)$ such that if $0<\e<\e_0$ then $\supp
h^{(L_0)}_k\subset\subset G_k$. Therefore, for $0<\e<\e_0$ we can
define $\gamma_{\e}(x)\in C^\infty(\R^N,\R^d)$ by
\begin{equation}
\label{L2009test1288kjojjjllkjjjkkkghghkk}
\gamma_{\e}(x):=\begin{cases}\gamma_{k,\e}(x)\quad\quad\text{if}\;\;x\in
G_k\quad\forall k=1,2,\ldots
k_0\,,\\0\quad\quad\quad\quad\;\;\text{otherwise.}\end{cases}\,,
\end{equation}
Then we can set
\begin{equation}\label{nositel1hhuhuhhihhujyiohkkkkkhhhkk}
v_\e(x):=\psi_\e(x)+\gamma_\e(x)\quad\quad\forall x\in\R^N\,.
\end{equation}
Thus, as before in the conditions of Proposition \ref{mnogohypsi88},
$\lim_{\e\to0^+}\vec A\cdot\nabla v_\e=\vec A\cdot \nabla v$ in
$L^p(\R^N,\R^m)$, $\lim_{\e\to0^+} \e\nabla\{\vec A\cdot \nabla
v_\e\}=0$ in $L^{p}(\R^N,\R^{m\times N})$, $\lim_{\e\to0^+}\vec
B\cdot v_\e=\vec B\cdot v$ in $W^{1,p}(\R^N,\R^k)$,
$\lim_{\e\to0^+}(\vec B\cdot v_\e-\vec B\cdot v)/\e=0$ in
$L^{p}(\R^N,\R^k)$  for every $p\geq 1$ and moreover,
 $\vec A\cdot\nabla v_\e$, $\e\nabla\{\vec A\cdot\nabla
v_\e\}$, $\vec B\cdot v_\e$ and $\nabla\{\vec B\cdot v_\e\}$ are
bounded in $L^\infty$ sequences, there exists a compact
$K\subset\subset\O$ such that $v_\e(x)=\psi_\e(x)$ for every
$0<\e<1$ and every $x\in\R^N\setminus K$ and we have
\er{bjdfgfghljjklkjkllllkkllkkkkkkllllkkllhffhhfh}.

 Next clearly, for $0<\e<\e_0$ we have
\begin{multline}\label{gugugukkkkkkkkkkkk}
\int\limits_{\O}\frac{1}{\e} F\Big(\e\nabla\{\vec A\cdot\nabla
v_{\e}\},\,\vec A\cdot\nabla v_{\e},\,\vec B\cdot
v_{\e},\,f\Big)\,dx-\int\limits_{\O}\frac{1}{\e}
F\Big(\e\nabla\{\vec A\cdot\nabla \psi_{\e}\},\,\vec A\cdot\nabla
\psi_{\e},\,\vec B\cdot
\psi_{\e},\,f\Big)\,dx\\=\sum\limits_{k=1}^{k_0}\Bigg\{\int\limits_{G_k}\frac{1}{\e}
F\Big(\e\nabla\{\vec A\cdot\nabla v_{\e}\},\,\vec A\cdot\nabla
v_{\e},\,\vec B\cdot
v_{\e},\,f\Big)\,dx-\int\limits_{G_k}\frac{1}{\e}
F\Big(\e\nabla\{\vec A\cdot\nabla \psi_{\e}\},\,\vec A\cdot\nabla
\psi_{\e},\,\vec B\cdot \psi_{\e},\,f\Big)\,dx\Bigg\}\\=
\sum\limits_{k=1}^{k_0}\Bigg\{\int\limits_{\O}\frac{1}{\e}
F\Big(\e\nabla\big\{\vec A\cdot\nabla
(\psi_\e+\gamma_{k,\e})\big\},\,\vec A\cdot\nabla
(\psi_\e+\gamma_{k,\e}),\,\vec B\cdot
(\psi_\e+\gamma_{k,\e}),\,f\Big)\,dx\\-\int\limits_{\O}\frac{1}{\e}
F\Big(\e\nabla\{\vec A\cdot\nabla \psi_{\e}\},\,\vec A\cdot\nabla
\psi_{\e},\,\vec B\cdot \psi_{\e},\,f\Big)\,dx\Bigg\}\,.
\end{multline}
On the other hand, by Proposition \ref{mnogohypsi88}, for every
$k=1,2\ldots,k_0$
\begin{multline}\label{guguguhkhkhkhkhkpppp}
\lim\limits_{\e\to 0^+}\Bigg\{\int\limits_{\O}\frac{1}{\e}
F\Big(\e\nabla\big\{\vec A\cdot\nabla
(\psi_\e+\gamma_{k,\e})\big\},\,\vec A\cdot\nabla
(\psi_\e+\gamma_{k,\e}),\,\vec B\cdot
(\psi_\e+\gamma_{k,\e}),\,f\Big)\,dx\\-\int\limits_{\O}\frac{1}{\e}
F\Big(\e\nabla\{\vec A\cdot\nabla \psi_{\e}\},\,\vec A\cdot\nabla
\psi_{\e},\,\vec B\cdot v_{\e},\,f\Big)\,dx\Bigg\}=
\int\limits_{S'_k}P_{x,k}\big(\lambda_k(\cdot,x)\big)\,d\mathcal{H}^{N-1}(x)-\int\limits_{S'_k}P_{x,k}\big(0\big)\,d\mathcal{H}^{N-1}(x)
\,.
\end{multline}
Plugging \er{guguguhkhkhkhkhkpppp} into \er{gugugukkkkkkkkkkkk} we
obtain
\begin{multline}\label{gugugukkkkkkkkkkkkgihyilllkk}
\lim\limits_{\e\to 0^+}\Bigg\{\int\limits_{\O}\frac{1}{\e}
F\Big(\e\nabla\{\vec A\cdot\nabla v_{\e}\},\,\vec A\cdot\nabla
v_{\e},\,\vec B\cdot
v_{\e},\,f\Big)\,dx-\int\limits_{\O}\frac{1}{\e}
F\Big(\e\nabla\{\vec A\cdot\nabla \psi_{\e}\},\,\vec A\cdot\nabla
\psi_{\e},\,\vec B\cdot v_{\e},\,f\Big)\,dx\Bigg\}=\\
\int\limits_{\cup_{k=1}^{k_0}S'_k}P_{x,k}\big(\lambda_k(\cdot,x),L_0\big)\,d\mathcal{H}^{N-1}(x)-\int\limits_{\cup_{k=1}^{k_0}S'_k}P_{x,k}\big(0,L_0\big)\,d\mathcal{H}^{N-1}(x)\,.
\end{multline}
On the other hand by Theorem \ref{vtporbound4},
\er{L2009limew03zeta71288888Contggiuuggyyyynew88789999vprop78899shtrihkkkllkkklklklkkkkjkuykolllkklloooohhhkkkll}
and
\er{L2009limew03zeta71288888Contggiuuggyyyynew88789999vprop78899shtrihkkkllyhjyukjkkmmmjhgfkk}
we have
\begin{multline}\label{gugugukkkkkkkkkkkkgihyilllkkmkjkjooouuolllkkk}
\lim\limits_{\e\to 0^+}\Bigg\{\int\limits_{\O}\frac{1}{\e}
F\Big(\e\nabla\{\vec A\cdot\nabla \psi_{\e}\},\,\vec A\cdot\nabla
\psi_{\e},\,\vec B\cdot v_{\e},\,f\Big)\,dx\Bigg\}=
\int\limits_{\O\cap J_{\vec A\cdot\nabla
v}}P_{x,k}\big(0,L_0\big)\,d\mathcal{H}^{N-1}(x)\,.
\end{multline}
Thus combining \er{gugugukkkkkkkkkkkkgihyilllkkmkjkjooouuolllkkk}
and \er{gugugukkkkkkkkkkkkgihyilllkk} together with
\er{L2009limew03zeta71288888Contggiuuggyyyynew88789999vprop78899shtrihkkkllkkklklklkkkkjkuykolllkklloooohhhkkkll},
\er{L2009limew03zeta71288888Contggiuuggyyyynew88789999vprop78899shtrihkkkllyhjyukjkkmmmjhgfkk}
and the fact that $P_{x,k}(0,L)=0$ if $x\notin \O\cap J_{\vec
A\cdot\nabla v}$, we obtain
\begin{multline}\label{gugugukkkkkkkkkkkkgihyilllkkvgnghjkkkkjjjjllljjjkk}
\lim\limits_{\e\to 0^+}\Bigg\{\int\limits_{\O}\frac{1}{\e}
F\Big(\e\nabla\{\vec A\cdot\nabla v_{\e}\},\,\vec A\cdot\nabla
v_{\e},\,\vec B\cdot
v_{\e},\,f\Big)\,dx\Bigg\}=\\
\int\limits_{\cup_{k=1}^{k_0}S'_k}P_{x,k}\big(\lambda_k(\cdot,x),L_0\big)\,d\mathcal{H}^{N-1}(x)+\int\limits_{\O\cap
J_{\vec A\cdot\nabla
v}\setminus\cup_{k=1}^{k_0}S'_k}P_{x,k}\big(0,L_0\big)\,d\mathcal{H}^{N-1}(x)=\\
\int\limits_{\cup_{k=1}^{k_0}S'_k}P_{x,k}\big(\lambda_k(\cdot,x),L_0\big)\,d\mathcal{H}^{N-1}(x)+
 \int_{\O\cap J_{\vec A\cdot\nabla
v}\setminus\cup_{k=1}^{k_0}S'_k}\Bigg\{\\
\int\limits_{\R}F\bigg(p(t,x)\Big\{(\vec A\cdot\nabla v)^+(x)-(\vec
A\cdot\nabla v)^-(x)\Big\}\otimes\vec \nu(x),
\,\Gamma(t,x),\,\vec B\cdot v(x),\,\zeta\big(t,f^-(x),f^+(x)\big) \bigg)\,dt
\Bigg\}\,d\mathcal{H}^{N-1}(x) \,.
\end{multline}
Therefore, by \er{ugguuggugugugggjgghghghghgh},
\er{L2009limew03zeta71288888Contggiuuggyyyynew88789999vprop78899shtrihkkklljhgfkknjhkjjhjhyiolkkkklkk}
and \er{gugugukkkkkkkkkkkkgihyilllkkvgnghjkkkkjjjjllljjjkk} we
obtain
\begin{multline}\label{gugugukkkkkkkkkkkkgihyilllkkvgnghjkkkkjjjjllljjjkkthytjukkk}
\lim\limits_{\e\to 0^+}\Bigg\{\int\limits_{\O}\frac{1}{\e}
F\Big(\e\nabla\{\vec A\cdot\nabla v_{\e}\},\vec A\cdot\nabla
v_{\e},\vec B\cdot v_{\e},f\Big)dx\Bigg\}<\frac{3\delta}{8}+
\int\limits_{\cup_{k=1}^{k_0}S'_k}\inf_{L\in(0,1)}\bigg\{\inf\limits_{\sigma\in
\mathcal{Q}(x)}H_x(\sigma,L)\bigg\}d\mathcal{H}^{N-1}(x)+\\
 \int_{\cup_{k=1}^{k_0}S_k\setminus\cup_{k=1}^{k_0}S'_k}\Bigg\{\int\limits_{\R}F\bigg(p(t,x)\Big\{(\vec
A\cdot\nabla v)^+(x)-(\vec A\cdot\nabla v)^-(x)\Big\}\otimes\vec
\nu(x), \Gamma(t,x),\vec B\cdot
v(x),\zeta\big(t,f^-(x),f^+(x)\big) \bigg)dt\\
\Bigg\}\,d\mathcal{H}^{N-1}(x) \,.
\end{multline}
Finally plugging \er{ghugtyffygdch} and
\er{L2009limew03zeta71288888Contggiuuggyyyynew88789999vprop78899shtrihkkkllkkklklklkkkkjkuykolllkklloooohhhkkkll}
into
\er{gugugukkkkkkkkkkkkgihyilllkkvgnghjkkkkjjjjllljjjkkthytjukkk} we
infer
\begin{multline}\label{gugugukkkkkkkkkkkkgihyilllkkvgnghjkkkkjjjjllljjjkkthytjukkkppopoop}
\lim\limits_{\e\to 0^+}\Bigg\{\int\limits_{\O}\frac{1}{\e}
F\Big(\e\nabla\{\vec A\cdot\nabla v_{\e}\},\,\vec A\cdot\nabla
v_{\e},\,\vec B\cdot v_{\e},\,f\Big)\,dx\Bigg\}<\\
\int\limits_{\cup_{k=1}^{k_0}S'_k}\inf_{L\in(0,1)}\bigg\{\inf\limits_{\sigma\in
\mathcal{Q}(x)}H_x(\sigma,L)\bigg\}\,d\mathcal{H}^{N-1}(x)+\frac{\delta}{2}
\leq \int\limits_{\O\cap J_{\vec A\cdot\nabla
v}}\bigg\{\inf\limits_{\sigma\in
\mathcal{Q}(x),L\in(0,1)}H_x(\sigma,L)\bigg\}\,d\mathcal{H}^{N-1}(x)+\frac{\delta}{2}\,,
\end{multline}
where the last inequality we infer since if $x\notin \O\cap J_{\vec
A\cdot\nabla v}$ then again by
\er{L2009limew03zeta71288888Contggiuuggyyyynew88789999vprop78899shtrihkkkllkkklklklkkkkjkuykolllkklloooohhhkkkll}
we have
$$0\leq \inf\limits_{\sigma\in
\mathcal{Q}(x)}H_x(\sigma,L)\leq H_x(0,L)=0\,.$$ This completes the
proof.
\end{proof}

\begin{theorem}\label{vtporbound4ghkfgkhk}
Let $\Omega\subset\R^N$ be an open set.
Furthermore, let $\vec A\in \mathcal{L}(\R^{d\times N};\R^m)$, $\vec
B\in \mathcal{L}(\R^{d};\R^k)$ and let $F\in C^1(\R^{m\times
N}\times \R^m\times \R^k\times\R^q,\R)$, satisfying $F\geq 0$. Let
$f\in BV_{loc}(\R^N,\R^q)\cap L^\infty$ and
$v\in\mathcal{D}'(\R^N,\R^d)$ be such that $\vec A\cdot\nabla v\in
BV(\R^N,\R^m)\cap L^\infty(\R^N,\R^m)$ and $\vec B\cdot v\in
Lip(\R^N,\R^k)\cap W^{1,1}(\R^N,\R^k)\cap L^\infty(\R^N,\R^k)$,
$\|D(\vec A\cdot \nabla v)\|(\partial\Omega)=0$ and $F\big(0,\{\vec
A\cdot\nabla v\}(x),\{\vec B\cdot v\}(x),f(x)\big)=0$ a.e.~in
$\Omega$. Moreover, assume that for $\mathcal{H}^{N-1}$-a.e. $x\in
J_{\vec A\cdot\nabla v}\cap\O$ there exists a distribution
$\gamma_x(\cdot)\in \mathcal{D}'(\R^N,\R^d)$ such that
\begin{equation}\label{hhjhhkhkhjhghfhlllkkkjjjhhhgffdddd}
\{\vec A\cdot\nabla\gamma_x\}(z)=\begin{cases}(\vec A\cdot\nabla
v)^+(x)\quad\quad\text{if}\quad z\cdot\vec\nu(x)\geq 0\,,\\
(\vec A\cdot\nabla v)^-(x)\quad\quad\text{if}\quad z\cdot\vec\nu(x)<
0\,
\end{cases}
\end{equation}
(with $\vec\nu(x)$ denoting the orientation vector of $J_{\vec
A\cdot\nabla v}$). Finally we assume that for
$\mathcal{H}^{N-1}$-a.e. $x\in J_{\vec A\cdot\nabla v}\cap\O$ for
every system $\{\vec k_1(x),\vec k_2(x),\ldots \vec k_N(x)\}$ of
linearly independent vectors in $\R^N$ satisfying $\vec
k_1(x)=\vec\nu(x)$ and $\vec k_j(x)\cdot\vec\nu(x)=0$ for $j\geq 2$,
and
for every $\xi(z)\in\mathcal{W}_x(\vec k_1,\vec k_2,\ldots \vec
k_N)$ there exists $\zeta(z)\in\mathcal{W}'_x(\vec k_1,\vec
k_2,\ldots \vec k_N)$ such that $\vec A\cdot\nabla\xi(z)\equiv\vec
A\cdot\nabla\zeta(z)$, where
\begin{multline}\label{L2009Ddef2hhhjjjj77788hhhkkkkllkjjjjkkkhhhhffggdddkkk}
\mathcal{W}_x\big(\vec k_1(x),\vec k_2(x),\ldots \vec k_N(x)\big):=
\bigg\{ u\in
C^\infty(\R^N,\R^d):\;\vec A\cdot\nabla
u(y)=0\;\text{ if }\;|y\cdot\vec\nu(x)|\geq 1/2,\\
\text{ and }\;\vec A\cdot\nabla u\big(y+\vec k_j(x)\big)=\vec
A\cdot\nabla u(y)\;\;\forall j=2,3,\ldots, N\bigg\}\,,
\end{multline}
and
\begin{multline}\label{L2009Ddef2hhhjjjj77788nullkkkkkkjjjjjkkk}
\mathcal{W}'_x\big(\vec k_1(x),\vec k_2(x),\ldots \vec
k_N(x)\big):=\\
\bigg\{ u\in
C^\infty(\R^N,\R^d)\cap L^\infty(\R^N,\R^d)\cap Lip(\R^N,\R^d):\;
\vec A\cdot\nabla u(y)=0\;\text{ if }\; |y\cdot\vec\nu(x)|\geq 1/2\;\\
\text{ and }\;
u\big(y+\vec k_j(x)\big)=u(y) \;\;\forall j=2,3,\ldots, N\bigg\}\,.
\end{multline}
Then, for $\eta\in \mathcal{V}_0$,
for every $\delta>0$  there exist a
sequence $\{v_\e\}_{0<\e<1}\subset C^\infty(\R^N,\R^d)$
such that $\lim_{\e\to0^+}\vec A\cdot\nabla v_\e=\vec A\cdot \nabla
v$ in $L^p(\R^N,\R^m)$, $\lim_{\e\to0^+} \e\nabla\{\vec A\cdot
\nabla v_\e\}=0$ in $L^{p}(\R^N,\R^{m\times N})$,
$\lim_{\e\to0^+}\vec B\cdot v_\e=\vec B\cdot v$ in
$W^{1,p}(\R^N,\R^k)$, $\lim_{\e\to0^+}(\vec B\cdot v_\e-\vec B\cdot
v)/\e=0$ in $L^{p}(\R^N,\R^k)$ for every $p\geq 1$ and
\begin{multline}
\label{L2009limew03zeta71288888Contggiuuggyyyynew88789999vprop78899shtrihkkklljkhkhgghhhjhhhjkkk}
0\leq\lim\limits_{\e\to 0}\int\limits_{\O}\frac{1}{\e}
F\Big(\e\nabla\{\vec A\cdot\nabla v_{\e}\},\,\vec A\cdot\nabla
v_{\e},\,\vec B\cdot v_{\e},\,f\Big)\,dx\\-\int\limits_{\O\cap
J_{\vec A\cdot\nabla
v}}\bigg\{\inf\limits_{L>0}\bigg(\inf\limits_{\sigma\in
\mathcal{W}_0(x,\vec k_1,\ldots,\vec
k_N)}E_x\big(\sigma(\cdot),L\big)\bigg)
\bigg\}\,d\mathcal{H}^{N-1}(x)<\delta\,,
\end{multline}
where
\begin{multline}
\label{L2009limew03zeta71288888Contggiuuggyyyynew88789999vprop78899shtrihkkkllyhjyukjkkmmmklklklhhhhkkffgghhjjjkkkllkk}
E_x\big(\sigma(\cdot),L\big):=
\int\limits_{I_{N}}\frac{1}{L}F\bigg(L\nabla\{\vec
A\cdot\nabla\sigma\}\big(S_{x}(s)\big), \,\{\vec
A\cdot\nabla\sigma\}\big(S_{x}(s)\big),\,\vec B\cdot
v(x),\,\zeta\big(s_1,f^+(x),f^-(x)\big) \bigg)\,ds\,,
\end{multline}
\begin{equation}\label{vhjgkjghklhlhj}
\zeta(s,a,b):=
\begin{cases}
a\quad\text{if}\;\; s>0\,,\\
b\quad\text{if}\;\;s<0\,,
\end{cases}
\end{equation}
\begin{multline}\label{L2009Ddef2hhhjjjj77788hhhkkkkllkjjjjkkkhhhhffggdddkkkgjhik}
\mathcal{W}_0\big(x,\vec k_1(x),\vec k_2(x),\ldots \vec
k_N(x)\big):=\\
\bigg\{u\in \mathcal{D}'(\R^N,\R^d): \vec A\cdot\nabla u\in
C^1(\R^N,\R^m),\;\;\vec A\cdot\nabla u(y)=(\vec A\cdot\nabla
v)^-(x)\;\text{ if }\;y\cdot\vec\nu(x)\leq-1/2,\\
\vec A\cdot\nabla u(y)=(\vec A\cdot\nabla v)^+(x)\;\text{ if }\;
y\cdot\vec\nu(x)\geq 1/2\;\text{ and }\;\vec A\cdot\nabla
u\big(y+\vec k_j(x)\big)=\vec A\cdot\nabla u(y)\;\;\forall
j=2,3,\ldots, N\bigg\}\,,
\end{multline}
\begin{equation}\label{cubidepgghkkkkllllllljjjkkkkkjjjhhhhlkkkffffggggddd}
I_N:= \Big\{s\in\R^N:\; -1/2<s_j<1/2\quad\forall
j=1,2,\ldots, N\Big\}\,,
\end{equation}
$\{\vec k_1(x),\vec k_2(x),\ldots \vec k_N(x)\}$ is an arbitrary
system of linearly independent vectors in $\R^N$ satisfying $\vec
k_1(x)=\vec\nu(x)$ and $\vec k_j(x)\cdot\vec\nu(x)=0$ for $j\geq 2$,
the linear transformation $S_x(s)=S_{x}(s_1,s_2,\ldots,
s_N):\R^N\to\R^N$ is defined by
\begin{equation}\label{chfhh2009y888shtrihkkkllbjgjkllkjkjjjjujkkllkkk}
S_{x}(s):=\sum\limits_{j=1}^{N}s_j\,\vec k_j(x)\,,
\end{equation}
and we assume that the orientation of $J_f$ coincides with the
orientation of $J_{\vec A\cdot\nabla v}$ $\mathcal{H}^{N-1}$ a.e. on
$J_f\cap J_{\vec A\cdot\nabla v}$. Moreover, $\vec A\cdot\nabla
v_\e$, $\e\nabla\{\vec A\cdot\nabla v_\e\}$, $\vec B\cdot v_\e$ and
$\nabla\{\vec B\cdot v_\e\}$ are bounded in $L^\infty$ sequences,
there exists a compact $K=K_\delta\subset\subset\O$ such that
$v_\e(x)=\psi_\e(x)$ for every $0<\e<1$ and every $x\in\R^N\setminus
K$, where
\begin{equation}\label{a4543chejhgufydfdtkkuuujk}
\psi_\e(x):=\frac{1}{\e^N}\Big<\eta\Big(\frac{y-x}{\e}\Big),v(y)\Big>\,,
\end{equation}
and
\begin{equation}\label{bjdfgfghljjklkjkllllkkllkkkkkkllllkkllhffhhfhftfffffhhhhkkkkk}
\limsup\limits_{\e\to 0^+}\Bigg|\frac{1}{\e}\bigg(\int_\O\vec
A\cdot\nabla v_\e(x)\,dx-\int_\O\{\vec A\cdot \nabla
v\}(x)\,dx\bigg)\Bigg|<+\infty\,,
\end{equation}
\end{theorem}
\begin{proof}
Let $\eta$ and $\psi_\e$
be the same as in Theorem \ref{vtporbound4} and
\begin{equation}\label{matrixgamadef7845389yuj9hfhfjklk}
\Gamma(t,x):=\bigg(\int_{-\infty}^t p(s,x)ds\bigg)\{\vec
A\cdot\nabla v\}^-(x)+\bigg(\int_t^{+\infty} p(s,x)ds\bigg)\{\vec
A\cdot\nabla v\}^+(x)\,,
\end{equation}
where $p(t,x)$ is defined by
\begin{equation}\label{defphfghjfgjkllk}
p(t,x):=\int_{H^0_{\vec \nu(x)}}\eta(t\vec \nu(x)+y)\,d \mathcal
H^{N-1}(y)\,.
\end{equation}
Then, by Lemma \ref{mnogohypsi88kkkll88lpppopkook} and Lemma
\ref{lijkkhnjjhjh}, for every $\delta>0$ there exist $N$
Borel-measurable functions $\vec p_j(x):J_{\vec A\cdot\nabla
v}\to\R^N$ for $1\leq j\leq N$, such that $\vec p_1(x):=\vec
\nu(x)$, $\vec p_j(x)\cdot\vec\nu(x)=0$ for $2\leq j\leq N$ and
$\{\vec p_j(x)\}_{1\leq j\leq N}$ is linearly independent system of
vectors for every $x$, and there exists a sequence
$\{v_\e\}_{0<\e<1}\subset C^\infty(\R^N,\R^d)$
such that $\lim_{\e\to0^+}\vec A\cdot\nabla v_\e=\vec A\cdot \nabla
v$ in $L^p(\R^N,\R^m)$, $\lim_{\e\to0^+} \e\nabla\{\vec A\cdot
\nabla v_\e\}=0$ in $L^{p}(\R^N,\R^{m\times N})$,
$\lim_{\e\to0^+}\vec B\cdot v_\e=\vec B\cdot v$ in
$W^{1,p}(\R^N,\R^k)$, $\lim_{\e\to0^+}(\vec B\cdot v_\e-\vec B\cdot
v)/\e=0$ in $L^{p}(\R^N,\R^k)$  for every $p\geq 1$ and
\begin{equation}
\label{L2009limew03zeta71288888Contggiuuggyyyynew88789999vprop78899shtrihkkklljkhkhgghkklkjjkjk}
0\leq\lim\limits_{\e\to 0}\int\limits_{\O}\frac{1}{\e}
F\Big(\e\nabla\{\vec A\cdot\nabla v_{\e}\},\,\vec A\cdot\nabla
v_{\e},\,\vec B\cdot v_{\e},\,f\Big)\,dx-\int\limits_{\O\cap J_{\vec
A\cdot\nabla v}}\bigg\{\inf\limits_{\sigma\in
\mathcal{Q}(x),L\in(0,1)}H_x\big(\sigma(\cdot),L\big)
\bigg\}\,d\mathcal{H}^{N-1}(x)<\delta\,,
\end{equation}
where
\begin{multline}
\label{L2009limew03zeta71288888Contggiuuggyyyynew88789999vprop78899shtrihkkkllyhjyukjkkmmmklklklhhhhkkbjjjhhjk}
H_x(\sigma,L):=
\int\limits_{I_{N}}\frac{1}{L}F\bigg(p(-s_1/L,x)\Big\{(\vec
A\cdot\nabla v)^+(x)-(\vec A\cdot\nabla
v)^-(x)\Big\}\otimes\vec \nu(x)+L\nabla\{\vec A\cdot\nabla\sigma\}\big(T_{x}(s)\big),\\
\,\Gamma(-s_1/L,x)+\vec A\cdot\nabla\sigma\big(T_{x}(s)\big),\,\vec
B\cdot v(x),\,\zeta\big(s_1,f^+(x),f^-(x)\big) \bigg)\,ds\,,
\end{multline}
\begin{multline}\label{L2009Ddef2hhhjjjj77788nullkkkooooollhgghhgkk}
\mathcal{Q}(x):=\mathcal{D}_1\big(\vec A,\vec\nu(x),\vec p_2(x),\vec
p_3(x),\ldots,\vec p_{N}(x)\big):=\\ \bigg\{ \sigma\in
C^\infty(\R^N,\R^d)\cap L^\infty(\R^N,\R^d)\cap Lip(\R^N,\R^d):\;
\vec A\cdot\nabla \sigma(y)=0\;\text{ if }\; |y\cdot\vec\nu(x)|\geq 1/2\;\\
\text{ and }\;
\sigma\big(y+\vec p_j(x)\big)=\sigma(y) \;\;\forall j=2,3,\ldots,
N\bigg\}\,,
\end{multline}
and the linear transformation $T_x(s)=T_{x}(s_1,s_2,\ldots,
s_N):\R^N\to\R^N$ is defined by
\begin{equation}\label{chfhh2009y888shtrihkkkllbjgjklljgkgjkjk}
T_{x}(s):=\sum\limits_{j=1}^{N}s_j\,\vec p_j(x)\,,
\end{equation}
Moreover, $\vec A\cdot\nabla v_\e$, $\e\nabla\{\vec A\cdot\nabla
v_\e\}$, $\vec B\cdot v_\e$ and $\nabla\{\vec B\cdot v_\e\}$ are
bounded in $L^\infty$ sequences, there exists a compact
$K\subset\subset\O$ such that $v_\e(x)=\psi_\e(x)$ for every
$0<\e<1$ and every $x\in\R^N\setminus K$ and we have
\er{bjdfgfghljjklkjkllllkkllkkkkkkllllkkllhffhhfhftfffffhhhhkkkkk}.

 So we only need to prove that
\begin{multline}\label{gkhljknmjkkjjhjffttryt}
\int\limits_{\O\cap J_{\vec A\cdot\nabla
v}}
\inf\limits_{L>0}\bigg(\inf\limits_{\sigma\in
\mathcal{W}_0(x,\vec k_1,\ldots,\vec
k_N)}E_x\big(\sigma(\cdot),L\big)\bigg)
\,d\mathcal{H}^{N-1}(x)=\int\limits_{\O\cap J_{\vec A\cdot\nabla
v}}\bigg\{\inf\limits_{\sigma\in
\mathcal{Q}(x),L\in(0,1)}H_x\big(\sigma(\cdot),L\big)
\bigg\}d\mathcal{H}^{N-1}(x),
\end{multline}
for any choice of the system $\{\vec k_1(x),\vec k_2(x),\ldots \vec
k_N(x)\}$ of linearly independent vectors in $\R^N$ satisfying $\vec
k_1(x)=\vec\nu(x)$ and $\vec k_j(x)\cdot\vec\nu(x)=0$ for $j\geq 2$.
By Proposition \ref{L2009cffgfProp8hhhjjj8887788} we have that the
left hand side of \er{gkhljknmjkkjjhjffttryt} is independent on the
choice of the system $\{\vec k_1(x),\vec k_2(x),\ldots \vec
k_N(x)\}$. Therefore, from now we may assume that $\vec k_j(x)=\vec
p_j(x)$ for every $x$ and $j$. Thus in particular $S_x=T_x$ and
$\mathcal{Q}(x)=\mathcal{W}'_x(\vec k_1,\vec k_2,\ldots \vec k_N)$.
On the other hand, for $\mathcal{H}^{N-1}$-a.e. $x\in J_{\vec
A\cdot\nabla v}\cap\O$ we have $\xi(z)\in\mathcal{W}_x(\vec k_1,\vec
k_2,\ldots \vec k_N)$ there exists $\zeta(z)\in\mathcal{W}'_x(\vec
k_1,\vec k_2,\ldots \vec k_N)$ such that $\vec
A\cdot\nabla\xi(z)\equiv\vec A\cdot\nabla\zeta(z)$, and therefore we
have
\begin{multline}\label{gkhljknmjkkjjhjffttrytrfrh}
\int\limits_{\O\cap J_{\vec A\cdot\nabla
v}}
\inf\limits_{\sigma\in
\mathcal{Q}(x),L\in(0,1)}H_x\big(\sigma(\cdot),L\big)\,
d\mathcal{H}^{N-1}(x)
=\int\limits_{\O\cap J_{\vec A\cdot\nabla
v}}
\inf\limits_{L\in(0,1)}\bigg(\inf\limits_{\sigma\in
\mathcal{W}_x(\vec k_1,\vec k_2,\ldots \vec
k_N)}H_x\big(\sigma(\cdot),L\big)
\bigg)
d\mathcal{H}^{N-1}(x).
\end{multline}
Next for $\mathcal{H}^{N-1}$-a.e. $x\in \O\cap J_{\vec A\cdot\nabla
v}$ there exists a distribution $\gamma_x(\cdot)\in
\mathcal{D}'(\R^N,\R^d)$ such that we have
\er{hhjhhkhkhjhghfhlllkkkjjjhhhgffdddd}. For any $\e>0$ and any
fixed $z\in\R^N$ set
\begin{equation}\label{a4543chejhgufydfdtjhyjikkk}
\bar\gamma_{x}(z):=\Big<\eta(y-z),\gamma_x(y)\Big>
\end{equation}
(see notations and definitions in the beginning of the paper). Then
$\bar\gamma_{x}(z)\in C^\infty(\R^N,\R^d)$. Moreover clearly
\begin{equation}\label{a4543chethtjkkkk}
\vec A\cdot \{\nabla \bar\gamma_{x}(z)\}=\int_{\R^N}\eta(y-z)
\cdot\{\vec A\cdot\nabla \gamma_x\}(y)\,dy=
\int_{\R^N}\eta(y)\cdot\big\{\vec A\cdot\nabla
\gamma_x\big\}(z+y)\,dy.
\end{equation}
Plugging it into \er{hhjhhkhkhjhghfhlllkkkjjjhhhgffdddd} we deduce
\begin{equation}\label{a4543chethtjkkkkgjhkjhukjjjjkkkk}
\vec A\cdot \{\nabla
\bar\gamma_{x}(z)\}=\Gamma\big(-z\cdot\vec\nu(x),x\big)\,.
\end{equation}
Thus for every $L\in(0,1)$,  the function
$\bar\gamma_{x,L}(z):=L\bar\gamma_{x}(z/L)$ belongs to
$\mathcal{W}'_0\big(x,\vec k_1(x),\vec k_2(x),\ldots \vec
k_N(x)\big)$
where
\begin{multline}\label{L2009Ddef2hhhjjjj77788hhhkkkkllkjjjjkkkhhhhffggdddkkkgjhikkjjkh}
\mathcal{W}'_0\big(x,\vec k_1(x),\vec k_2(x),\ldots \vec
k_N(x)\big):=\\
\bigg\{u\in C^\infty(\R^N,\R^d):
\;\;\vec A\cdot\nabla u(y)=(\vec A\cdot\nabla
v)^-(x)\;\text{ if }\;y\cdot\vec\nu(x)\leq-1/2,\\
\vec A\cdot\nabla u(y)=(\vec A\cdot\nabla v)^-(x)\;\text{ if }\;
y\cdot\vec\nu(x)\geq 1/2\;\text{ and }\;\vec A\cdot\nabla
u\big(y+\vec k_j(x)\big)=\vec A\cdot\nabla u(y)\;\;\forall j\geq
2\bigg\}\,.
\end{multline}
Moreover, $\vec A\cdot \{\nabla
\bar\gamma_{x,L}\}\big(T_x(s)\big)=\Gamma\big(-s_1/L,x\big)$ and
$\nabla\big(\vec A\cdot \{\nabla
\bar\gamma_{x,L}\}\big)\big(T_x(s)\big)=p(-s_1/L,x)\big\{(\vec
A\cdot\nabla v)^+(x)-(\vec A\cdot\nabla v)^-(x)\big\}\otimes\vec
\nu(x)$. Thus
\begin{multline}\label{gkhljknmjkkjjhjffttrytrfrhjhjhjhk}
\int\limits_{\O\cap J_{\vec A\cdot\nabla
v}}\bigg\{\inf\limits_{L\in(0,1)}\bigg(\inf\limits_{\sigma\in
\mathcal{W}_x(\vec k_1,\vec k_2,\ldots \vec
k_N)}H_x\big(\sigma(\cdot),L\big)
\bigg)\bigg\}\,d\mathcal{H}^{N-1}(x)\\=\int\limits_{\O\cap J_{\vec
A\cdot\nabla
v}}\bigg\{\inf\limits_{L\in(0,1)}\bigg(\inf\limits_{\sigma\in
\mathcal{W}'_0(x,\vec k_1,\ldots,\vec
k_N)}E_x\big(\sigma(\cdot),L\big)\bigg)
\bigg\}\,d\mathcal{H}^{N-1}(x)\,,
\end{multline}
and plugging it into \er{gkhljknmjkkjjhjffttrytrfrh} we obtain
\begin{multline}\label{gkhljknmjkkjjhjffttrytrfrhjhjhjhkgugyfhfkk}
\int\limits_{\O\cap J_{\vec A\cdot\nabla
v}}\bigg\{\inf\limits_{\sigma\in
\mathcal{Q}(x),L\in(0,1)}H_x\big(\sigma(\cdot),L\big)
\bigg\}\,d\mathcal{H}^{N-1}(x)\\=\int\limits_{\O\cap J_{\vec
A\cdot\nabla
v}}\bigg\{\inf\limits_{L\in(0,1)}\bigg(\inf\limits_{\sigma\in
\mathcal{W}'_0(x,\vec k_1,\ldots,\vec
k_N)}E_x\big(\sigma(\cdot),L\big)\bigg)
\bigg\}\,d\mathcal{H}^{N-1}(x)\,,
\end{multline}
Finally, using \er{gkhljknmjkkjjhjffttrytrfrhjhjhjhkgugyfhfkk} and
applying Lemma \ref{lijkkhnjjhjh} we obtain
\er{gkhljknmjkkjjhjffttryt}.
\end{proof}

By the same method we can prove the following more general result.
\begin{theorem}\label{vtporbound4ghkfgkhkgen}
Let $\Omega\subset\R^N$ be an open set.
Furthermore, let $\vec A\in \mathcal{L}(\R^{d\times N};\R^m)$, $\vec
B\in \mathcal{L}(\R^{d};\R^k)$ and let $F\in C^1\big(\R^{m\times
N^n}\times\R^{m\times N^{(n-1)}}\times\ldots
\times \R^{m\times N}\times \R^m\times \R^k\times\R^q,\R\big)$,
satisfying $F\geq 0$. Let $f\in BV_{loc}(\R^N,\R^q)\cap L^\infty$
and $v\in\mathcal{D}'(\R^N,\R^d)$ be such that $\vec A\cdot\nabla
v\in BV(\R^N,\R^m)\cap L^\infty(\R^N,\R^m)$ and $\vec B\cdot v\in
Lip(\R^N,\R^k)\cap W^{1,1}(\R^N,\R^k)\cap L^\infty(\R^N,\R^k)$,
$\|D(\vec A\cdot \nabla v)\|(\partial\Omega)=0$ and
$$F\Big(0,0,\ldots,0,\{\vec A\cdot\nabla v\}(x),\{\vec B\cdot
v\}(x),f(x)\Big)=0\quad\text{a.e.~in}\;\Omega\,.$$ Moreover, assume
that for $\mathcal{H}^{N-1}$-a.e. $x\in \O\cap J_{\vec A\cdot\nabla
v}$ there exists a distribution $\gamma_x(\cdot)\in
\mathcal{D}'(\R^N,\R^d)$ such that
\begin{equation}\label{hhjhhkhkhjhghfhlllkkkjjjhhhgffddddgen}
\{\vec A\cdot\nabla\gamma_x\}(z)=\begin{cases}(\vec A\cdot\nabla
v)^+(x)\quad\quad\text{if}\quad z\cdot\vec\nu(x)\geq 0\,,\\
(\vec A\cdot\nabla v)^-(x)\quad\quad\text{if}\quad z\cdot\vec\nu(x)<
0\,
\end{cases}
\end{equation}
(with $\vec\nu(x)$ denoting the orientation vector of $J_{\vec
A\cdot\nabla v}$). Finally we assume that for
$\mathcal{H}^{N-1}$-a.e. $x\in \O\cap J_{\vec A\cdot\nabla v}$ for
every system $\{\vec k_1(x),\vec k_2(x),\ldots \vec k_N(x)\}$ of
linearly independent vectors in $\R^N$ satisfying $\vec
k_1(x)=\vec\nu(x)$ and $\vec k_j(x)\cdot\vec\nu(x)=0$ for $j\geq 2$,
and
for every $\xi(z)\in\mathcal{W}_{(x,n)}(\vec k_1,\vec k_2,\ldots
\vec k_N)$ there exists $\zeta(z)\in\mathcal{W}'_{(x,n)}(\vec
k_1,\vec k_2,\ldots \vec k_N)$ such that $\vec
A\cdot\nabla\xi(z)\equiv\vec A\cdot\nabla\zeta(z)$, where
\begin{multline}\label{L2009Ddef2hhhjjjj77788hhhkkkkllkjjjjkkkhhhhffggdddkkkgen}
\mathcal{W}_{(x,n)}\big(\vec k_1(x),\vec k_2(x),\ldots \vec
k_N(x)\big):=
\bigg\{ u\in C^\infty(\R^N,\R^d):\;\vec A\cdot\nabla
u(y)=0\;\text{ if }\;|y\cdot\vec\nu(x)|\geq 1/2,\\
\text{ and }\;\vec A\cdot\nabla u\big(y+\vec k_j(x)\big)=\vec
A\cdot\nabla u(y)\;\;\forall j=2,3,\ldots, N\bigg\}\,,
\end{multline}
and
\begin{multline}\label{L2009Ddef2hhhjjjj77788nullkkkkkkjjjjjkkkgen}
\mathcal{W}'_{(x,n)}\big(\vec k_1(x),\vec k_2(x),\ldots \vec
k_N(x)\big):=\\
\bigg\{ u\in C^\infty(\R^N,\R^d)\cap L^\infty(\R^N,\R^d)
:\;\nabla^j u\in L^\infty(\R^N,\R^{d\times N^j})\;\text{for}\;j\leq n\,,\\
\vec A\cdot\nabla u(y)=0\;\text{ if }\; |y\cdot\vec\nu(x)|\geq 1/2\;
\text{ and }\;
u\big(y+\vec k_j(x)\big)=u(y) \;\;\forall j=2,3,\ldots, N\bigg\}\,.
\end{multline}
Then, for $\eta\in \mathcal{V}_0$,
for every $\delta>0$  there exist a sequence
$\{v_\e\}_{0<\e<1}\subset C^\infty(\R^N,\R^d)$
such that $\lim_{\e\to0^+}\vec A\cdot\nabla v_\e=\vec A\cdot \nabla
v$ in $L^p(\R^N,\R^m)$, $\lim_{\e\to0^+} \e^j\nabla^j\{\vec A\cdot
\nabla v_\e\}=0$ in $L^{p}(\R^N,\R^{m\times N^j})$ for every
$j=1,2,\ldots n$, $\lim_{\e\to0^+}\vec B\cdot v_\e=\vec B\cdot v$ in
$W^{1,p}(\R^N,\R^k)$, $\lim_{\e\to0^+}(\vec B\cdot v_\e-\vec B\cdot
v)/\e=0$ in $L^{p}(\R^N,\R^k)$  for every $p\geq 1$ and
\begin{multline}
\label{L2009limew03zeta71288888Contggiuuggyyyynew88789999vprop78899shtrihkkklljkhkhgghhhjhhhjkkkgen}
0\leq\lim\limits_{\e\to 0}\int\limits_{\O}\frac{1}{\e}
F\bigg(\e^n\nabla^n\{\vec A\cdot\nabla
v_{\e}\},\,\e^{n-1}\nabla^{n-1}\{\vec A\cdot\nabla
v_{\e}\},\,\ldots,\,\e\nabla\{\vec A\cdot\nabla v_{\e}\},\,\vec
A\cdot\nabla v_{\e},\,\vec B\cdot
v_{\e},\,f\bigg)\,dx\\-\int\limits_{\O\cap J_{\vec A\cdot\nabla
v}}\Bigg\{\inf\limits_{L>0}\bigg(\inf\limits_{\sigma\in
\mathcal{W}^{(n)}_0(x,\vec k_1,\ldots,\vec
k_N)}E_{x,n}\big(\sigma(\cdot),L\big)\bigg)
\Bigg\}\,d\mathcal{H}^{N-1}(x)<\delta\,,
\end{multline}
where
\begin{multline}
\label{L2009limew03zeta71288888Contggiuuggyyyynew88789999vprop78899shtrihkkkllyhjyukjkkmmmklklklhhhhkkffgghhjjjkkkllkkgen}
E_{x,n}\big(\sigma(\cdot),L\big):=
\int\limits_{I_{N}}\frac{1}{L}F\bigg(L^n\nabla^n\{\vec
A\cdot\nabla\sigma\}\big(S_{x}(s)\big),\,L^{n-1}\nabla^{n-1}\{\vec
A\cdot\nabla\sigma\}\big(S_{x}(s)\big),\,\ldots,\,\\ L\nabla\{\vec
A\cdot\nabla\sigma\}\big(S_{x}(s)\big),\,\{\vec
A\cdot\nabla\sigma\}\big(S_{x}(s)\big),\,\vec B\cdot
v(x),\,\zeta\big(s_1,f^+(x),f^-(x)\big)
\bigg)\,ds\,,
\end{multline}
\begin{equation}\label{vhjgkjghklhlhjgjkgujh}
\zeta(s,a,b):=
\begin{cases}
a\quad\text{if}\;\; s>0\,,\\
b\quad\text{if}\;\;s<0\,,
\end{cases}
\end{equation}
\begin{multline}\label{L2009Ddef2hhhjjjj77788hhhkkkkllkjjjjkkkhhhhffggdddkkkgjhikgen}
\mathcal{W}^{(n)}_0\big(x,\vec k_1(x),\vec k_2(x),\ldots \vec
k_N(x)\big):=\\
\bigg\{u\in \mathcal{D}'(\R^N,\R^d): \vec A\cdot\nabla u\in
C^n(\R^N,\R^m),\;\;\vec A\cdot\nabla u(y)=(\vec A\cdot\nabla
v)^-(x)\;\text{ if }\;y\cdot\vec\nu(x)\leq-1/2,\\
\vec A\cdot\nabla u(y)=(\vec A\cdot\nabla v)^+(x)\;\text{ if }\;
y\cdot\vec\nu(x)\geq 1/2\;\text{ and }\;\vec A\cdot\nabla
u\big(y+\vec k_j(x)\big)=\vec A\cdot\nabla u(y)\;\;\forall
j=2,3,\ldots, N\bigg\}\,,
\end{multline}
\begin{equation}\label{cubidepgghkkkkllllllljjjkkkkkjjjhhhhlkkkffffggggdddgen}
I_N:= \Big\{s\in\R^N:\; -1/2<s_j<1/2\quad\forall
j=1,2,\ldots, N\Big\}\,,
\end{equation}
$\{\vec k_1(x),\vec k_2(x),\ldots \vec k_N(x)\}$ is an arbitrary
system of linearly independent vectors in $\R^N$ satisfying $\vec
k_1(x)=\vec\nu(x)$ and $\vec k_j(x)\cdot\vec\nu(x)=0$ for $j\geq 2$,
the linear transformation $S_x(s)=S_{x}(s_1,s_2,\ldots,
s_N):\R^N\to\R^N$ is defined by
\begin{equation}\label{chfhh2009y888shtrihkkkllbjgjkllkjkjjjjujkkllkkkgen}
S_{x}(s):=\sum\limits_{j=1}^{N}s_j\,\vec k_j(x)\,,
\end{equation}
and we assume that the orientation of $J_f$ coincides with the
orientation of $J_{\vec A\cdot\nabla v}$ $\mathcal{H}^{N-1}$ a.e. on
$J_f\cap J_{\vec A\cdot\nabla v}$. Moreover, $\vec A\cdot\nabla
v_\e$, $\e\nabla\{\vec A\cdot\nabla v_\e\}$, $\e^2\nabla^2\{\vec
A\cdot\nabla v_\e\}$, $\ldots$, $\e^n\nabla^n\{\vec A\cdot\nabla
v_\e\}$, $\vec B\cdot v_\e$ and $\nabla\{\vec B\cdot v_\e\}$ are
bounded in $L^\infty$ sequences, there exists a compact
$K=K_\delta\subset\subset\O$ such that $v_\e(x)=\psi_\e(x)$ for
every $0<\e<1$ and every $x\in\R^N\setminus K$, where
\begin{equation}\label{a4543chejhgufydfdtgjkgkjghgh}
\psi_\e(x):=\frac{1}{\e^N}\Big<\eta\Big(\frac{y-x}{\e}\Big),v(y)\Big>\,,
\end{equation}
and
\begin{equation}\label{bjdfgfghljjklkjkllllkkllkkkkkkllllkkllhffhhfhftfffffhhhhkkkkkgen}
\limsup\limits_{\e\to 0^+}\Bigg|\frac{1}{\e}\bigg(\int_\O\vec
A\cdot\nabla v_\e(x)\,dx-\int_\O\{\vec A\cdot \nabla
v\}(x)\,dx\bigg)\Bigg|<+\infty\,,
\end{equation}
\end{theorem}

\section{The applications}
\begin{theorem}\label{vtporbound4ghkfgkhkhhhjj}
Let $\Omega\subset\R^N$ be an open set.
Furthermore,
let $$F\in C^1\Big(\R^{k\times N\times N}\times\R^{d\times N\times
N}\times\R^{m\times N}\times\R^{k\times N}\times\R^{d\times
N}\times\R^{m}\times \R^k\times\R^q\,,\,\R\Big)\,$$ satisfying
$F\geq 0$. Let $f\in BV_{loc}(\R^N,\R^q)\cap L^\infty$, $v\in
Lip(\R^N,\R^k)\cap L^1\cap L^\infty$, $\bar m\in BV(\R^N,\R^{d\times
N})\cap L^\infty$ and $\varphi\in BV(\R^N,\R^{m})\cap L^\infty$ be
such that $\nabla v\in BV(\R^N,\R^{k\times N})$,
$\|D(\nabla v)\|(\partial\Omega)=0$, $\|D \bar
m\|(\partial\Omega)=0$, $\|D \varphi\|(\partial\Omega)=0$, $div_x
\bar m(x)=0$ a.e.~in $\R^N$ and $F\big(0,0,0,\nabla v(x),\bar
m(x),\varphi(x),v(x),f(x)\big)=0$ a.e.~in $\Omega$. Then, for
$\eta\in \mathcal{V}_0$,
for every $\delta>0$ there exist sequences $\{v_\e\}_{0<\e<1}\subset
C^\infty(\R^N,\R^k)$, $\{m_\e\}_{0<\e<1}\subset
C^\infty(\R^N,\R^{d\times N})$ and $\{\psi_\e\}_{0<\e<1}\subset
C^\infty(\R^N,\R^{m})$
such that $div_x m_\e(x)\equiv 0$ in $\R^N$,
$\int_\O\psi_\e(x)\,dx=\int_\O \varphi(x)\,dx$,
$\lim_{\e\to0^+}v_\e=v$ in $W^{1,p}(\R^N,\R^k)$,
$\lim_{\e\to0^+}(v_\e-v)/\e=0$ in $L^{p}(\R^N,\R^k)$,
$\lim_{\e\to0^+}m_\e=\bar m$ in $L^{p}(\R^N,\R^{d\times N})$,
$\lim_{\e\to0^+}\psi_\e=\varphi$ in $L^{p}(\R^N,\R^{m})$,
$\lim_{\e\to0^+} \e\nabla^2 v_\e=0$ in $L^{p}(\R^N,\R^{k\times
N\times N})$, $\lim_{\e\to0^+} \e\nabla m_\e=0$ in
$L^{p}(\R^N,\R^{d\times N\times N})$, $\lim_{\e\to0^+}
\e\nabla\psi_\e=0$ in $L^{p}(\R^N,\R^{m\times N})$ for every $p\geq
1$ and
\begin{multline}
\label{L2009limew03zeta71288888Contggiuuggyyyynew88789999vprop78899shtrihkkklljkhkhgghhhjhhhjkkkhhhjj}
0\leq\lim\limits_{\e\to 0}\int\limits_{\O}\frac{1}{\e}
F\bigg(\e\nabla^2v_{\e}(x),\,\e\nabla
m_\e(x),\,\e\nabla\psi_\e(x),\,\nabla
v_{\e}(x),\,m_\e(x),\,\psi_\e(x),\,
v_{\e}(x),\,f(x)\bigg)\,dx\\-\int\limits_{\O\cap (J_{\nabla v}\cup
J_{\bar m}\cup J_{\varphi})}\Bigg( \inf\bigg\{\bar
E_x\big(\sigma(\cdot),\theta(\cdot),\gamma(\cdot),L\big):\;\;L>0,\\
\;\sigma\in \mathcal{W}^{(1)}_0(x,\vec k_1,\ldots,\vec
k_N),\,\theta\in \mathcal{W}^{(2)}_0(x,\vec k_1,\ldots,\vec
k_N),\,\gamma\in \mathcal{W}^{(3)}_0(x,\vec k_1,\ldots,\vec
k_N)\bigg\} \Bigg)\,d\mathcal{H}^{N-1}(x)<\delta\,,
\end{multline}
where
\begin{multline}
\label{L2009limew03zeta71288888Contggiuuggyyyynew88789999vprop78899shtrihkkkllyhjyukjkkmmmklklklhhhhkkffgghhjjjkkkllkkhhhjj}
\bar E_x\big(\sigma(\cdot),\theta(\cdot),\gamma(\cdot),L\big):=
\int\limits_{I^+_{\vec k_1,\ldots,\vec
k_N}}\frac{1}{L}F\bigg(L\nabla^2\sigma\big(y\big),
\,L\nabla\theta\big(y\big),
\,L\nabla\gamma\big(y\big), \,\nabla\sigma\big(y\big),\,\theta\big(y\big),\,\gamma\big(y\big),\,v(x),\,f^+(x) \bigg)\,dy\\
+\int\limits_{I^-_{\vec k_1,\ldots,\vec
k_N}}\frac{1}{L}F\bigg(L\nabla^2\sigma\big(y\big),
\,L\nabla\theta\big(y\big), \,L\nabla\gamma\big(y\big),
\,\nabla\sigma\big(y\big),\,\theta\big(y\big),\,\gamma\big(y\big),\,v(x),\,f^-(x)
\bigg)\,dy\,,
\end{multline}
\begin{multline}\label{L2009Ddef2hhhjjjj77788hhhkkkkllkjjjjkkkhhhhffggdddkkkgjhikhhhjj}
\mathcal{W}^{(1)}_0\big(x,\vec k_1(x),\vec k_2(x),\ldots \vec
k_N(x)\big):=
\bigg\{u\in C^2(\R^N,\R^k):\;\;\nabla u(y)=(\nabla
v)^-(x)\;\text{ if }\;y\cdot\vec\nu(x)\leq-1/2,\\
\nabla u(y)=(\nabla v)^+(x)\;\text{ if }\; y\cdot\vec\nu(x)\geq
1/2\;\text{ and }\;\nabla u\big(y+\vec k_j(x)\big)=\nabla
u(y)\;\;\forall j=2,3,\ldots, N\bigg\}\,,
\end{multline}
\begin{multline}\label{L2009Ddef2hhhjjjj77788hhhkkkkllkjjjjkkkhhhhffggdddkkkgjhikhhhjjhhhh}
\mathcal{W}^{(2)}_0\big(x,\vec k_1(x),\vec k_2(x),\ldots \vec
k_N(x)\big):=\\
\bigg\{\xi\in C^1(\R^N,\R^{d\times N}):\;\;div_y \xi(y)\equiv
0,\;\;\xi(y)=\bar m^-(x)\;\text{ if }\;y\cdot\vec\nu(x)\leq-1/2,\\
\xi(y)=\bar m^+(x)\;\text{ if }\; y\cdot\vec\nu(x)\geq 1/2\;\text{
and }\;\xi\big(y+\vec k_j(x)\big)=\xi(y)\;\;\forall j=2,3,\ldots,
N\bigg\}\,,
\end{multline}
\begin{multline}\label{L2009Ddef2hhhjjjj77788hhhkkkkllkjjjjkkkhhhhffggdddkkkgjhikhhhjjdddd}
\mathcal{W}^{(3)}_0\big(x,\vec k_1(x),\vec k_2(x),\ldots \vec
k_N(x)\big):=
\bigg\{\zeta\in C^1(\R^N,\R^m):\;\;\zeta(y)=\varphi^-(x)\;\text{ if }\;y\cdot\vec\nu(x)\leq-1/2,\\
\zeta(y)=\varphi^+(x)\;\text{ if }\; y\cdot\vec\nu(x)\geq
1/2\;\text{ and }\;\zeta\big(y+\vec k_j(x)\big)=\zeta(y)\;\;\forall
j=2,3,\ldots, N\bigg\}\,,
\end{multline}
\begin{equation}\label{cubidepgghkkkkllllllljjjkkkkkjjjhhhhlkkkffffggggdddhhhjj}
\begin{split}
I^-_{\vec k_1,\vec k_2,\ldots,\vec k_N}:=
\Big\{y\in\R^N:\;-1/2<y\cdot\vec k_1<0,\;\; |y\cdot\vec
k_j|<1/2\quad\forall
j=2,3,\ldots, N\Big\}\,,\\
I^+_{\vec k_1,\vec k_2,\ldots,\vec k_N}:=
\Big\{y\in\R^N:\;0<y\cdot\vec k_1<1/2,\;\; |y\cdot \vec
k_j|<1/2\quad\forall j=2,3,\ldots, N\Big\}\,,
\end{split}
\end{equation}
$\{\vec k_1(x),\vec k_2(x),\ldots \vec k_N(x)\}$ is an orthonormal
base in $\R^N$, satisfying $\vec k_1(x)=\vec\nu(x)$,
and we assume that the orientations of $J_{\nabla v}$, $J_{\bar m}$,
$J_\varphi$ and $J_f$ coincides $\mathcal{H}^{N-1}$ a.e. and given
by the vector $\vec\nu(x)$. Moreover, $\nabla v_\e$, $\e\nabla^2
v_\e$, $v_\e$, $m_\e$, $\e\nabla m_\e$, $\psi_\e$ and
$\e\nabla\psi_\e$
are bounded in $L^\infty$ sequences, and there exists a compact
$K=K_\delta\subset\subset\O$ such that $v_\e(x)=v^{(0)}_\e(x)$,
$m_\e(x)=m^{(0)}_\e(x)$ and $\psi_\e(x)=\psi^{(0)}_\e(x)$ for every
$0<\e<1$ and every $x\in\R^N\setminus K$, where
\begin{multline*}
v^{(0)}_\e(x)=\frac{1}{\e^N}\int_{\R^N}\eta\Big(\frac{y-x}{\e}\Big)\,v(y)\,dy\,,\quad
m^{(0)}_\e(x)=\frac{1}{\e^N}\int_{\R^N}\eta\Big(\frac{y-x}{\e}\Big)\,\bar
m(y)\,dy\,,
\\
\quad\psi^{(0)}_\e(x)=\frac{1}{\e^N}\int_{\R^N}\eta\Big(\frac{y-x}{\e}\Big)\,\varphi(y)\,dy\,.
\end{multline*}
\end{theorem}
\begin{proof}
%
%
%
%
%
%
%
%
%
%
Define the Borel sets:
\begin{align}\label{gfgfyfgdgdfggh}
&\mathcal{K}_j:=\Big\{x\in \O\cap (J_{\nabla v}\cup J_{\bar m}\cup
J_{\varphi}):\;\vec\nu(x)\cdot\vec e_j\neq 0\Big\}\quad\quad\forall
j=1,2,\ldots N\,,\\
&\mathcal{A}_j:=\Bigg\{\vec k\in
\R^N:\;\;\mathcal{H}^{N-1}\Big(\big\{x\in
\mathcal{K}_j:\;\vec\nu(x)\cdot\vec
k=0\big\}\Big)>0\Bigg\}\quad\quad\forall
j=1,2,\ldots N\,,\\
&\mathcal{A}:=\Bigg\{\vec k\in
\R^N:\;\;\mathcal{H}^{N-1}\Big(\big\{x\in \O\cap (J_{\nabla v}\cup
J_{\bar m}\cup J_{\varphi}):\;\vec\nu(x)\cdot\vec
k=0\big\}\Big)>0\Bigg\}\,,
\end{align}
where $\{\vec e_1,\vec e_2,\ldots,\vec e_N\}$ is the standard
orthonormal base in $\R^N$. We will prove now that
\begin{equation}\label{vvhcggchhhgghgkkk}
\mathcal{L}^N(\mathcal{A})=0\,.
\end{equation}
Indeed, since $\O\cap (J_{\nabla v}\cup J_{\bar m}\cup
J_{\varphi})=\bigcup_{j=1}^{N}\mathcal{K}_j$, we have
$\mathcal{A}=\bigcup_{j=1}^{N}\mathcal{A}_j$. Therefore it is
sufficient to prove
\begin{equation}\label{vvhcggchhhgghgkkkhjkjh}
\mathcal{L}^N(\mathcal{A}_j)=0\quad\quad\forall j=1,2,\ldots N\,.
\end{equation}
Without any loss of generality we can prove it only in the
particular case $j=1$. Then
$$\mathcal{A}_1:=\Bigg\{(k_1,k'):=\vec k\in
\R^N:\;\;\mathcal{H}^{N-1}\bigg(\bigg\{x\in \mathcal{K}_1:\;
k_1=-\frac{\nu'(x)\cdot k'}{\nu_1(x)}\bigg\}\bigg)>0\Bigg\}\,,$$
where $\vec k=(k_1,k')\in\R\times\R^{N-1}=\R^N$ and $\vec
\nu(x)=\big(\nu_1(x),\nu'(x)\big)\in\R\times\R^{N-1}=\R^N$. On the
other hand set
$$\mathcal{B}(k'):=\Bigg\{k_1\in
\R:\;\;\mathcal{H}^{N-1}\bigg(\bigg\{x\in \mathcal{K}_1:\;
k_1=-\frac{\nu'(x)\cdot
k'}{\nu_1(x)}\bigg\}\bigg)>0\Bigg\}\quad\quad\forall
k'\in\R^{N-1}\,.$$ Thus since
$$\bigg\{x\in \mathcal{K}_1:\;
a=-\frac{\nu'(x)\cdot k'}{\nu_1(x)}\bigg\}\cap\bigg\{x\in
\mathcal{K}_1:\; b=-\frac{\nu'(x)\cdot
k'}{\nu_1(x)}\bigg\}=\emptyset\quad\text{if}\; a\neq b\,,$$ and
since the set $\mathcal{K}_1$ is $\mathcal{H}^{N-1}$ $\sigma$-finite
we obtain that for every $k'\in\R^{N-1}$ the set $\mathcal{B}(k')$
is at most countable. Therefore, for every $k'\in\R^{N-1}$ we have
$\mathcal{L}^1\big(\mathcal{B}(k')\big)=0$ and thus
$$\mathcal{L}^N(\mathcal{A}_1)=\int_{\R^{N-1}}\mathcal{L}^1\big(\mathcal{B}(k')\big)dk'=0\,,$$
So we proved \er{vvhcggchhhgghgkkkhjkjh}, which implies
\er{vvhcggchhhgghgkkk}.

 In particular, by \er{vvhcggchhhgghgkkk} we deduce that
$S^{N-1}\setminus \mathcal{A}\neq\emptyset$. So there exists $\vec
r_0\in S^{N-1}\setminus \mathcal{A}$ and we have
\begin{equation}\label{vhgvtguyiiuijjkjkk}
\mathcal{H}^{N-1}\Big(\big\{x\in \O\cap (J_{\nabla v}\cup J_{\bar
m}\cup J_{\varphi}):\;\vec\nu(x)\cdot\vec r_0=0\big\}\Big)=0\,.
\end{equation}
Without loss of generality we may assume that $\vec r_0=\vec
e_1:=(1,0,0,\ldots,0)$. Therefore from this point we assume that
\begin{equation}\label{vhgvtguyiiuijjkjkkjggjkjjh}
\mathcal{H}^{N-1}\Big(\big\{x\in \O\cap (J_{\nabla v}\cup J_{\bar
m}\cup J_{\varphi}):\;\vec\nu(x)\cdot\vec e_1=0\big\}\Big)=0\,,
\end{equation}
i.e. $\nu_1(x)\neq 0$ for $\mathcal{H}^{N-1}$-a.e. $x\in \O\cap
(J_{\nabla v}\cup J_{\bar m}\cup J_{\varphi})$. Next define
$\Phi:\R^N\to\R^m$ and $M:\R^N\to\R^{d\times (N-1)}$ by
\begin{equation}\label{vhgvtguyiiuijjkjkkjggjkjjhjkkllg}
\Phi(x):=\int_{-\infty}^{x_1}\varphi(s,x')ds\,,\quad\text{and}\quad
M(x):=\int_{-\infty}^{x_1}m'(s,x')ds\quad\quad\forall
x=(x_1,x'):=(x_1,x_2,\ldots x_N)\in\R^{N}\,,
\end{equation}
where we denote by $m_1(x):\R^N\to\R^d$ and $m'(x):\R^N\to
\R^{d\times (N-1)}$ the first column and the rest of the matrix
valued function $\bar m(x):\R^N\to\R^{d\times N}$, so that
$\big(m_1(x),m'(x)\big):=\bar m(x):\R^N\to\R^{d\times N}$. Then,
since $div_x \bar m\equiv 0$, by
\er{vhgvtguyiiuijjkjkkjggjkjjhjkkllg} we obtain
\begin{equation}\label{vhgvtguyiiuijjkjkkjggjkjjhjkkllghjjhjhhkj}
\frac{\partial\Phi}{\partial x_1}(x)=\varphi(x)\,,\quad
\frac{\partial M}{\partial x_1}(x)=m'(x)\,,\quad\text{and}\quad
-div_{x'}M(x)=m_1(x)\quad\quad\text{for a.e.}\;\;
x=(x_1,x')\in\R^{N}\,.
\end{equation}
Next for every $x\in \O\cap J_{\nabla v}\cup J_{\bar m}\cup
J_{\varphi}$ define
\begin{multline}\label{ffghghjkkkhkgmjhkjjjjjkk}
Q_x(z):=\begin{cases}\varphi^+(x)\quad\quad\text{if}\quad z\cdot\vec\nu(x)\geq 0\,,\\
\varphi^-(x)\quad\quad\text{if}\quad z\cdot\vec\nu(x)< 0\,
\end{cases}
\quad
P_x(z):=\begin{cases}\bar m^+(x)\quad\quad\text{if}\quad z\cdot\vec\nu(x)\geq 0\,,\\
\bar m^-(x)\quad\quad\text{if}\quad z\cdot\vec\nu(x)< 0\,
\end{cases}\\ \text{and}\quad H_x(z):=\begin{cases}(\nabla v)^+(x)\quad\quad\text{if}\quad z\cdot\vec\nu(x)\geq 0\,,\\
(\nabla v)^-(x)\quad\quad\text{if}\quad z\cdot\vec\nu(x)< 0\,.
\end{cases}
\end{multline}
Since $div_x \bar m\equiv 0$ and $curl(\nabla v)\equiv 0$, clearly
for $\mathcal{H}^{N-1}$-a.e. $x\in \O\cap (J_{\nabla v}\cup J_{\bar
m}\cup J_{\varphi})$ we have $div_z P(z)=0$ and $curl_z H_x(z)=0$.
In particular, clearly exists a function $\gamma_x(z):\R^N\to\R^k$
such that $\nabla_z\gamma_x(z)=H_x(z)$. Moreover, since
$\nu_1(x)\neq 0$ for $\mathcal{H}^{N-1}$-a.e. $x\in \O\cap
(J_{\nabla v}\cup J_{\bar m}\cup J_{\varphi})$, then if $\nu_1(x)>0$
we define
\begin{equation}\label{vhgvtguyiiuijjkjkkjggjkjjhjkkllgkhhhhhlhklhklhkl}
R_x(z):=\int_{-\infty}^{z_1}\Big(Q_x(s,z')-\varphi^-(x)\Big)ds\;\text{and}\;
D_x(z):=\int_{-\infty}^{z_1}\Big(P'_x(s,z')-\big\{m^-(x)\big\}'\Big)ds\;\;\forall
z=(z_1,z')\in\R^{N}\,,
\end{equation}
where we denote by $Y_1:\R^N\to\R^d$ and $Y':\R^N\to \R^{d\times
(N-1)}$ the first column and the rest of the matrix
$Y:\R^N\to\R^{d\times N}$. Then for such $x$
we will have
\begin{multline}\label{vhgvtguyiiuijjkjkkjggjkjjhjkkllghjjhjhhkjkukuuk}
\frac{\partial}{\partial z_1}
\Big(R_x(z)+z_1\varphi^-(x)\Big)=Q(z),\quad \frac{\partial}{\partial
z_1}
\bigg(D_x(z)+z_1\{m'\}^-(x)-\frac{1}{N-1}m^-_1(x)\otimes z'\bigg)=P'_x(z),\\
-div_{z'}\bigg(D_x(z)+z_1\{m'\}^-(x)-\frac{1}{N-1}m^-_1(x)\otimes
z'\bigg)=\{P_x\}_1(z),\quad \nabla_z\gamma_x(z)=H_x(z)\;\;\forall
z=(z_1,z')\in\R^{N}.
\end{multline}
If $\nu_1(x)<0$ we interchange the role of $(\varphi^-, \bar m^-)$
by $(\varphi^+, \bar m^+)$ in
\er{vhgvtguyiiuijjkjkkjggjkjjhjkkllgkhhhhhlhklhklhkl}. Thus in
general for $\mathcal{H}^{N-1}$-a.e. $x\in \O\cap (J_{\nabla v}\cup
J_{\bar m}\cup J_{\varphi})$ we have
\begin{multline}\label{vhgvtguyiiuijjkjkkjggjkjjhjkkllghjjhjhhkjkukuukjkkjk}
\frac{\partial}{\partial z_1}
\Big(R_x(z)+z_1\varphi^\pm(x)\Big)=Q(z),\quad
\frac{\partial}{\partial z_1}
\bigg(D_x(z)+z_1\{m'\}^\pm(x)-\frac{1}{N-1}m^\pm_1(x)\otimes z'\bigg)=P'_x(z),\\
-div_{z'}\bigg(D_x(z)+z_1\{m'\}^\pm(x)-\frac{1}{N-1}m^\pm_1(x)\otimes
z'\bigg)=\{P_x\}_1(z),\quad \nabla_z\gamma_x(z)=H_x(z)\;\;\forall
z=(z_1,z')\in\R^{N}.
\end{multline}
Here we take the sign "$-$" if $\nu_1(x)>0$ and the sign "$+$" if
$\nu_1(x)<0$.

 Next for every $x\in \O\cap (J_{\nabla v}\cup J_{\bar m}\cup J_{\varphi})$ fix an
arbitrary system $\{\vec k_1(x),\vec k_2(x),\ldots \vec k_N(x)\}$ of
linearly independent vectors in $\R^N$ satisfying $\vec
k_1(x)=\vec\nu(x)$ and $\vec k_j(x)\cdot\vec\nu(x)=0$ for $j\geq 2$.
Then define
\begin{multline}\label{L2009Ddef2hhhjjjj77788hhhkkkkllkjjjjkkkhhhhffggdddkkkhhhjjhjjgghghgkffhhf}
\mathcal{W}^{(1)}_1(x):=
\bigg\{ u\in C^\infty(\R^N,\R^{k}):\;\;\nabla u(y)=0\;\text{ if }\;|y\cdot\vec\nu(x)|\geq 1/2,\\
\text{ and }\;\nabla u\big(y+\vec k_j(x)\big)=\nabla u(y)\;\;\forall
j=2,3,\ldots, N\bigg\}\,,
\end{multline}
\begin{multline}\label{L2009Ddef2hhhjjjj77788hhhkkkkllkjjjjkkkhhhhffggdddkkkhhhjjhjjggh}
\mathcal{W}^{(2)}_1(x):=
\bigg\{ \xi\in C^\infty(\R^N,\R^{d\times N}):\;div_y\xi(y)\equiv 0,\;\;\xi(y)=0\;\text{ if }\;|y\cdot\vec\nu(x)|\geq 1/2,\\
\text{ and }\;\xi\big(y+\vec k_j(x)\big)=\xi (y)\;\;\forall
j=2,3,\ldots, N\bigg\}\,,
\end{multline}
\begin{multline}\label{L2009Ddef2hhhjjjj77788hhhkkkkllkjjjjkkkhhhhffggdddkkkhhhjj}
\mathcal{W}^{(3)}_1(x):=
\bigg\{ \zeta\in C^\infty(\R^N,\R^m):\;\zeta(y)=0\;\text{ if }\;|y\cdot\vec\nu(x)|\geq 1/2,\\
\text{ and }\;\zeta\big(y+\vec k_j(x)\big)=\zeta(y)\;\;\forall
j=2,3,\ldots, N\bigg\}\,.
\end{multline}
Then since $\nu_1(x)\neq 0$ for $\mathcal{H}^{N-1}$-a.e. $x\in
\O\cap (J_{\nabla v}\cup J_{\bar m}\cup J_{\varphi})$, for such
fixed $x$ and for every $\xi\in\mathcal{W}^{(2)}_1(x)$ and every
$\zeta\in\mathcal{W}^{(3)}_1(x)$ we can define
$\Xi_\zeta(y):\R^N\to\R^m$ and $\Upsilon_\xi(y):\R^N\to\R^{d\times
(N-1)}$ by
\begin{equation}\label{vhgvtguyiiuijjkjkkjggjkjjhjkkllgkkhh}
\Xi_\zeta(y):=\int_{-\infty}^{y_1}\zeta(s,y')ds\,,\quad\text{and}\quad
\Upsilon_\xi(y):=\int_{-\infty}^{x_1}\xi'(s,y')ds\quad\quad\forall
y=(y_1,y')\in\R^{N}\,,
\end{equation}
Then, since $div_y \zeta\equiv 0$, we obtain
\begin{equation}\label{vhgvtguyiiuijjkjkkjggjkjjhjkkllghjjhjhhkjkjhhkhhk}
\frac{\partial\Xi_\zeta}{\partial y_1}(y)=\zeta(y)\,,\quad
\frac{\partial \Upsilon_\xi}{\partial
y_1}(x)=\xi'(y)\,,\quad\text{and}\quad
-div_{y'}\Upsilon_\xi(y)=\xi_1(y)\quad\quad\forall
y=(y_1,y')\in\R^{N}\,.
\end{equation}
Moreover clearly by the definition for all $x$ such that
$\nu_1(x)\neq 0$ i.e. for $\mathcal{H}^{N-1}$-a.e. $x\in \O\cap
J_{\nabla v}\cup J_{\bar m}\cup J_{\varphi}$ we have $\Xi_\zeta\in
C^\infty(\R^N,\R^m)\cap L^\infty(\R^N,\R^m)\cap Lip(\R^N,\R^m)$,
$\Upsilon_\xi\in C^\infty(\R^N,\R^{d\times(N-1)})\cap
L^\infty(\R^N,\R^{d\times(N-1)})\cap Lip(\R^N,\R^{d\times(N-1)})$
and $\Xi_\zeta\big(y+\vec k_j(x)\big)=\Xi_\zeta(y) \;\;\forall
j=2,3,\ldots, N$, $\Upsilon_\xi\big(y+\vec
k_j(x)\big)=\Upsilon_\xi(y) \;\;\forall j=2,3,\ldots, N$. On the
other hand, for every $u(y)\in\mathcal{W}^{(1)}_1(x)$ we clearly
have $u\in C^\infty(\R^N,\R^k)\cap L^\infty(\R^N,\R^k)\cap
Lip(\R^N,\R^k)$ and $u\big(y+\vec k_j(x)\big)=u(y) \;\;\forall
j=2,3,\ldots, N$. So
\begin{multline}\label{hhdhgdkgkdgdhghdh}
\Xi_\zeta\in C^\infty(\R^N,\R^m)\cap L^\infty\cap
Lip\;\;\;\;\text{and}\;\; \Xi_\zeta\big(y+\vec
k_j(x)\big)=\Xi_\zeta(y) \;\;\forall j=2,3,\ldots, N\,,\\
\Upsilon_\xi\in C^\infty(\R^N,\R^{d\times(N-1)})\cap L^\infty\cap
Lip\;\;\;\;\text{and}\;\;\Upsilon_\xi\big(y+\vec
k_j(x)\big)=\Upsilon_\xi(y) \;\;\forall j=2,\ldots, N\,,\\
u\in C^\infty(\R^N,\R^k)\cap L^\infty\cap
Lip\;\;\;\;\text{and}\;\;u\big(y+\vec k_j(x)\big)=u(y) \;\;\forall
j=2,3,\ldots, N\,.
\end{multline}
Then using
\er{vhgvtguyiiuijjkjkkjggjkjjhjkkllghjjhjhhkjkukuukjkkjk},
\er{vhgvtguyiiuijjkjkkjggjkjjhjkkllghjjhjhhkjkjhhkhhk},
\er{hhdhgdkgkdgdhghdh} and Theorem \ref{vtporbound4ghkfgkhk} we
deduce that for every $\delta>0$  there exist sequences
$\{v_\e\}_{0<\e<1}\subset C^\infty(\R^N,\R^k)$,
$\{M_\e\}_{0<\e<1}\subset C^\infty(\R^N,\R^{d\times (N-1)})$ and
$\{\Psi_\e\}_{0<\e<1}\subset C^\infty(\R^N,\R^{m})$ such that if we
denote
\begin{equation}\label{vhgvtguyiiuijjkjkkjggjkjjhjkkllghjjhjhhkjyjyju}
\frac{\partial\Psi_\e}{\partial x_1}(x)=\bar\psi_\e(x)\,,\quad
\frac{\partial M_\e}{\partial
x_1}(x)=m'_\e(x)\,,\quad\text{and}\quad
-div_{x'}M_\e(x)=(m_\e)_1(x)\quad\quad\forall x=(x_1,x')\in\R^{N}\,.
\end{equation}
then we will have
\begin{multline}
\label{L2009limew03zeta71288888Contggiuuggyyyynew88789999vprop78899shtrihkkklljkhkhgghhhjhhhjkkkhhhjjhjhlhjggjk}
0\leq\lim\limits_{\e\to 0}\int\limits_{\O}\frac{1}{\e}
F\bigg(\e\nabla^2v_{\e}(x),\,\e\nabla
m_\e(x),\,\e\nabla\bar\psi_\e(x),\,\nabla
v_{\e}(x),\,m_\e(x),\,\bar\psi_\e(x),\,
v_{\e}(x),\,f(x)\bigg)\,dx\\-\int\limits_{\O\cap (J_{\nabla v}\cup
J_{\bar m}\cup J_{\varphi})}\Bigg( \inf\bigg\{\bar
E_x\big(\sigma(\cdot),\theta(\cdot),\gamma(\cdot),L\big):\;\;L>0,\\
\;\sigma\in \mathcal{W}^{(1)}_0(x,\vec k_1,\ldots,\vec
k_N),\,\theta\in \mathcal{W}^{(2)}_0(x,\vec k_1,\ldots,\vec
k_N),\,\gamma\in \mathcal{W}^{(3)}_0(x,\vec k_1,\ldots,\vec
k_N)\bigg\} \Bigg)\,d\mathcal{H}^{N-1}(x)<\delta\,,
\end{multline}
and $\lim_{\e\to0^+}v_\e=v$ in $W^{1,p}(\R^N,\R^k)$,
$\lim_{\e\to0^+}(v_\e-v)/\e=0$ in $L^{p}$, $\lim_{\e\to0^+}m_\e=\bar
m$ in $L^{p}$, $\lim_{\e\to 0^+}\bar\psi_\e=\varphi$ in $L^{p}$,
$\lim_{\e\to0^+} \e\nabla^2 v_\e=0$ in $L^{p}$, $\lim_{\e\to0^+}
\e\nabla m_\e=0$ in $L^{p}$, $\lim_{\e\to0^+} \e\nabla\bar\psi_\e=0$
in $L^{p}$ for every $p\geq 1$. Moreover, $\nabla v_\e$, $\e\nabla^2
v_\e$, $v_\e$, $m_\e$, $\e\nabla m_\e$, $\bar\psi_\e$ and
$\e\nabla\bar\psi_\e$
will be bounded in $L^\infty$ sequences, for some compact $\bar
K\subset\subset\O$ we will have $v_\e(x)=v^{(0)}_\e(x)$,
$m_\e(x)=m^{(0)}_\e(x)$ and $\bar\psi_\e(x)=\psi^{(0)}_\e(x)$ for
every $x\in\R^N\setminus \bar K$ and we will have
\begin{equation}\label{bjdfgfghljjklkjkllllkkllkkkkkkllllkkllhffhhfhftfffffhhhhkkkkkhhhjj}
\limsup\limits_{\e\to
0^+}\Bigg|\frac{1}{\e}\bigg(\int_\O\bar\psi_\e(x)\,dx-\int_\O
\varphi(x)\,dx\bigg)\Bigg|<+\infty\,.
\end{equation}

 Finally we will slightly modify the sequence $\bar\psi_\e$, so that
all the properties, presented above, will preserve and moreover the
modified sequence $\psi_\e$ will satisfy the additional constraint
$\int_\O\psi_\e=\int_\O\varphi$. For this purpose we define
\begin{equation}\label{bjdfgfghljjklkjkllllkkllkkkkkkllllkkllhffhhfhftfffffhhhhkkkkkhhhjjjjkjkjkjjk}
d_\e:=\int_\O \varphi(x)\,dx-\int_\O\bar\psi_\e(x)\,dx\,.
\end{equation}
Then by
\er{bjdfgfghljjklkjkllllkkllkkkkkkllllkkllhffhhfhftfffffhhhhkkkkkhhhjj}
there exists a constant $C_0>0$, independent on $\e$ such that
\begin{equation}\label{bjdfgfghljjklkjkllllkkllkkkkkkllllkkllhffhhfhftfffffhhhhkkkkkhhhjjjjkjkjkjjkbhhhhjjk}
|d_\e|\le C_0\e\quad\forall \e\in (0,1)\,.
\end{equation}
Now fix some smooth function $\lambda(x)\in C_c^\infty(\O,\R)$, such
that $\int_\O\lambda(x)dx=1$, and a compact set $K\subset\subset\O$
such that $\supp\lambda\subset K$ and $\bar K\subset K$. Then define
$\psi_\e(x)\in C^\infty(\R^N,\R^m)$ by
\begin{equation}\label{bjdfgfghljjklkjkllllkkllkkkkkkllllkkllhffhhfhftfffffhhhhkkkkkhhhjjjjkjkjkjjkbhhhhjjkkllkklkkkkk}
\psi_\e(x):=\bar\psi_\e(x)+\lambda(x)d_\e\quad\forall
x\in\R^N\;\forall\e\in(0,1)\,.
\end{equation}
Thus clearly by
\er{bjdfgfghljjklkjkllllkkllkkkkkkllllkkllhffhhfhftfffffhhhhkkkkkhhhjjjjkjkjkjjk}
we have
\begin{equation}\label{bjdfgfghljjklkjkllllkkllkkkkkkllllkkllhffhhfhftfffffhhhhkkkkkhhhjjjjkjkjkjjkjjjjjjjkkk}
\int_\O\psi_\e(x)\,dx=\int_\O \varphi(x)\,dx\,.
\end{equation}
Moreover, clearly by the definition,
$\lim_{\e\to0^+}\psi_\e=\varphi$ in $L^{p}$ and $\lim_{\e\to0^+}
\e\nabla\psi_\e=0$ in $L^{p}$ for every $p\geq 1$, $\psi_\e$ and
$\e\nabla\psi_\e$ are bounded in $L^\infty$ sequences and
$\psi_\e(x)=\psi^{(0)}_\e(x)$ for every $x\in\R^N\setminus K$. Thus
for the final conclusion it is sufficient to prove
\begin{multline}
\label{L2009limew03zeta71288888Contggiuuggyyyynew88789999vprop78899shtrihkkklljkhkhgghhhjhhhjkkkhhhjjhjhlhjggjkghjhhlhlkkk}
\lim\limits_{\e\to 0}\Bigg\{\int\limits_{\O}\frac{1}{\e}
F\bigg(\e\nabla^2v_{\e},\,\e\nabla m_\e,\,\e\nabla\bar\psi_\e+\e
d_\e\otimes\nabla\lambda,\,\nabla
v_{\e},\,m_\e,\,\bar\psi_\e+\lambda d_\e,\,
v_{\e},\,f\bigg)\,dx-\\
\int\limits_{\O}\frac{1}{\e} F\bigg(\e\nabla^2v_{\e},\,\e\nabla
m_\e,\,\e\nabla\bar\psi_\e,\,\nabla v_{\e},\,m_\e,\,\bar\psi_\e,\,
v_{\e},\,f\bigg)\,dx\Bigg\}=0\,.
\end{multline}
Indeed
\begin{multline}
\label{L2009limew03zeta71288888Contggiuuggyyyynew88789999vprop78899shtrihkkklljkhkhgghhhjhhhjkkkhhhjjhjhlhjggjkghjhhlhlkkkphhhkk}
\int\limits_{\O}\frac{1}{\e} F\bigg(\e\nabla^2v_{\e},\,\e\nabla
m_\e,\,\e\nabla\bar\psi_\e+\e d_\e\otimes\nabla\lambda,\,\nabla
v_{\e},\,m_\e,\,\bar\psi_\e+\lambda d_\e,\,
v_{\e},\,f\bigg)\,dx-\\
\int\limits_{\O}\frac{1}{\e} F\bigg(\e\nabla^2v_{\e},\,\e\nabla
m_\e,\,\e\nabla\bar\psi_\e,\,\nabla v_{\e},\,m_\e,\,\bar\psi_\e,\,
v_{\e},\,f\bigg)\,dx=\\
\int_0^1\int_\O \Big(d_\e\otimes\nabla\lambda\Big):D_1
F\bigg(\e\nabla^2v_{\e},\,\e\nabla m_\e,\,\e\nabla\bar\psi_\e+t\e
d_\e\otimes\nabla\lambda,\,\nabla
v_{\e},\,m_\e,\,\bar\psi_\e+t\lambda d_\e,\, v_{\e},\,f\bigg)
dxdt\\+\int_0^1\int_\O \frac{d_\e}{\e}\cdot \nabla_2
F\bigg(\e\nabla^2v_{\e},\,\e\nabla m_\e,\,\e\nabla\bar\psi_\e+t\e
d_\e\otimes\nabla\lambda,\,\nabla
v_{\e},\,m_\e,\,\bar\psi_\e+t\lambda d_\e,\,
v_{\e},\,f\bigg)\lambda\, dxdt\,.
\end{multline}
Here $D_1 F$ is the gradient of $F$ on the argument in the place of
$\nabla\bar\psi_\e$ and $\nabla_2 F$ is the gradient of $F$ on the
argument in the place of $\bar\psi_\e$. On the other hand
$F(0,0,0,\nabla v,\bar m,\varphi,v,f)=0$ a.e. and therefore, since
$F\geq 0$ we also have $D F(0,0,0,\nabla v,\bar m,\varphi,v,f)=0$.
Thus, while being uniformly bounded,
$$D F\Big(\e\nabla^2v_{\e},\,\e\nabla
m_\e,\,\e\nabla\bar\psi_\e+t\e d_\e\otimes\nabla\lambda,\,\nabla
v_{\e},\,m_\e,\,\bar\psi_\e+t\lambda d_\e,\, v_{\e},\,f\Big)\to
0\quad\text{a.e. in}\;\O\,.$$ Therefore, since $\lambda$ has a
compact support and since by
\er{bjdfgfghljjklkjkllllkkllkkkkkkllllkkllhffhhfhftfffffhhhhkkkkkhhhjjjjkjkjkjjkbhhhhjjk}
$d_\e/\e$ is bounded, we deduce that the r.h.s. of
\er{L2009limew03zeta71288888Contggiuuggyyyynew88789999vprop78899shtrihkkklljkhkhgghhhjhhhjkkkhhhjjhjhlhjggjkghjhhlhlkkkphhhkk}
goes to zero as $\e\to 0^+$. So we proved
\er{L2009limew03zeta71288888Contggiuuggyyyynew88789999vprop78899shtrihkkklljkhkhgghhhjhhhjkkkhhhjjhjhlhjggjkghjhhlhlkkk}.
\end{proof}
By the same method, using Theorem \ref{vtporbound4ghkfgkhkgen}, we
can prove the following more general Theorem.
%
%
%
%
%
%
%
%
%
%
%
%
\begin{theorem}\label{vtporbound4ghkfgkhkhhhjjklkjjhliuik}
Let $\Omega\subset\R^N$ be an open set.
Furthermore,
let $F$ be a $C^1$ function defined on
$$
\big\{\R^{k\times N^{n+1}}\times\R^{d\times
N^{n+1}}\times\R^{m\times
N^n}\big\}\times\ldots\times\big\{\R^{k\times N\times
N}\times\R^{d\times N\times N}\times\R^{m\times
N}\big\}\times\big\{\R^{k\times N}\times\R^{d\times
N}\times\R^{m}\big\}\times \R^k\times\R^q,
$$
taking values in $\R$ and satisfying $F\geq 0$. Let $f\in
BV_{loc}(\R^N,\R^q)\cap L^\infty$, $v\in Lip(\R^N,\R^k)\cap L^1\cap
L^\infty$, $\bar m\in BV(\R^N,\R^{d\times N})\cap L^\infty$ and
$\varphi\in BV(\R^N,\R^{m})\cap L^\infty$ be such that $\nabla v\in
BV(\R^N,\R^{k\times N})$,
$\|D(\nabla v)\|(\partial\Omega)=0$, $\|D \bar
m\|(\partial\Omega)=0$, $\|D \varphi\|(\partial\Omega)=0$, $div_x
\bar m(x)=0$ a.e.~in $\R^N$ and
$$F\Big(0,0,\ldots,0,\nabla v,\bar m,\varphi,v,f\Big)=0
\quad\text{a.e.~in}\; \Omega\,.$$
Then, for $\eta\in \mathcal{V}_0$,
for every $\delta>0$  there exist sequences
$\{v_\e\}_{0<\e<1}\subset C^\infty(\R^N,\R^k)$,
$\{m_\e\}_{0<\e<1}\subset C^\infty(\R^N,\R^{d\times N})$ and
$\{\psi_\e\}_{0<\e<1}\subset C^\infty(\R^N,\R^{m})$
such that $div_x m_\e(x)\equiv 0$ in $\R^N$,
$\int_\O\psi_\e(x)\,dx=\int_\O \varphi(x)\,dx$,
$\lim_{\e\to0^+}v_\e=v$ in $W^{1,p}$, $\lim_{\e\to0^+}(v_\e-v)/\e=0$
in $L^{p}$, $\lim_{\e\to0^+}m_\e=m$ in $L^{p}$,
$\lim_{\e\to0^+}\psi_\e=\varphi$ in $L^{p}$, $\lim_{\e\to0^+}
\e^j\nabla^{1+j} v_\e=0$ in $L^{p}$, $\lim_{\e\to0^+} \e^j\nabla^j
m_\e=0$ in $L^{p}$, $\lim_{\e\to0^+} \e^j\nabla^j\psi_\e=0$ in
$L^{p}$ for every $p\geq 1$ and any $j\in\{1,\ldots,n\}$ and
\begin{multline}
\label{L2009limew03zeta71288888Contggiuuggyyyynew88789999vprop78899shtrihkkklljkhkhgghhhjhhhjkkkhhhjjjkkj}
0\leq\lim\limits_{\e\to 0}\int\limits_{\O}\frac{1}{\e}\times\\
F\bigg( \big\{\e^n\nabla^{n+1}v_{\e},\,\e^n\nabla^n
m_\e,\,\e^n\nabla^n\psi_\e\big\},\,\ldots\,,\big\{\e\nabla^2v_{\e},\,\e\nabla
m_\e,\,\e\nabla\psi_\e\big\},\,\big\{\nabla
v_{\e},\,m_\e,\,\psi_\e\big\},\,
v_{\e},\,f\bigg)\,dx\\-\int\limits_{\O\cap (J_{\nabla v}\cup J_{\bar
m}\cup J_{\varphi})}\Bigg( \inf\bigg\{\bar
E^{(n)}_x\big(\sigma(\cdot),\theta(\cdot),\gamma(\cdot),L\big):\;\;L>0,\\
\;\sigma\in \mathcal{W}^{(1)}_n(x,\vec k_1,\ldots,\vec
k_N),\,\theta\in \mathcal{W}^{(2)}_n(x,\vec k_1,\ldots,\vec
k_N),\,\gamma\in \mathcal{W}^{(3)}_n(x,\vec k_1,\ldots,\vec
k_N)\bigg\} \Bigg)\,d\mathcal{H}^{N-1}(x)<\delta\,,
\end{multline}
where
\begin{multline}
\label{L2009limew03zeta71288888Contggiuuggyyyynew88789999vprop78899shtrihkkkllyhjyukjkkmmmklklklhhhhkkffgghhjjjkkkllkkhhhjjjlkjjjlk}
\bar E^{(n)}_x\big(\sigma(\cdot),\theta(\cdot),\gamma(\cdot),L\big):=\\
\int\limits_{I^+_{\vec k_1,,\ldots,\vec
k_N}}\frac{1}{L}F\Bigg(\Big\{L^n\nabla^{n+1}\sigma(y),L^n\nabla^n\theta(y),
L^n\nabla^n\gamma(y)\Big\},\ldots,\Big\{\nabla\sigma(y),\theta(y),\gamma(y)\Big\},v(x),f^+(x) \Bigg)dy+\\
\int\limits_{I^-_{\vec k_1,\ldots,\vec
k_N}}\frac{1}{L}F\Bigg(\Big\{L^n\nabla^{n+1}\sigma(y),
L^n\nabla^n\theta(y), L^n\nabla^n\gamma(y)\Big\},\ldots,
\Big\{\nabla\sigma(y),\theta(y),\gamma(y)\Big\},v(x),f^-(x)
\Bigg)dy\,,
\end{multline}
\begin{multline}\label{L2009Ddef2hhhjjjj77788hhhkkkkllkjjjjkkkhhhhffggdddkkkgjhikhhhjjnkmklklk}
\mathcal{W}^{(1)}_n\big(x,\vec k_1(x),\vec k_2(x),\ldots \vec
k_N(x)\big):=
\bigg\{u\in C^{n+1}(\R^N,\R^k):\;\;\nabla u(y)=(\nabla
v)^-(x)\;\text{ if }\;y\cdot\vec\nu(x)\leq-1/2,\\
\nabla u(y)=(\nabla v)^+(x)\;\text{ if }\; y\cdot\vec\nu(x)\geq
1/2\;\text{ and }\;\nabla u\big(y+\vec k_j(x)\big)=\nabla
u(y)\;\;\forall j=2,3,\ldots, N\bigg\}\,,
\end{multline}
\begin{multline}\label{L2009Ddef2hhhjjjj77788hhhkkkkllkjjjjkkkhhhhffggdddkkkgjhikhhhjjhhhhklklklk}
\mathcal{W}^{(2)}_n\big(x,\vec k_1(x),\vec k_2(x),\ldots \vec
k_N(x)\big):=\\
\bigg\{\xi\in C^n(\R^N,\R^{d\times N}):\;\;div_y \xi(y)\equiv
0,\;\;\xi(y)=\bar m^-(x)\;\text{ if }\;y\cdot\vec\nu(x)\leq-1/2,\\
\xi(y)=\bar m^+(x)\;\text{ if }\; y\cdot\vec\nu(x)\geq 1/2\;\text{
and }\;\xi\big(y+\vec k_j(x)\big)=\xi(y)\;\;\forall j=2,3,\ldots,
N\bigg\}\,,
\end{multline}
\begin{multline}\label{L2009Ddef2hhhjjjj77788hhhkkkkllkjjjjkkkhhhhffggdddkkkgjhikhhhjjddddkkjkjlkjlk}
\mathcal{W}^{(3)}_n\big(x,\vec k_1(x),\vec k_2(x),\ldots \vec
k_N(x)\big):=
\bigg\{\zeta\in C^n(\R^N,\R^m):\;\;\zeta(y)=\varphi^-(x)\;\text{ if }\;y\cdot\vec\nu(x)\leq-1/2,\\
\zeta(y)=\varphi^+(x)\;\text{ if }\; y\cdot\vec\nu(x)\geq
1/2\;\text{ and }\;\zeta\big(y+\vec k_j(x)\big)=\zeta(y)\;\;\forall
j=2,3,\ldots, N\bigg\}\,,
\end{multline}
\begin{equation}\label{cubidepgghkkkkllllllljjjkkkkkjjjhhhhlkkkffffggggdddhhhjjkljjkljljlkj}
\begin{split}
I^-_{\vec k_1,\vec k_2,\ldots,\vec k_N}:=
\Big\{y\in\R^N:\;-1/2<y\cdot\vec k_1<0,\;\; |y\cdot\vec
k_j|<1/2\quad\forall
j=2,3,\ldots, N\Big\}\,,\\
I^+_{\vec k_1,\vec k_2,\ldots,\vec k_N}:=
\Big\{y\in\R^N:\;0<y\cdot\vec k_1<1/2,\;\; |y\cdot \vec
k_j|<1/2\quad\forall j=2,3,\ldots, N\Big\}\,,
\end{split}
\end{equation}
$\{\vec k_1(x),\vec k_2(x),\ldots \vec k_N(x)\}$ is an orthonormal
base
in $\R^N$ satisfying $\vec
k_1(x)=\vec\nu(x)$
and we assume that the orientations of $J_{\nabla v}$, $J_{\bar m}$,
$J_\varphi$ and $J_f$ coincides $\mathcal{H}^{N-1}$ a.e. and given
by the vector $\vec\nu(x)$. Moreover, $\nabla v_\e$,
$\e^j\nabla^{j+1} v_\e$, $v_\e$, $m_\e$, $\e^j\nabla^j m_\e$,
$\psi_\e$ and $\e^j\nabla^j\psi_\e$
are bounded in $L^\infty$ sequences for every $j\in\{1,\ldots,n\}$,
and there exists a compact $K=K_\delta\subset\subset\O$ such that
$v_\e(x)=v^{(0)}_\e(x)$, $m_\e(x)=m^{(0)}_\e(x)$ and
$\psi_\e(x)=\psi^{(0)}_\e(x)$ for every $0<\e<1$ and every
$x\in\R^N\setminus K$, where
\begin{multline*}
v^{(0)}_\e(x)=\frac{1}{\e^N}\int_{\R^N}\eta\Big(\frac{y-x}{\e}\Big)\,v(y)\,dy\,,\quad
m^{(0)}_\e(x)=\frac{1}{\e^N}\int_{\R^N}\eta\Big(\frac{y-x}{\e}\Big)\,\bar
m(y)\,dy\,,
\\
\quad\psi^{(0)}_\e(x)=\frac{1}{\e^N}\int_{\R^N}\eta\Big(\frac{y-x}{\e}\Big)\,\varphi(y)\,dy\,.
\end{multline*}
\end{theorem}

\appendix
\section{Appendix: Notations and basic results about $BV$-functions}
\label{sec:pre}
\begin{itemize}
\item For given a real topological linear space $X$ we denote by $X^*$ the dual space (the space of continuous linear functionals from $X$ to $\R$).
\item For given $h\in X$ and $x^*\in X^*$ we denote by $\big<h,x^*\big>_{X\times X^*}$ the value in $\R$ of the functional $x^*$ on the vector $h$.
\item For given two normed linear spaces $X$ and $Y$ we denote by $\mathcal{L}(X;Y)$ the linear space of continuous (bounded) linear operators from $X$ to $Y$.
\item For given $A\in\mathcal{L}(X;Y)$  and $h\in X$ we denote by $A\cdot h$ the value in $Y$ of the operator $A$ on the vector $h$.
\item For given two reflexive Banach spaces $X,Y$ and
$S\in\mathcal{L}(X;Y)$ we denote by $S^*\in \mathcal{L}(Y^*;X^*)$
the corresponding adjoint operator, which satisfies
\begin{equation*}\big<x,S^*\cdot y^*\big>_{X\times X^*}:=\big<S\cdot
x,y^*\big>_{Y\times Y^*}\quad\quad\text{for every}\; y^*\in
Y^*\;\text{and}\;x\in X\,.
\end{equation*}
\item Given open set $G\subset\R^N$ we denote by
$\mathcal{D}(G,\R^d)$  the real topological linear space of
compactly supported $\R^d$-valued test functions i.e.
$C^\infty_c(G,\R^d)$ with the usual topology.
\item
We denote $\mathcal{D}'(G,\R^d):=\big\{\mathcal{D}(G,\R^d)\big\}^*$
(the space of $\R^d$ valued distributions in $G$).
\item
Given $h\in\mathcal{D}'(G,\R^d)$ and $\delta\in\mathcal{D}(G,\R^d)$
we denote $<\delta,h>:=\big<\delta,h\big>_{\mathcal{D}(G,\R^d)\times
\mathcal{D}'(G,\R^d)}$ i.e. the value in $\R$ of the distribution
$h$ on the test function $\delta$.
\item
Given a linear operator $\vec A\in\mathcal{L}(\R^d;\R^k)$ and a
distribution $h\in\mathcal{D}'(G,\R^d)$ we denote by $\vec A\cdot h$
the distribution in $\mathcal{D}'(G,\R^k)$ defined by
\begin{equation*}
<\delta,\vec A \cdot h>:=<\vec A^*\cdot
\delta,h>\quad\quad\forall\delta\in\mathcal{D}(G,\R^k).
\end{equation*}
\item
Given $h\in\mathcal{D}'(G,\R^d)$ and $\delta\in\mathcal{D}(G,\R)$ by
$<\delta,h>$ we denote the vector in $\R^d$ which satisfy
$<\delta,h>\cdot \vec e:=<\delta\vec e,h>$ for every $\vec
e\in\R^d$.
\item
For a $p\times q$ matrix $A$ with $ij$-th entry $a_{ij}$ we denote
by $|A|=\bigl(\Sigma_{i=1}^{p}\Sigma_{j=1}^{q}a_{ij}^2\bigr)^{1/2}$
the Frobenius norm of $A$.

\item For two matrices $A,B\in\R^{p\times q}$  with $ij$-th entries
$a_{ij}$ and $b_{ij}$ respectively, we write\\
$A:B\,:=\,\sum\limits_{i=1}^{p}\sum\limits_{j=1}^{q}a_{ij}b_{ij}$.

\item For a $p\times q$ matrix $A$ with
$ij$-th entry $a_{ij}$ and for a $q\times d$ matrix $B$ with $ij$-th
entry $b_{ij}$ we denote by $AB:=A\cdot B$ their product, i.e. the
$p\times d$ matrix, with $ij$-th entry
$\sum\limits_{k=1}^{q}a_{ik}b_{kj}$.
\item We identify
a $\vec u=(u_1,\ldots,u_q)\in\R^q$ with the $q\times 1$ matrix
having $i1$-th entry $u_i$, so that for a $p\times q$ matrix $A$
with $ij$-th entry $a_{ij}$ and for $\vec
v=(v_1,v_2,\ldots,v_q)\in\R^q$ we denote by $A\,\vec v :=A\cdot\vec
v$ the $p$-dimensional vector $\vec u=(u_1,\ldots,u_p)\in\R^p$,
given by $u_i=\sum\limits_{k=1}^{q}a_{ik}v_k$ for every $1\leq i\leq
p$.
\item As usual $A^T$ denotes the transpose of the matrix $A$.
\item For
$\vec u=(u_1,\ldots,u_p)\in\R^p$ and $\vec
v=(v_1,\ldots,v_p)\in\R^p$ we denote by $\vec u\vec v:=\vec
u\cdot\vec v:=\sum\limits_{k=1}^{p}u_k v_k$ the standard scalar
product. We also note that $\vec u\vec v=\vec u^T\vec v=\vec v^T\vec
u$ as products of matrices.
\item For $\vec
u=(u_1,\ldots,u_p)\in\R^p$ and $\vec v=(v_1,\ldots,v_q)\in\R^q$ we
denote by $\vec u\otimes\vec v$ the $p\times q$ matrix with $ij$-th
entry $u_i v_j$ (i.e. $\vec u\otimes\vec v=\vec u\,\vec v^T$ as
product of matrices).
\item For
any $p\times q$ matrix $A$ with $ij$-th entry $a_{ij}$ and $\vec
v=(v_1,v_2,\ldots,v_d)\in\R^d$ we denote by $A\otimes\vec v$ the
$p\times q\times d$ tensor with $ijk$-th entry $a_{ij}v_k$.
\item
Given a vector valued function
$f(x)=\big(f_1(x),\ldots,f_k(x)\big):\O\to\R^k$ ($\O\subset\R^N$) we
denote by $Df$ or by $\nabla_x f$ the $k\times N$ matrix with
$ij$-th entry $\frac{\partial f_i}{\partial x_j}$.
\item
Given a matrix valued function
$F(x):=\{F_{ij}(x)\}:\R^N\to\R^{k\times N}$ ($\O\subset\R^N$) we
denote by $div\,F$ the $\R^k$-valued vector field defined by
$div\,F:=(l_1,\ldots,l_k)$ where
$l_i=\sum\limits_{j=1}^{N}\frac{\partial F_{ij}}{\partial x_j}$.
\item Given a
matrix valued function $F(x)=\big\{f_{ij}(x)\big\}(1\leq i\leq
p,\,1\leq j\leq q):\O\to\R^{p\times q}$ ($\O\subset\R^N$) we denote
by $DF$ or by $\nabla_x F$ the $p\times q\times N$ tensor with
$ijk$-th entry $\frac{\partial f_{ij}}{\partial x_k}$.
\item For every dimension $d$
we denote by $I$ the unit $d\times d$-matrix and by $O$ the null
$d\times d$-matrix.
\item Given a vector valued
measure $\mu=(\mu_1,\ldots,\mu_k)$ (where for any $1\leq j\leq k$,
$\mu_j$ is a finite signed measure) we denote by $\|\mu\|(E)$ its
total variation measure of the set $E$.
\item For any $\mu$-measurable function $f$, we define the product measure
$f\cdot\mu$ by: $f\cdot\mu(E)=\int_E f\,d\mu$, for every
$\mu$-measurable set $E$.
\item Throughout this paper we assume that
$\O\subset\R^N$ is an open set.
\end{itemize}

 In what follows we present some known results on BV-spaces. We rely mainly on the book \cite{amb}
 by Ambrosio, Fusco and Pallara. Other sources are the books
 by Hudjaev and Volpert~\cite{vol}, Giusti~\cite{giusti} and  Evans and Gariepy~\cite{evans}.
 We begin by introducing some notation.
For every $\vec\nu\in S^{N-1}$ (the unit sphere in $\R^N$) and $R>0$
we set
\begin{align}
B_{R}^+(x,\vec\nu)&=\{y\in\R^N\,:\,|y-x|<R,\,
(y-x)\cdot\vec\nu>0\}\,,\label{eq:B+}\\
B_{R}^-(x,\vec\nu)&=\{y\in\R^N:|y-x|<R,\,
(y-x)\cdot\vec\nu<0\}\,,\label{eq:B-}\\
H_+(x,\vec\nu)&=\{y\in\R^N:(y-x)\cdot\vec\nu>0\}\,,\label{HN+}\\
H_-(x,\vec\nu)&=\{y\in\R^N:(y-x)\cdot\vec\nu<0\}\label{HN-}\\
\intertext{and} H^0_{\vec \nu}&=\{ y\in\R^N\,:\, y\cdot\vec
\nu=0\}\label{HN}\,.
\end{align}
 Next we recall the definition of the space of functions with bounded
 variation. In what follows, ${\mathcal L}^N$ denotes the Lebesgue measure in $\R^N$.
\begin{definition}
Let $\Omega$ be a domain in $\R^N$ and let $f\in L^1(\Omega,\R^m)$.
We say that $f\in BV(\Omega,\R^m)$ if
\begin{equation*}
\int_\Omega|Df|:=\sup\bigg\{\int_\Omega\sum\limits_{k=1}^{m}f_k\Div\,\f_k\,d\mathcal{L}^N
:\;\f_k\in C^1_c(\Omega,\R^N)\;\forall
k,\,\sum\limits_{k=1}^{m}|\f_k(x)|^2\leq 1\;\forall
x\in\Omega\bigg\}
\end{equation*}
is finite. In this case we define the BV-norm of $f$ by
$\|f\|_{BV}:=\|f\|_{L^1}+\int_\Omega|D f|$.
\end{definition}
 We recall below some basic notions in Geometric Measure Theory (see
 \cite{amb}).
\begin{definition}\label{defjac889878}
Let $\Omega$ be a domain in $\R^N$. Consider a function
$f\in L^1_{loc}(\Omega,\R^m)$ and a point $x\in\Omega$.\\
i) We say that $x$ is a point of {\em approximate continuity} of $f$
if there exists $z\in\R^m$ such that
$$\lim\limits_{\rho\to 0^+}\frac{\int_{B_\rho(x)}|f(y)-z|\,dy}
{\mathcal{L}^N\big(B_\rho(x)\big)}=0\,.
$$
In this case $z$ is called an {\em approximate limit} of $f$ at $x$
and
 we denote $z$ by $\tilde{f}(x)$. The set of points of approximate continuity of
$f$ is denoted by $G_f$.\\
ii) We say that $x$ is an {\em approximate jump point} of $f$ if
there exist $a,b\in\R^m$ and $\vec\nu\in S^{N-1}$ such that $a\neq
b$ and \begin{equation}\label{aprplmin} \lim\limits_{\rho\to
0^+}\frac{\int_{B_\rho^+(x,\vec\nu)}|f(y)-a|\,dy}
{\mathcal{L}^N\big(B_\rho(x)\big)}=0,\quad \lim\limits_{\rho\to
0^+}\frac{\int_{B_\rho^-(x,\vec\nu)}|f(y)-b|\,dy}
{\mathcal{L}^N\big(B_\rho(x)\big)}=0.
\end{equation}
The triple $(a,b,\vec\nu)$, uniquely determined by \eqref{aprplmin}
up to a permutation of $(a,b)$ and a change of sign of $\vec\nu$, is
denoted by $(f^+(x),f^-(x),\vec\nu_f(x))$. We shall call
$\vec\nu_f(x)$ the {\em approximate jump vector} and we shall
sometimes write simply $\vec\nu(x)$ if the reference to the function
$f$ is clear. The set of approximate jump points is denoted by
$J_f$. A choice of $\vec\nu(x)$ for every $x\in J_f$ (which is
unique up to sign) determines an orientation of $J_f$. At a point of
approximate continuity $x$, we shall use the convention
$f^+(x)=f^-(x)=\tilde f(x)$.
\end{definition}
We recall the following results on BV-functions that we shall use in
the sequel. They are all taken from \cite{amb}. In all of them
$\Omega$ is a domain in $\R^N$ and $f$ belongs to $BV(\Omega,\R^m)$.
\begin{theorem}[Theorems 3.69 and 3.78 from \cite{amb}]\label{petTh}
  $ $ \\i) $\mathcal{H}^{N-1}$-almost every point in
$\Omega\setminus J_f$ is a point of approximate continuity of $f$.\\
ii) The set $J_f$  is a countably $\mathcal{H}^{N-1}$-rectifiable
Borel set, oriented by $\vec\nu(x)$. In other words, $J_f$ is
$\sigma$-finite with respect to $\mathcal{H}^{N-1}$, there exist
countably many $C^1$ hypersurfaces $\{S_k\}^{\infty}_{k=1}$ such
that
$\mathcal{H}^{N-1}\Big(J_f\setminus\bigcup\limits_{k=1}^{\infty}S_k\Big)=0$,
and for $\mathcal{H}^{N-1}$-almost every $x\in J_f\cap S_k$, the
approximate jump vector $\vec\nu(x)$ is  normal to $S_k$ at the
point $x$.\\iii) $\big[(f^+-f^-)\otimes\vec\nu_f\big](x)\in
L^1(J_f,d\mathcal{H}^{N-1})$.
\end{theorem}
\begin{theorem}[Theorems 3.92 and 3.78 from \cite{amb}]\label{vtTh}
The distributional gradient
 $D f$ can be decomposed
as a sum of three Borel regular finite matrix-valued measures on
$\Omega$,
\begin{equation*}
D f=D^a f+D^c f+D^j f
\end{equation*}
with
\begin{equation*}
 D^a f=(\nabla f)\,\mathcal{L}^N ~\text{ and }~ D^j f=(f^+-f^-)\otimes\vec\nu_f
\mathcal{H}^{N-1}\llcorner J_f\,.
\end{equation*}
$D^a$, $D^c$ and $D^j$ are called absolutely continuous part, Cantor
and jump part of $D f$, respectively, and  $\nabla f\in
L^1(\Omega,\R^{m\times N})$ is the approximate differential of $f$.
The three parts are mutually singular to each  other. Moreover we
have the
following properties:\\
i) The support of $D^cf$ is concentrated on a set of
$\mathcal{L}^N$-measure zero, but $(D^c f) (B)=0$ for any Borel set
$B\subset\Omega$ which is $\sigma$-finite with respect to
$\mathcal{H}^{N-1}$;\\ii) $[D^a f]\big(f^{-1}(H)\big)=0$ and $[D^c
f]\big(\tilde f^{-1}(H)\big)=0$ for every $H\subset\R^m$ satisfying
$\mathcal{H}^1(H)=0$.
\end{theorem}
\begin{theorem}[Volpert chain rule, Theorems 3.96 and 3.99 from \cite{amb}]\label{trTh}
Let $\Phi\in C^1(\R^m,\R^q)$ be a Lipschitz function satisfying
$\Phi(0)=0$ if $|\Omega|=\infty$. Then, $v(x)=(\Phi\circ f)(x)$
belongs to $BV(\Omega,\R^q)$ and we have
\begin{equation*}\begin{split}
D^a v = \nabla\Phi(f)\,\nabla f\,\mathcal{L}^N,\; D^c v =
\nabla\Phi(\tilde f)\,D^c f,\; D^j v =
\big[\Phi(f^+)-\Phi(f^-)\big]\otimes\vec\nu_f\,
\mathcal{H}^{N-1}\llcorner J_f\,.
\end{split}
\end{equation*}
\end{theorem}
We also recall that the trace operator $T$ is a continuous map from
$BV(\Omega)$, endowed with the strong topology (or more generally,
the topology induced by strict convergence), to
$L^1(\partial\Omega,{\mathcal H}^{N-1}\llcorner\partial\Omega)$,
provided that $\Omega$ has a bounded Lipschitz boundary (see
\cite[Theorems 3.87 and 3.88]{amb}).


\begin{thebibliography}{66}
\bibitem{ARS} F.~Alouges, T.~Rivière, S.~Serfaty,  {\em N\'{e}el and cross-tie wall energies
for planar micromagnetic configurations},
ESAIM Control Optim. Calc. Var. {\bf 8} (2002), 31--68.

\bibitem{ambrosio} L.~Ambrosio, {\em Metric space valued functions of
    bounded variation},
 Ann. Scuola Norm. Sup. Pisa Cl. Sci. (4)  {\bf 17}  (1990), 439--478.
\bibitem{adm} L.~Ambrosio, C.~De Lellis and C.~ Mantegazza, {\em  Line
energies for gradient vector fields in the plane}, Calc. Var. PDE
{\bf 9 }(1999), 327--355.
\bibitem{amb} L.~Ambrosio, N.~Fusco and D.~Pallara, Functions of
Bounded Variation and Free Discontinuity Problems, Oxford
Mathematical Monographs. Oxford University Press, New York, 2000.
\bibitem{ag1} P.~Aviles and  Y.~Giga, {\em A mathematical problem related to the physical theory of liquid
 crystal configurations}, Proc. Centre Math. Anal. Austral. Nat. Univ. {\bf 12} (1987), 1--16.
\bibitem{ag2} P.~Aviles and  Y.~Giga, {\em On lower semicontinuity of a
defect energy obtained by a singular limit of the
 Ginzburg-Landau type energy for gradient
 fields}, Proc. Roy. Soc. Edinburgh Sect. A  {\bf 129} (1999), 1--17.

\bibitem{CdL} Sergio Conti and Camillo de Lellis, {\em Sharp upper bounds for a variational problem with singular
perturbation}, Math. Ann. {\bf 338} (2007), no. 1, 119--146.

\bibitem{contiS} Sergio Conti and Ben Schweizer {\em A sharp-interface limit for a
two-well problem in geometrically linear elasticity}. Arch. Ration.
Mech. Anal. 179 (2006), no. 3, 413--452.
\bibitem{contiS1} Sergio Conti and Ben Schweizer {\em Rigidity and Gamma convergence for solid-solid phase
transitions with $SO(2)$-invariance}, Comm. Pure Appl. Math. 59
(2006), no. 6, 830--868.
\bibitem{contiFL} S. Conti, I. Fonseca, G. Leoni {\em A $\Gamma$-convergence
result for the two-gradient theory of phase transitions}, Comm. Pure
Appl. Math. {\bf 55} (2002), pp. 857-936.
\bibitem{conti} S.~Conti and C.~De Lellis, {\em Sharp upper bounds for a variational problem
with singular perturbation}, Math. Ann. 338 (2007), no. 1, 119--146.
\bibitem{dl} C.~De Lellis, {\em An example in the gradient theory of phase
transitions} ESAIM Control Optim. Calc. Var. {\bf 7} (2002),
285--289 (electronic).

\bibitem{otto} A.~DeSimone,  S.~M\"uller, R.V.~Kohn and F.~Otto, {\em A compactness result
in the gradient theory of phase transitions}, Proc. Roy. Soc.
Edinburgh Sect. A  {\bf 131} (2001), 833--844.
\bibitem{DKMO} A.~DeSimone,  S.~M\"uller, R.V.~Kohn and F.~Otto,
{\em Recent analytical developments in micromagnetics}, In Giorgio
Bertotti and Isaak Mayergoyz, editors, The Science of Hysteresis,
volume 2, chapter 4, pages 269-381. Elsevier Academic Press, 2005.

\bibitem{FM}  I.~Fonseca and C.~Mantegazza, {\em Second order singular perturbation models
for phase transitions}, SIAM J. Math. Anal. 31 (2000), no. 5,
1121--1143 (electronic).

\bibitem{FonP}  I.~Fonseca, C.~Popovici, {\em Coupled singular perturbations for phase transitions},
Assymptotic Analisis {\bf 44} (2005), 299-325.

\bibitem{evansbook} L.C.~Evans, Partial Differential Equations, Graduate
 Studies in Mathematics, Vol.~{\bf 19}, American Mathematical Society, 1998.
\bibitem{evans}  L.C.~Evans and   R.F.~Gariepy,  Measure Theory and Fine
Properties of Functions, Studies in Advanced Mathematics, CRC Press,
Boca Raton, FL, 1992.
\bibitem{gt} D.~Gilbarg and N.~Trudinger, Elliptic Partial Differential
 Equations of Elliptic Type, 2nd ed., Springer-Verlag,
  Berlin-Heidelberg, 1983.
\bibitem{giusti} E.~Giusti, Minimal Surfaces and Functions of Bounded
  Variation, Monographs in Mathematics, {\bf 80}, Birkh{\"a}user Verlag,
  Basel, 1984.



\bibitem{HaS} A.~Hubert and R.~Sch\"{a}fer, Magnetic domains, Springer, 1998.

\bibitem{jin} W.~Jin and R.V.~Kohn, {\em Singular perturbation and
the energy of folds}, J. Nonlinear Sci. {\bf 10} (2000), 355--390.


\bibitem{CDFO} C.~De Lellis, F.~Otto {\em Structure of entropy solutions to the eikonal equation}
J. Eur. Math. Soc. {\bf 5}, (2003), 107--145.


\bibitem{modica} L.~Modica,
{\em The gradient theory of  phase transitions and the minimal
 interface criterion}, Arch.  Rational Mech. Anal. {\bf 98} (1987),  123--142.
\bibitem{mm1} L.~Modica and S.~Mortola, {\em Un esempio di $\Gamma
    \sp{-}$-convergenza},
Boll. Un. Mat. Ital. B {\bf 14 } (1977),  285--299.
\bibitem{mm2} L.~Modica and S.~Mortola, {\em Il limite nella $\Gamma
    $-convergenza di una famiglia di funzionali ellittici},
  Boll. Un. Mat. Ital.  A  {\bf 14}  (1977),  526--529.

\bibitem{polgen} A.~Poliakovsky, {\em A general technique to prove upper bounds for
singular perturbation problems}, Journal d'Analyse Mathematique,
{\bf 104} (2008), no. 1, 247-290.

\bibitem{P3} A.~Poliakovsky, {\em Sharp upper bounds for a singular perturbation
problem related to micromagnetics}, Annali della Scuola Normale
Superiore di Pisa, Classe di Scienze. {\bf 6} (2007), no. 4,
673--701.

\bibitem{polmag} A.~Poliakovsky, {\em Upper bounds for a class of energies containing a non-local
term}, , ESAIM: Control, Optimization and Calculus of Variations,
{\bf 16} (2010),  856--886.

\bibitem{polcras} A.~Poliakovsky, {\em A method for establishing upper bounds for  singular
perturbation problems}, C. R. Math. Acad. Sci. Paris 341 (2005), no.
2, 97--102.

\bibitem{pol} A.~Poliakovsky, {\em Upper bounds for  singular perturbation problems involving gradient
fields}, J.~Eur.~Math.~Soc., {\bf 9} (2007), 1--43.
\bibitem{pollift} A.~Poliakovsky, {\em On a singular perturbation problem related to optimal lifting
  in BV-space}, Calculus of Variations and PDE, {\bf
28} (2007), 411--426.

\bibitem{P4} A.~Poliakovsky, {\em On a
variational approach to the Method of Vanishing Viscosity for
Conservation Laws}, Advances in Mathematical Sciences and
Applications, {\bf 18} (2008), no. 2., 429--451.



\bibitem{PII} A.~Poliakovsky, {\em On the $\Gamma$-limit of singular perturbation problems with
optimal profiles which are not one-dimensional. Part II: The lower
bound}, preprint, http://arxiv.org/abs/1112.2968

\bibitem{PIII} A.~Poliakovsky, {\em On the $\Gamma$-limit of singular perturbation problems with
optimal profiles which are not one-dimensional. Part III: The
energies with non local terms}, preprint,
http://arxiv.org/abs/1112.2971

\bibitem{RS1} T. Rivi\`{e}re and S. Serfaty, {\em Limiting domain wall energy for a problem
related to micromagnetics}, Comm. Pure Appl. Math., 54 No 3 (2001),
294-338.
\bibitem{RS2} T. Rivi\`{e}re and S. Serfaty, {\em Compactness, kinetic formulation and entropies for a
problem related to mocromagnetics}, Comm. in Partial Differential
Equations 28 (2003), no. 1-2, 249-269.

\bibitem{sternberg} P.~Sternberg,
 {\em The effect of a singular perturbation on nonconvex
  variational problems}, Arch. Rational Mech. Anal. {\bf 101} (1988), 209--260.

\bibitem{vol}  A.I.~Volpert and S.I.~Hudjaev, Analysis in Classes of Discontinuous Functions and
Equations of Mathematical Physics,  Martinus Nijhoff Publishers,
Dordrecht, 1985.


\end{thebibliography}
\end{document}